\documentclass[12pt]{article}
\usepackage{amsfonts}
\usepackage{comment}
\usepackage{tabularx}
\usepackage{enumerate}
\usepackage{bbm}
\usepackage[mathlines,displaymath,pagewise]{lineno}

\usepackage{tikz}
\usepackage{latexsym}
\usepackage[pdftitle={Dynamical percolation},
bookmarks=true,bookmarksopen=true,bookmarksopenlevel=3]{hyperref}

\usepackage{amsmath, amsthm, amssymb, color, graphicx, mathrsfs,wasysym}
\hypersetup{
    linktoc=page,
    linkcolor=red,          
    citecolor=blue,        
    filecolor=blue,      
    urlcolor=cyan,
   colorlinks=true           
}
\usepackage[margin=1in]{geometry}
\DeclareGraphicsExtensions{.jpg,.pdf,.eps}
\graphicspath{{./Figures/}}

\newtheorem{thm}{Theorem}[section]
\newtheorem{theorem}[thm]{Theorem}
\newtheorem{proposition}[thm]{Proposition}
\newtheorem{prop}[thm]{Proposition}
\newtheorem{lemma}[thm]{Lemma}
 
\newtheorem{corollary}[thm]{Corollary} 

\newtheorem{definition}[thm]{Definition}

\newtheorem{claim}[thm]{Claim}

\theoremstyle{definition}
\newtheorem{remark}[thm]{Remark}

\numberwithin{equation}{section}

\renewcommand{\ni}{\noindent}

\newcommand{\old}[1]{{}} 
 
 \newcommand{\mesh}{\eta}

\newcommand{\M}{\mathcal{M}}
\newcommand{\C}{\mathbb{C}}

\newcommand{\Os}{\mathcal O}
\newcommand{\1}{\mathbf{1}}
\newcommand{\D}{\mathbb{D}}
\newcommand{\E}{\mathbb{E}}
\newcommand{\N}{\mathbb{N}}
\newcommand{\Q}{\mathbb{Q}}
\newcommand{\Z}{\mathbb{Z}}
\newcommand{\R}{\mathbb{R}}
\renewcommand{\P}{\mathbb{P}}

\newcommand{\ol}{\overline}
\newcommand{\wt}{\widetilde}

\newcommand{\eps}{\epsilon}
\def\P{\mathbb{P}}
\def\E{\mathbb{E}}
\DeclareMathOperator{\Cov}{Cov}

\DeclareMathOperator{\Var}{Var}

\def\cQ{\mathcal{Q}}

\def\cM{\mathcal{M}}

\def\cI{\mathcal{I}}

\def\cE{\mathcal{E}}

\def\cA{\mathcal{A}}

\newcommand{\notion}[1]{{\bf  #1}}

\newcommand{\ga}{\gamma}

\newcommand{\II}{\mathcal I}

\newcommand{\QQ}{\mathcal{Q}}

\renewcommand{\ni}{\noindent}

\def\@rst #1 #2other{#1}
\newcommand\MR[1]{\relax\ifhmode\unskip\spacefactor3000 \space\fi
		\MRhref{\expandafter\@rst #1 other}{#1}}\newcommand{\MRhref}[2]{\href{http://www.ams.org/mathscinet-getitem?mr=#1}{MR#2}}

\def\MR#1{\href{http://www.ams.org/mathscinet-getitem?mr=#1}{MR#1}}

\newcommand{\aryb}{\begin{eqnarray*}}
\newcommand{\arye}{\end{eqnarray*}}
\def\alb#1\ale{\begin{align*}#1\end{align*}}
\newcommand{\eqb}{\begin{equation}}
\newcommand{\eqe}{\end{equation}}
\newcommand{\eqbn}{\begin{equation*}}
\newcommand{\eqen}{\end{equation*}}

\newcommand{\op}{\operatorname}

\newcommand{\eqD}{\overset{d}{=}}

\newcommand{\ep}{\epsilon}
\newcommand{\rta}{\rightarrow}

\newcommand{\wh}{\widehat} 
\newcommand{\mcl}{\mathcal}
\newcommand{\scr}{\mathscr}

\newcommand{\EE}{\mathcal{ E}}
\renewcommand{\II}{\mathcal{I}}
\renewcommand{\S}{\mathscr{S}}

\def\bi{\begin{itemize}}
	\def\ei{\end{itemize}}

\def\diam{\mathrm{diam}}

\def\md{\mid}
\def\Bb#1#2{{\def\md{\bigm| }#1\bigl[#2\bigr]}}
\def\Pb{\Bb\P}
\def\Eb{\Bb\E}
\def \p {{\partial}}

\def\ev#1{{\mathcal{#1}}}

\def\Spec{\mathscr{S}}

\def\Tg{\mathbb{T}} 
\def\Quad{Q} 

\def\qs{\mcl Q}

\newcommand{\HH}{\mathscr{H}}

\def\Sk{\mathsf{Sk}}
\def\Rs{\mathcal{Z}}
\def\llwb{\lambda_{B,W}}
\def\coa{\vartheta}

\def\nc{\mathsf{nc}}
\def\MST{\mathsf{MST}}

\begin{document}
\title{Liouville dynamical percolation}
\author{
	\begin{tabular}{c} Christophe Garban\\ [-5pt] \small Univ Lyon 1 \end{tabular}
	\begin{tabular}{c} Nina Holden\\ [-5pt] \small ETH Z\"urich \end{tabular}
	\begin{tabular}{c} Avelio Sep\'ulveda\\ [-5pt] \small Univ Lyon 1 \end{tabular}
	\begin{tabular}{c} Xin Sun\\ [-5pt] \small Columbia \end{tabular}
}
\date{}

\maketitle

\begin{abstract}
We construct and analyze a  continuum dynamical percolation process which evolves in a random environment given by a $\gamma$-Liouville measure.
The homogeneous counterpart of this process describes the scaling limit  of discrete dynamical percolation on the rescaled triangular lattice. Our focus here is to study the same limiting dynamics, but where the speed of microscopic updates is highly inhomogeneous in space and is driven by the $\gamma$-Liouville measure associated with a two-dimensional log-correlated field $h$. Roughly speaking, this continuum percolation process evolves very rapidly where the field $h$ is high and barely moves where the field $h$ is low. Our main results can be summarized as follows.
\begin{itemize}
\item First, we build this inhomogeneous dynamical percolation which we call $\gamma$-Liouville dynamical percolation (LDP) by taking the scaling limit of the associated process on the triangular lattice. We work with three different regimes each requiring different tools: $\gamma\in [0,2-\sqrt{5/2})$, $\gamma\in [2-\sqrt{5/2}, \sqrt{3/2})$, and $\gamma\in(\sqrt{3/2},2)$. 
\item When $\gamma<\sqrt{3/2}$, we prove that $\gamma$-LDP is mixing in the Schramm-Smirnov space as $t\to \infty$, quenched in the log-correlated field $h$. On the contrary, when $\gamma>\sqrt{3/2}$ the process is frozen in time. 
The ergodicity result is a crucial piece of the Cardy embedding project of the second and fourth coauthors, where LDP for $\gamma=\sqrt{1/6}$ is used to study the scaling limit of a variant of dynamical percolation on uniform triangulations.
\item When $\gamma<\sqrt{3/4}$, we obtain quantitative bounds on the mixing of quad crossing events.
\end{itemize}

\end{abstract}


\section{Introduction}

Given an arbitrary graph, (site) \textbf{percolation} corresponds to an independent black/white coloring of the vertices. For infinite graphs and if $p\in[0,1]$ denotes the probability that a vertex is black, there is a critical value $p_c$ for $p$, which is defined as the infimum of $p$ for which there is an infinite cluster of black vertices a.s. Critical percolation on a wide range of planar graphs and lattices is believed to have a conformally invariant scaling limit. This was proved by Smirnov \cite{smirnov-cardy} for the triangular lattice. 

\textbf{Dynamical percolation} on a graph is a percolation valued process indexed by the non-negative real numbers $\R_+$, such that each vertex has an independent Poisson clock and the color of the vertex is resampled every time its clock rings.
The first important properties of this model were proved in \cite{haggstrom-peres-steif, schramm-steif}.  
The subsequent papers \cite{gps-fourier,gps-pivotal,gps-near-crit}
 studied dynamical percolation on the triangular lattice in the homogeneous case where all the Poisson clocks have the same rate. They proved that this process has a c\`adl\`ag scaling limit $(\omega^0(t))_{t\in\R_+}$ when the rate is chosen appropriately. The limiting process can be defined directly in the continuum as a stationary process, such that at each fixed time $t\geq 0$, $\omega^0(t)$ has the law of  the percolation scaling limit.

Our main focus in this paper is to study dynamical percolation evolving in a random environment corresponding to the so-called 
{\bf Liouville quantum gravity} (LQG). The latter is a theory of random fractal surfaces (see for example \cite{shef-kpz,dkrv-lqg-sphere, wedges}). Let $\gamma\in(0,2)$ and let $h$ be an instance of a Gaussian free field (GFF) or another log-correlated field (see Section \ref{ss.GMC}) in a planar domain $D$. Heuristically, a $\gamma$-LQG surface may be defined as the Riemannian manifold with metric tensor $e^{\gamma h}(dx^2+dy^2)$.\footnote{In fact, the correct metric tensor should be ``$e^{\frac{2\gamma}{d_H(\gamma)}  h}(dx^2+dy^2)$''. See the introduction of \cite{ding-goswami-watabiki}.} This definition does not make rigorous sense since $h$ is a distribution and not a function, but via regularization of $h$ one can make rigorous sense of the area measure $e^{\gamma h}\,d^2z$ and certain other measures of the form $e^{\gamma h}\,d\sigma$ for $\gamma\in(0,2)$ and a base measure $\sigma$. Recently it was also proven that, via regularization of $h$, a $\gamma$-LQG surface is associated with a natural metric (i.e., distance function) \cite{gm-metric}. As explained in Section \ref{ss.motiv}, LQG is intimately related to the scaling limits of random planar maps and this connection with planar maps constitutes the main motivation underlying this work.

The process we will focus on is called  {\bf (continuum) Liouville dynamical percolation} (cLDP). Informally, it can be described as a continuum dynamical percolation where the clocks are driven by an independent LQG measure. 

We construct cLDP as the scaling limit of {\bf (discrete) Liouville dynamical percolation} (dLDP) on the triangular lattice.\footnote{The notion ``Liouville dynamical percolation'' (LDP) may refer to either dLDP or cLDP, and the meaning will be clear from the context.} The rate of the Poisson clocks will now be inhomogeneous and determined by a background log-correlated field $h$. 
More explicitly, let $\Tg_\eta$ denote the regular triangular lattice rescaled such that the distance between adjacent vertices is $\eta$. Let $\alpha_4^\eta(\eta,1)$ be the probability of having a so-called 4-arm event from a site to distance 1 (see Section \ref{sec:gps-dp}). For $x$ a vertex on $\Tg_\eta$, let $B^{\op h}_\eta(x)$ denote the hexagon corresponding to $x$ in the dual graph of $\Tg_\eta$. The following defines dynamical percolation driven by any given fixed measure $\sigma$ (not necessarily a Liouville measure yet).
\begin{definition}[Dynamical percolation driven by a general measure $\sigma$]\label{d.dp}
Fix $\eta>0$. Let $\sigma$ be a measure on $\R^2$.
A dynamical percolation on $\Tg_\mesh$
driven by $\sigma$ is a dynamics on percolation configurations denoted by $\omega^\sigma_\eta(\cdot)$ that is built via the following procedures:
	\begin{itemize}
		\item At time $t=0$, $\omega^\sigma_\eta(0)$ is a percolation configuration on $\Tg_\mesh$ in which the sites are colored independently white (closed) or black (open) with probability $1/2$.
		\item Each site $x\in \Tg_\eta$ of the lattice has an associated Poisson clock with rate $\sigma(B^{\op h}_\eta(x))\alpha^\eta_4(\eta,1)^{-1}$, which rings independently of all other information.
		\item Each time a clock rings, we re-color $x$ white (closed) or black (open) with probability $1/2$ independently of all other randomness. 
	\end{itemize} 
\end{definition}
 Let us note that for this dynamic the law of $\omega^\sigma_\eta(0)$ is the invariant measure. As explained above, in this paper, we study the limit of this space-inhomogeneous dynamical percolation for a specific family of fractal measures $\sigma$, namely the LQG area measures.
 \begin{definition}[discrete Liouville dynamical percolation (dLDP)]
 Let $\mu_{\gamma h}$ be the $\gamma$-LQG measure associated with a log-correlated field $h$. We define $\omega^\gamma_\eta(\cdot)=\omega^{\mu_{\gamma h}}_\eta(\cdot)$ to be the dynamical percolation driven by the measure $\mu_{\gamma h}$, such that $h$ and $\omega^\ga_\eta(0)$ are independent. 
 \end{definition}

\subsection{Main results}
Let $D\subset\C$ be a bounded simply connected domain with smooth boundary. Let $0\leq \gamma <2$ and let $h$ be a centred Gaussian log-correlated field on $D$.\footnote{See Section \ref{ss.GMC} for the precise class of fields we consider.} We denote by $\mu_{\gamma h}$ the $\gamma$-LQG area measure associated with $h$. 

The following theorem gives convergence of dLDP. 
We refer to Section \ref{s.quad} for a discussion of the various topologies which can be used to represent the scaling limit  $\omega$ of critical percolation. In this work we will mainly rely on the Schramm-Smirnov space $\scr H$ introduced in \cite{ss-planar-perc}, see also the other topologies from \cite{cn-scalinglimit,shef-cle}.\footnote{We remark that the loop ensemble space  considered in e.g.\ \cite{cn-scalinglimit} and the quad crossing space considered in \cite{ss-planar-perc,gps-near-crit} and this paper are equivalent in the sense that the  associated $\sigma$-algebras are the same. See \cite{cn-scalinglimit} and \cite[Section 2.3]{gps-pivotal} for a proof that the loops determine the quad crossing information, and see the upcoming work \cite{hs-cardy} for the converse result. Therefore $(\omega^\gamma(t))_{t\geq 0}$ can also be viewed as a process with values in the loop ensemble space.} 
The space of c\`adl\`ag functions with values in $\scr H$ will be equipped with two different topologies: either  Skorokhod topology or the $L^1$ topology, where the latter topology is weaker and is generated by a metric where we integrate the distance between two processes over the considered time interval. 
\begin{thm} $\ $
	\begin{itemize}
		\item[(i)] If $\gamma\in[0,2-\sqrt{5/2})$, then $(\omega^\gamma_\eta(t))_{t\geq 0}$ converges in law to a c\`adl\`ag process $(\omega^\ga_\infty(t))_{t\geq 0}$. The convergence holds for the finite-dimensional law and in the Skorokhod topology for the Schramm-Smirnov space $\scr H$.
		\item[(ii)] If $\gamma \in[2-\sqrt{5/2},\sqrt{3/2})$, then $(\omega^\gamma_\eta(t))_{t\geq 0}$ converges in law to a c\`adl\`ag  process $(\omega^\gamma_\infty(t))_{t\geq 0}$. The convergence holds for the finite-dimensional law and for the $L^1$-topology  in the Schramm-Smirnov space $\scr H$.
		\item[(iii)] If $\ga\in(\sqrt{3/2},2)$, then the conclusion of (ii) still holds. Furthermore, the limiting process is constant (i.e., $\omega^\ga_\infty(t)=\omega^\ga_\infty(0)$ for all $t\geq 0$). 
	\end{itemize}
	In all cases (i)-(iii), conditionally on $h$, $t \mapsto \omega^\gamma_\infty(t)$ is a Markov process on the Schramm-Smirnov space $\scr H$.
	\label{thm1}
	\end{thm}

	We refer to Section \ref{ss.sketch} for an intuitive explanation of the transition points at $\gamma=2-\sqrt{5/2}$ and $\gamma=\sqrt{3/2}$.

We call the limiting process $(\omega^\gamma_\infty(t))_{t\geq 0}$ {\bf continuum Liouville dynamical percolation} (cLDP). A fundamental property of cLDP for $\gamma\in[0,\sqrt
{3/2})$ is its mixing property, which we state next and which will be instrumental in the forthcoming work \cite{hs-cardy}. See Section \ref{s.quad} for the definition of a rectangular quad.

\begin{thm}\label{p.quantitative decorrelation}\textcolor{white}{}
Let $Q$ be a rectangular quad, and for any $t\geq 0$ let $A(t)$ be the event that $\omega_{\infty}^\gamma(t)$ crosses $Q$. For measurable sets $B,C\subset\HH$ and $t\in\R_+$ define $B(t)=\1_{\omega^\ga_\infty(t)\in B}$ and $C(t)=\1_{\omega^\ga_\infty(t)\in C}$. 
\begin{itemize}
	\item[(i)] If $\gamma\in[0,\sqrt{3/2})$ then $\omega^\gamma_\infty(\cdot)$ is mixing. More precisely, for $B(0)$ and $C(t)$ as above, 
	\[
	\lim_{t\to\infty} \Cov(\1_{B(0)},\1_{C(t)})=0.
	\]
	\item[(ii)] Let $\gamma\in[0,\sqrt{3/4})$ and $\theta(d,\gamma):= \frac {d - \gamma^2} {d + \gamma^2}$. 
	Then for $A(t)$ as above and any $\xi< \frac{2 \theta} 5 $ for $ \theta=\theta(3/4,\gamma)$, we have that 
	\[\lim_{t\to\infty}\Cov(\1_{A(0)},\1_{A(t)})t^{\xi}= 0.\]
	More generally, for $B(0)$ and $A(t)$ as above, 
	\begin{equation}\label{e.decorrelation 2 p}
	\lim_{t\to\infty}\Cov(\1_{B(0)},\1_{A(t)})t^{\xi/2}= 0.
	\end{equation}
	\item[(iii)] The results above also hold in the \notion{quenched} sense, i.e., a.s.,
\end{itemize}	
\[
\lim_{t\to\infty} \Cov(\1_{B(0)},\1_{C(t)}\mid h)=\lim_{t\to\infty}\Cov(\1_{A(0)},\1_{A(t)} \mid h)t^{\xi}=	\lim_{t\to\infty}\Cov(\1_{B(0)},\1_{A(t)}\mid h)t^{\xi/2}= 0.
\]
\end{thm}
\begin{remark}
	Since Theorem~\ref{p.quantitative decorrelation} holds for $\gamma=0$, we prove mixing for the Euclidean dynamical percolation studied by Garban, Pete, and Schramm \cite{gps-near-crit}. In particular, we answer the question asked in \cite[Remark 12.3]{gps-near-crit}. It was previously known \cite[Section 12]{gps-near-crit} that we have polynomial mixing for the event of crossing a single rectangular quad for $\gamma=0$. That is, if $A(t)$ denotes that event that a fixed rectangular quad $Q$ is crossed at time $t$, it was know that $\Cov(\1_{A(0)},\1_{A(t)})\leq C_Qt^{-2/3}$ for a constant $C_Q>0$ depending only on $Q$ (see in \cite[Theorem 12.1]{gps-near-crit}).
\end{remark}

In \cite{gps-near-crit,gps-mst} convergence results for near-critical percolation and the minimal spanning tree are established building on the method for dynamical percolation. In the regime $\gamma \in[0,2-\sqrt{5/2})$ (where microscopic stability is satisfied, see Proposition \ref{pr.stab}), we obtain analogous results in our inhomogeneous case. More precisely, Corollaries \ref{c.near-critic} and \ref{c.mst} below are immediate corollaries of the proof of Theorem \ref{thm1} in this regime $\gamma \in[0,2-\sqrt{5/2})$, proceeding similarly as in \cite{gps-near-crit,gps-mst}.

Let us define the \textbf{$\gamma$-near-critical coupling} $(\omega_\eta^{\gamma, \nc}(\lambda))_{\lambda\in \R}$ to be the following process: 
\bi
\item[(i)] Sample $\omega_\eta^{\gamma,\nc}(\lambda=0)$ according to $\P_\eta$, the  law of critical percolation on $\Tg_\eta$. 
\item[(ii)] As $\lambda$ increases, closed (white) hexagons switch to open (black) at exponential rate $r(\eta)=\mu_{\gamma h}(B^{\op h}_\eta(x))\alpha^\eta_4(\eta,1)^{-1}$.
\item[(iii)] As $\lambda$ decreases, open (black) hexagons switch to closed (white) at same rate $r(\eta)$.
\ei
As such, for any $\lambda\in \R$, the near-critical percolation $\omega_\eta^{\gamma,\nc}(\lambda)$ corresponds exactly to a percolation configuration on $\Tg_\eta$ with parameter
\[
\begin{cases}
p=p_c + 1-e^{-\lambda\, r(\eta) } \quad \text{if } \lambda\geq 0, \\
p=p_c - (1-e^{-|\lambda|\, r(\eta) }) \quad \text{if } \lambda< 0\,.  
\end{cases}
\]

\begin{corollary}\label{c.near-critic}
For any $0\leq \gamma < 2 -\sqrt{5/2}$, the c\`adl\`ag process $(\omega^{\gamma,\nc}_\eta(\lambda))_{\lambda \in \R }$ converges for the Skorokhod topology on $\scr H$ to a limiting Markov process  $(\omega^{\gamma,\nc}_\infty(\lambda))_{\lambda \in \R }$. 
\end{corollary}
Similarly, following \cite{gps-mst}, one may readily define an inhomogeneous model of minimal spanning tree on the triangular lattice $\Tg_\eta$ induced by an LQG measure. The microscopic definition of the model for which one can prove a scaling limit in \cite{gps-mst} is in fact a bit subtle, but it generalizes immediately to our setting and we refer to \cite{gps-mst} for details. However, one technical restriction in \cite{gps-mst} is that the scaling limit results are stated on $\C$ and on the tori $\mathbb{L}_M^2= \R^2/ (M \Z^2)$ but not on general planar domains $D$. We will thus consider a log-correlated field $h$ on $\mathbb{L}_M^2$ with mean zero. Let us call $\MST_\eta^{\gamma}$ the minimal spanning tree on the torus $\Tg_\eta \cap \mathbb{L}_M^2$. 

\begin{corollary}\label{c.mst}
For any $\gamma\in[0, 2 - \sqrt{5/2})$, as $\eta\to 0$, the spanning tree $\MST^\gamma_\eta$ on $\Tg_\eta\cap  \mathbb{L}_M^2$ converges in distribution
(under the setup introduced in \cite{abnw99}) to a limiting tree $\MST^\gamma_\infty$. 
\end{corollary}

\begin{remark}\label{}
Note that for these two models, we are only able to treat the regime $0\leq \gamma < 2 -\sqrt{5/2}$. For dynamical percolation, the fact that the process is stationary together with the Fourier technology help tremendously. However, it seems natural to guess that the other two models will keep behaving similarly thanks to a more delicate stability analysis for $\ga\in[2 -\sqrt{5/2}, \wh \gamma_c)$. The question whether $\wh \gamma_c$ is equal to $\sqrt{3/2}$ or not does not seem obvious to us. Indeed for the out-of-equilibrium case corresponding to near-critical percolation, the microscopic stability may cease to exist before reaching $\sqrt{3/2}$ while in the equilibrium case, on and off switches may still compensate each other.
\label{rmk3}
\end{remark}

Finally, in Section \ref{sec:spectralmeasure}, we prove a scaling limit result for the so-called {\bf spectral measures}. The spectral measures are certain random measures in $\D$, each associated with a percolation crossing event (see Section \ref{sec:fourier}). Spectral measures play an important role in several of our proofs and were a crucial tool in \cite{gps-fourier}. In particular, we use the spectral measures as a tool to prove mixing in the subcritical regime and to prove that the supercritical process is trivial.  
In Proposition \ref{pr.convS}, we prove that spectral measures associated with a large class of crossing events converge in law as $\eta\rta 0$. This answers half of the third open problem stated in \cite{gps-fourier}. 

Other results of independent interest not mentioned above are found in the appendices. In particular, we want to highlight Proposition \ref{p.convergence_of_LQG}, where we prove under mild assumptions that if $(\sigma^n)_{n\in\N}$ is a sequence of measures converging in probability to a limiting measure $\sigma$, then the LQG measure with base measure $\sigma^n$ converges to the LQG measure with base measure $\sigma$. Furthermore, in Appendix \ref{app:spectralsample} we prove upper and lower bounds for the total mass of the spectral measure associated with the crossing of multiple quads. This extends the main result of \cite{gps-fourier}, where the case of a single rectangular quad was considered.

\subsection{Motivation from random planar maps}\label{ss.motiv}
The current work is an important input to a program of the second and fourth authors, which proves the convergence of uniform triangulations to $\sqrt{8/3}$-LQG under the so-called Cardy embedding.

Let us start by shortly discussing how LQG surfaces arise as the scaling limit of discrete surfaces known as random planar maps (RPM). A planar map is a graph drawn on the sphere (without edge-crossings) viewed modulo continuous deformations. Le Gall \cite{legall-uniqueness}, Miermont \cite{miermont-brownian-map}, and others proved that certain uniformly sampled RPM equipped with the graph distance converge in law for the Gromov-Hausdorff distance to a limiting metric space known as the \notion{Brownian map}. Miller and Sheffield \cite{lqg-tbm1,lqg-tbm2,lqg-tbm3} proved that the Brownian map is equivalent to $\sqrt{8/3}$-LQG in the sense that an instance of the Brownian map can be coupled together with an instance of $\sqrt{8/3}$-LQG such that the two surfaces determine each other in a natural way. 

An alternative notion of convergence for RPM to LQG is provided by the so-called peanosphere topology. Convergence of RPM to LQG in this topology has been established for RPM in several universality classes. The idea of this topology is to decorate the RPM with a statistical physics model (see e.g.\ \cite{shef-burger,bhs-site,kmsw-bipolar,ghs-bipolar,schnyder}), and show that the decorated map is encoded by a 2d walk which converges in the scaling limit to a correlated Brownian motion. On the other hand, by \cite{wedges}, the Brownian motion encodes an instance of SLE-decorated LQG in a same manner as in the discrete.

Our  main motivations from LQG/planar maps are the following ones:

\begin{enumerate}
\item First, as mentioned above, the Cardy embedding is a discrete conformal embedding which is based on percolation crossing probabilities on planar maps. In \cite{ghs-metric-peano} (based on \cite{wedges,bhs-site,gwynne-miller-perc,aasw-type2}
and other works) it is proved that a uniform percolated triangulation converges as a loop-decorated metric measure space to a $\sqrt{8/3}$-LQG surface decorated by an independent CLE$_6$. By applying the mixing result of our Theorem \ref{p.quantitative decorrelation}, it is shown in \cite{hs-cardy} that the convergence to CLE$_6$ is quenched, i.e., the limiting CLE$_6$ is independent of the randomness of the planar maps. This allows us to conclude that the Cardy embedded random planar maps converge to $\sqrt{8/3}$-LQG.

\item Second, the study of the conjectural scaling limit of dynamical percolation on random planar maps. In \cite{hs-cardy} it is proved that dynamical percolation (with a certain cut-off) on a uniformly chosen triangulation converges to the process built in this paper, namely Liouville dynamical percolation with parameter $\gamma=1/\sqrt{6}$.\footnote{Note that uniform planar maps do \emph{not} correspond in this setting to $\gamma=\sqrt{8/3}$, which is the more commonly considered $\gamma$ value for this universality class of planar maps. The reason for such a smaller $\gamma$ value here is that the dynamical percolation is driven by an LQG measure on a lower-dimensional fractal rather than an open subset of the complex plane. Namely, the measure is supported on the percolation pivotal points, which have dimension $d=3/4$. See Section \ref{sec:central-charge} for further explanation.} 
For general values of $\gamma\in(0,\sqrt{3/2})$, Liouville dynamical percolation should represent the scaling limit of dynamical percolation on random planar maps in other universality classes.

The relationship between $\gamma$ and its corresponding central charge $c=c(\gamma)$  is discussed in details in  Section \ref{sec:central-charge}. Let us point out that in this correspondence we have $c<1$ (resp.\ $c\in(1,16)$) if and only if $\gamma(c)\in(0,2-\sqrt{5/2})$ (resp.\ $\gamma(c)\in(2-\sqrt{5/2},\sqrt{3/2})$). In particular, note that our paper studies a non-trivial dynamical percolation process even when $c\in(1,16)$, which lies outside the more classical range $c\leq 1$ for Liouville quantum gravity surfaces.

\item Finally, we conjecture that the near-critical percolation and the minimal spanning tree studied in Corollaries \ref{c.near-critic} and \ref{c.mst} represent the scaling limit of the associated models on random planar maps. In particular, this work falls into the class of works that study natural continuum processes inspired by statistical physics models on random planar maps. Other works of this type are the works of Miller and Sheffield on the Quantum Loewner evolution (QLE) \cite{qle,lqg-tbm1}. 
QLE represents the conjectural scaling limit of growth models such as the Eden model and DLA on random planar maps.

\end{enumerate}

\subsection{Sketch of proofs}\label{ss.sketch}
We consider three ranges of $\gamma$-values in our proofs, corresponding to the three ranges considered in Theorem \ref{thm1}. This is related to the regularity of LQG measures. 

When $\gamma<2-\sqrt{5/2}$, one proves Theorem \ref{thm1} by adapting the proofs of \cite{gps-near-crit}.  The key reason that the proofs carry through in this regime, is that\footnote{Above we defined $B^{\op h}_r(x)$ only for $x\in\Tg_{\eta}$, but the definition extends immediately to $x\in \C$ by considering $\Tg_\eta$ recentred so that $x\in\Tg_\eta$.}
\begin{equation}\label{eq:key}
\mu_{\gamma h}(B^{\op h}_\eta(x))\ll \alpha^\eta_4(\eta,1)\qquad \textrm{a.s.\ for all $x\in D$ and sufficiently small $\eta>0$.}
\end{equation}
 In a certain sense, on the scale of the microscopic grid $\Tg_\eta$,  this means that there are no sites in $D$ whose clock rates are of order one or higher. See Remark \ref{rmk1}.

In the case $2-\sqrt{5/2}<\gamma < \sqrt{3/2}$, one cannot directly apply the techniques of \cite{gps-near-crit} to prove Theorem \ref{thm1}. The reason is the failure of \eqref{eq:key}. In fact, in this case there are a.s.\ points $x\in D$ such that $\mu_{\gamma h}(B^{\op h}_\eta(x))\gg \alpha^\eta_4(\eta,1)$ as $\eta\to 0$.
To fix this issue one needs to modify the dLDP. 
For some $\varrho>0$ (depending on $\gamma$) to be determined,\footnote{We will choose $\varrho$ such that Lemma \ref{prop:thickpt} is satisfied.} define the \textbf{moderate points of constant $C$} by
\begin{equation}\label{e.MC intro}
\M_{C}:=\{x \in D: \mu_{\gamma h} (B^{\op h}_{2^{-n}}(x))<C\alpha^{2^{-n}}_4(2^{-n},1)2^{-n\varrho}, \text{ for all } n\in \N \}\,.
\end{equation}
Then define the measure $ \wt \mu^C_{\gamma h}(\cdot) := \mu_{\gamma h}(\cdot \cap \M_C)$, and define the dynamical percolation with measure $\wt \mu^C_{\gamma h}$ by 
\eqb
\wt \omega_{\eta}^{C,\gamma}(\cdot):= \omega_{\eta}^{\wt \mu^C_{\gamma h}}(\cdot)\,.
\label{eq:cutoffmeasure}
\eqe
The points of $\M_C$ are called \notion{moderate} since they are points where the rate of the associated clock is $o(1)$.

It is easy to see that for this modified measure the proofs of \cite{gps-near-crit} work, so $\wt \omega_{\eta}^{C,\gamma}(\cdot)$ converges to some process $\wt \omega_{\infty}^{C,\gamma}(\cdot)$. The difficulty now is concentrated in showing that 
$\lim_{\eta \to 0} \omega_\eta^\gamma(\cdot)= \lim_{C \to \infty} \wt\omega_{\infty}^{C,\gamma} (\cdot)$
a.s. To do this we first couple $\wt \omega_{\eta}^{C,\gamma}(\cdot)$ with $\omega_{\eta}^{\gamma}(\cdot)$ in the natural way. Then we show that when $C$ is big enough, for any fixed quad $Q$ and $t\geq 0$, the probability that $Q$ is crossed for 
$\wt \omega_{\eta}^{C,\gamma}(t)$ and not for $\omega_{\eta}^{\gamma}(t)$, or vice versa, converges to $0$.

Finally, when $\gamma\in(\sqrt{3/2},2)$, we study $\P(\Quad \in \omega_{\eta}^\ga(0)\Delta \omega^\ga_{\eta}(t))$ for a fixed quad $\Quad$. This probability can be expressed in terms of the so-called Liouville spectral measure. In this regime it is possible to show that this measure converges in probability to $0$, which implies the same for the considered probability. Intuitively, cLDP is trivial in this regime since the limiting pivotal points are disjoint from the so-called $\gamma$-thick points of $h$, which implies that the limiting LQG pivotal measure $\mu^{\lambda^\eps}_{\gamma h}$ is trivial since $\gamma$-LQG measures are supported on $\gamma$-thick points.

The attentive reader may have realized that for the second and third cases we made reference only to the distribution of quads at a given time $t$. In fact, we are not able to prove convergence for all times simultaneously, and we only get convergence of the finite-dimensional marginals and for the $L^1$ topology rather than for the Skorokhod topology for these cases. However, it is not difficult to see that the process $\omega^\gamma_\infty(\cdot)$ has a c\`adl\`ag modification. When proving this we estimate the number of times a given quad changes from being crossed to not being crossed. See Remark \ref{rmk2}.

Let us now explain how we prove the mixing properties of cLDP when $\gamma<\sqrt{3/2}$. First, we show that if $\S$ is the scaling limit of the spectral measure of the crossing of a given quad $\Quad$, then one can define $\mu_{\gamma h}^{\S}$, which is the LQG measure with base measure $\S$. Then we show the following key identity
\begin{align}\label{e. cov intro}
\Cov(\1_{Q\in \omega_\infty^{\gamma}(0)},\1_{Q\in \omega_\infty^{\gamma}(t)}\mid h)= \E\left[e^{-\mu_{\gamma h}^\S (D)}\1_{\S(D)\neq 0} \mid h \right]\,. 
\end{align}
Using that a.s.\ $\mu_{\gamma h}^\S(D)\neq 0$ on the event $\S(D)\neq 0$, we can prove convergence of the right side to $0$ as $t \to \infty$. The same is true when one studies the events in Theorem \ref{p.quantitative decorrelation}.

To prove the quantitative speed of decorrelation, we use \eqref{e. cov intro} again. The new idea is to take expected value and prove, first, the result in the annealed regime. To do that, we use the quantitative estimates obtained in \cite{GHSS}, which allow us to give explicit polynomial decay, at least for $\gamma<\sqrt{3/4}$. Then we deduce the quenched result from the annealed result, using, among other properties, that the covariance decreases in time. 
\begin{remark}
	We observe in Section \ref{sec:central-charge} that the transition point $\gamma=2-\sqrt{5/2}$ corresponds exactly to central charge $c=1$. It is an interesting coincidence that our stability argument breaks down exactly at $c=1$; note that almost all mathematics literature on LQG considers only the classical range $c\leq 1$ and not the more exotic range $c>1$. We see no apparent reason why the desired stability property (namely, $\mu_{\gamma h}(B^{\op h}_\eta(z))\ll \alpha^\eta_4(\eta,1)$ for all $z\in D$ and sufficiently small $\eta$) should break down exactly at $c=1$. We leave as a curiosity for the interested reader to investigate this further. 
	\label{rmk1}
\end{remark}

\begin{remark}\label{rmk2}
	Similarly as in Remark \ref{rmk3}, it is not clear whether one would expect convergence in Skorokhod topology to hold also for $\gamma>2-\sqrt{5/2}$. As mentioned above, for this range there are vertices where the Poisson clocks have rate of order strictly larger than 1. These are points where the field $h$ has so-called thick points.\footnote{More precisely, the field has thick points of thickness at least $\beta$ for some $\beta\in(0,2)$ depending on $\ga$. Furthermore, we need that the area of the hexagon $B^{\op h}_\eta(z)$ is at least as large as expected assuming that the field is thick at $z$.} The set of thick points forms a fractal set, and one can ask the following question. Given a quad $Q$ and an instance of critical percolation $\omega$, if one is allowed to change the color at the thick points arbitrarily, can one obtain both events $\{\omega\in Q \}$ and $\{\omega\not\in Q \}$ with high probability? If the answer to this question is negative, one would be able to prove Skorokhod convergence, but it is not clear to us whether to expect a positive or negative answer. 
\end{remark}

The paper is organised in the following way. We present some preliminaries on dynamical percolation and Liouville quantum gravity in Section \ref{s.prelim}. In Section \ref{s.3} we prove convergence in Skorokhod topology of dLDP for $\gamma\in(0,2-\sqrt{5/2})$, adapting the techniques in \cite{gps-near-crit}. Based on this technique along with Fourier analysis, we prove convergence in law of the spectral measure in Section \ref{sec:spectralmeasure}. Section \ref{s.ConvLDP} is the main technical contribution of the paper. After establishing convergence of the modified dLDP, we prove that this process is close to true dLDP in the scaling limit, and we prove that the limiting process is c\`adl\`ag. 
 In Section \ref{s6} we prove mixing for cLDP via Fourier analysis techniques, and in Section \ref{s8} we upgrade to quantitative mixing using \cite{GHSS}. In Section \ref{s7} we prove Theorem \ref{thm1} for the supercritical case. In Appendix \ref{app:lqg} we prove various convergence results for LQG measures, and in Appendix \ref{app:spectralsample} we prove upper and lower bounds for the spectral sample associated with the crossing of multiple quads.

\vspace{10pt}
\ni
\textbf{Acknowledgments:}
We thank Ewain Gwynne for helpful comments.
The research of C.G.\ is supported by the 
ANR grant \textsc{Liouville} ANR-15-CE40-0013 and the ERC grant LiKo 676999.
The research of N.H.\ is supported by a fellowship from the Norwegian Research Council, Dr.\ Max R\"ossler, the Walter Haefner Foundation, and the ETH Z\"urich Foundation.
The research of A.S.\ is supported by the ERC grant LiKo 676999.
The research of X.S.\ is supported by Simons Society of Fellows under Award 527901 and by NSF Award DMS-1811092. 
Part of the work on this paper were carried out during the visit of N.H.\ and X.S.\ to Lyon in November 2017 and 2018. 
They thank for the hospitality and for the funding through the ERC grant LiKo 676999. A.S\ would also like to thank the hospitality of N\'ucleo Milenio ``Stochastic models of complex and disordered systems'' for repeated invitation to Santiago were part of this paper was written.

\section{Preliminaries and further background}
\label{s.prelim}

\subsection{Fourier analysis for Boolean functions}
\label{sec:fourier}
In this subsection, we present the theory of Fourier analysis of Boolean functions. 
A function $f$ is said to be \notion{Boolean} if for some finite set $\II$ it is a function from $\{-1,1\}^\II$ to $\{-1,1\}$. We endow $\{-1,1\}^\II$ with the uniform probability measure.

First we will define an appropriate orthonormal basis for the inner product $(f,g)\mapsto\E[fg]$. For any $S\subseteq \II$, let $\chi_S$ be the Boolean function defined by $\chi_S=\prod_{i\in S}\chi_i$, where $\chi_i$ is the Boolean function defined by projection onto the $i$th coordinate. Note that $(\chi_S)_{S\subseteq \II}$ is an orthonormal basis for the functions on $\{-1,1\}^\II$. Therefore,
defining $\wh f(S)=\E[f\chi_S]$ for any Boolean function $f:\{-1,1\}^\II\to \{-1,1\}$,
\begin{equation*}
f=\sum_{S\subseteq \II} \wh f(S) \chi_S\,.
\end{equation*}
By Parseval's formula,
\begin{equation*}
\sum_{S\subseteq\II} (\wh f(S))^2=1.
\end{equation*}
This allows us to define, for every Boolean function $f$, a random variable $\S$ such that $\P(\S=S)=(\wh f(S))^2$. We call $\S$ the \textbf{spectral sample} associated with $f$.

In this paper, we are interested in Boolean functions whose domain are subsets of $\Tg_\eta$. This motivates us to abuse notation and identify its spectral sample $\S_\eta$ with the measure
\footnote{In \cite{gps-fourier}, the distinction is made between the spectral sample $\S_\eta$, viewed as a random set and the \textbf{counting measure} $\lambda_\eta$ on that spectral sample $\S_\eta$. Here for convenience, we identify the concepts.} 
\begin{equation}\label{e. identification}
\S_\eta(d^2z):= \alpha_4^\eta(\eta,1)^{-1}\sum_{x\in \S_\eta}  \1_{z\in B^{\op h}_\eta(x)}d^2z\,.
\end{equation}
Then we can talk about weak convergence of $\S_\eta$ in the space of measures, and  define $\mu_{\gamma h}^{\S_\eta}$, the LQG measure with base measure $\S_\eta$. In the remainder of the paper (except in Appendix \ref{app:spectralsample}) we refer to this measure (rather than the subset of $\C$) when we talk about $\S_\eta$ and $\S$. For $C\in\N$ and $\varrho>0$ as in Lemma \ref{l.bound measure MC} we also define the following truncated Liouville measure
\begin{equation}\label{e. liouville spectral sample}
\wt \mu_{\gamma h}^{C,\S_\eta}(d^2z)= \alpha^\eta_4(\eta,1)^{-1}\sum_{x\in \S_\eta}\1_{z\in B^{\op h}_\eta(x)\cap  \M_C} \mu_{\gamma h}(d^2z)\,,
\end{equation}
where $\M_C$ is as in \eqref{e.MC intro}. Let us note that for any set $E\subseteq D$, $\wt \mu_{\gamma h}^{C,\S_\eta}(E)= \mu_{\gamma h}^{\S_\eta}(E\cap \M_C)$.

The following key identities express the covariance between $f(\omega^\ga_\eta(0))$ and $f(\omega^\ga_\eta(t))$ (and with $\wt\omega_\eta^{C,\ga}(\cdot)$ instead of $\omega^\ga_\eta(\cdot)$) in terms of the spectral measure.
\begin{lemma}\label{l.covariance}
	Let $f$ be a Boolean function defined on a finite subset of $\Tg_\eta$ and let $\S_\eta$ be the associated spectral measure. Then 
	\begin{align}\label{e.decreasingcovariancecond}
	&\Cov\left[f(\omega^\ga_\eta(0)),f(\omega^\ga_\eta(t)) \mid h\right]=\E\left[ e^{-t\mu_{\gamma h}^{\S_\eta}(D)} \1_{\S_\eta(D)\neq 0}\mid h\right],\\
	&\label{e.decreasingcovariancetildecond}	\Cov\left[f(\wt\omega^{C,\ga}_\eta(0)),f(\wt \omega^{C,\ga}_\eta(t)) \mid h\right]=\E \left[e^{-t\wt\mu_{\gamma h}^{\S_\eta}(D\cap \M_C)} \1_{\S_\eta(D)\neq 0}\mid h \right]\,.
	\end{align}
	(See Sections \ref{ss.sketch} and \ref{s.ConvLDP} for the definition of the process $\wt \omega^{C,\ga}_\eta(t)$). Furthermore, the function $(f,g)\mapsto \Cov[f(\wt \omega^{C,\ga}(0)), g(\wt \omega^{C,\ga}(t))\mid h]$ is an inner product if we identify functions that differ by a constant,\footnote{We identify functions that differ by a constant since $(f,g)=0$ if $f\equiv a$ or $g\equiv a$ for some constant $a$.} so the Cauchy-Schwarz inequality gives that for any Boolean function $g$, 
		\begin{align}\label{eq:fg}
		\Cov[f(\wt \omega^{C,\ga}_\eta(0)), g(\wt \omega^{C,\ga}_\eta(t))\mid h] \le \Cov[f(\wt \omega^{C,\ga}_\eta(0)), f(\wt \omega^{C,\ga}_\eta(t))\mid h]^{1/2}\,.
		\end{align}
\end{lemma}
\begin{proof}
	Let us first note that if $S\neq S'$, then $\E\left[\chi_S(\omega^\ga_\eta(0))\chi_S'(\omega^\ga_\eta(t))\mid h\right]=0$. Furthermore, letting $\omega_{\eta,x}^\ga(t)\in\{-1,1 \}$ describe whether $x\in \Tg_\eta$ is open or closed, since for any $x\in S$ we have 
	$\E[\omega^\ga_{\eta,x}(0)\omega^\ga_{\eta,x}(t)\mid h]=\exp(-t\alpha^\eta_4(\eta,1)^{-1}\mu_{\gamma h}(B^{\op h}_\eta(x)))$, we see that
	\begin{equation*}
	\begin{split}
	\E\left[\chi_S(\omega^\ga_\eta(0))\chi_S(\omega^\ga_\eta(t))\mid h\right]
	&=\exp\Big (-\sum_{x\in S}t\alpha^\eta_4(\eta,1)^{-1}\mu_{\gamma h}(B^{\op h}_\eta(x))\Big )\\
	&=\exp(-t\mu_{\gamma h}^{S}(D))\,.
	\end{split}
	\end{equation*}
	 This implies that
	\eqb
	\E\left[f(\omega^\ga_\eta(0))f(\omega^\ga_\eta(t))\mid h \right]= \sum_{S\subseteq \II} (\wh f(S))^2 \exp(-t\mu_{\gamma h}^{S}(D))= \E\left[ \exp\left (-t\mu_{\gamma h}^{\S_\eta}(D)\right )\mid h\right]\,.
	\label{eq5}
	\eqe
	We conclude the proof of \eqref{e.decreasingcovariancecond} by noting that $\wh f(\emptyset)= \E\left[f(\omega^\ga_\eta(0))\right]$
	and subtracting $\wh f(\emptyset)^2$ on both sides. The same proof works for $\wt \omega_\eta^{C,\ga}(\cdot)$.
	
	For the second part, by the same calculation,
	\[
	\Cov[f(\wt \omega^{C,\ga}_\eta(0)), g(\wt \omega^{C,\ga}_\eta(t))\mid h] =\sum_{S(D)\neq 0} \wh f(S)\wh g(S)\exp(-t
			\mu^{\S_\eta,C}_{\gamma h}(D)).
	\]	
	From this identity we see that  $(f,g)\mapsto\Cov[f(\wt \omega_\eta^{C,\ga}(0)), g(\wt \omega_\eta^{C,\ga}(t))]$ is an inner product if we identify $f$ with the constant function $x\mapsto 0$ if $\wh f(S)=0$ for all $S\neq\emptyset$; equivalently, we identify $f$ with $x\mapsto 0$ if $f\equiv a$ for some constant $a$.
	The Cauchy-Schwarz inequality gives
	\begin{align*}
	\Cov[&f(\wt \omega_\eta^{C,\ga}(0)), g(\wt \omega_\eta^{C,\ga}(t))\mid h]\\
	& \le \Cov[f(\wt \omega^{C,\ga}_\eta(0)), f(\wt \omega^{C,\ga}_\eta(t))\mid h]^{1/2} \Cov[g(\wt \omega^{C,\ga}_\eta(0)), g(\wt \omega^{C,\ga}_\eta(t))\mid h]^{1/2}\\
	&\le \Cov[f(\wt \omega^{C,\ga}_\eta(0)), f(\wt \omega^{C,\ga}_\eta(t))\mid h]^{1/2}\,.
	\end{align*}
\end{proof}
\begin{remark} By taking expected value in  equations \eqref{e.decreasingcovariancecond} and \eqref{e.decreasingcovariancetildecond}, one obtains
	\begin{align}\label{e.decreasingcovariance}
	&\Cov\left[f(\omega^\ga_\eta(0)),f(\omega^\ga_\eta(t)) \right]=\E\left[e^{-t\mu_{\gamma h}^{\S_\eta}(D)} \1_{\S_\eta(D)\neq 0} \right],\\
	&\label{e.decreasingcovariancetilde}	\Cov\left[f(\wt\omega^{C,\ga}_\eta(0)),f(\wt \omega^{C,\ga}_\eta(t)) \right]=\E \left[e^{-t\wt\mu_{\gamma h}^{C,\S_\eta}(D)} \1_{\S_\eta(D)\neq 0} \right]\,.
	\end{align}
	Furthermore, if $g$ is any other Boolean function,
	\begin{align}\label{eq:fg2}
	\Cov[f(\wt \omega^{C,\ga}_\eta(0)), g(\wt \omega^{C,\ga}_\eta(t))] \le \Cov[f(\wt \omega^{C,\ga}_\eta(0)), f(\wt \omega^{C,\ga}_\eta(t))]^{1/2}\,.
	\end{align}
\end{remark}

\subsection{Quad-crossing space}
\label{s.quad}
The idea in \cite{ss-planar-perc} is to consider a percolation configuration as the set of all quads crossed by it. Let us start by defining what a quad is.
\begin{definition} \label{def:quad}
	Let $D\subset\C$ be bounded. A {\bf quad} $\Quad$ in $D$ is a homeomorphism $\Quad:[0,1]^2\to D$. Let $\QQ_D$ denote the set of quads, equipped with the topology generated by the following (pseudo)metric $d_Q(Q_1,Q_2):=\inf_\phi \sup_{z\in\partial[0,1]^2}|Q_1(z)-Q_2(z)|$, where the infimum is over all homeomorphisms $\phi:[0,1]^2\to[0,1]^2$ which preserve the four corners of the square. A {\bf crossing} of a quad $\Quad$ is a connected closed subset of $\Quad([0,1]^2)$ that intersects both boundaries  $\partial_1 \Quad = \Quad(\{0\} \times [0,1])$ and $\partial_3 \Quad= \Quad(\{1\}\times[0,1])$.
\end{definition}
We say that a quad is rectangular if $Q([0,1]^2)$ is a rectangle and if the four corners of $[0,1]^2$ are mapped to the four corners of $Q([0,1]^2)$ by the map $Q$.

The space of quads has a natural partial order induced by the crossings. We write $\Quad_1\leq \Quad_2$ if any crossing of $\Quad_2$ contains a crossing of $\Quad_1$. We say that a subset $S\subseteq \QQ_D$ is \textbf{hereditary} if, whenever $Q\in S$ and $Q'\in \QQ_D$ satisfies $Q'\leq Q$, we have $Q'\in S$. Note that if we are given an instance of site percolation on $\Tg_\eta$ (equivalently, a percolation on the faces of the hexagonal lattice) and let $S$ be the set of quads which are crossed by the set of open hexagons, then $S$ is necessarily hereditary.
\begin{definition}[The space $\HH$] Let $\HH$ be the collection of all closed hereditary subsets of $\QQ_D$.
\end{definition}
In this paper we will consider two different topologies on $\QQ_D$. The first topology is the so-called Schramm-Smirnov topology, which was also considered in \cite{gps-near-crit} and originally introduced in \cite{ss-planar-perc}. This gives a compact, Polish, and metrizable space, and we let $d_\HH$ denote a metric which generates the topology. For any $k\geq 1$, let $\QQ^k$ be the set of all quads which are polygonal in $D\cap (2^{-k}\Z^2)$, i.e., quads whose boundaries are included in $D\cap (2^{-k} \Z^2)$, and denote by $\QQ^\infty$ the union $\QQ^\infty:=\cap_{k\in\N}\QQ^k$. Then the Borel $\sigma$-algebra of $(\HH,d_{\HH})$ is generated by the sets $\{Q\in\omega \}$ for $Q\in\QQ^\infty$. We refer to \cite[Section 2.2]{gps-near-crit} for further properties of the space $(\HH,d_{\HH})$. 

The other topology on $\HH$ is slightly stronger. Define the following distance on $\HH$
\begin{equation}\label{e.distance}
d_\HH^{\op{mod}}(\omega,\omega'):= \sup\{2^{-k}: \text{there is } \Quad \in \QQ^k \text{ such that } \Quad \in \omega\Delta \omega' \}\,.
\end{equation}
It is possible to see that this metric generates a finer topology than the one of Schramm-Smirnov.\footnote{One can check that, in the notation of \cite{gps-near-crit}, both $\boxminus_\Quad^c$ and $\boxdot_U^c$ are open for the topology generated by $d_\HH^{\op{mod}}$.} Additionally, let us note that under this metric $\HH$ can be identified with $\{0,1\}^{\QQ_\infty}$ equipped with the product topology and the appropriate metric. The space $\HH$ is not complete under the metric $d_\HH^{\op{mod}}$.\footnote{For example, consider the sequence of elements in $\HH$ such that the $n$th element consists of the quad $Q_n( (x,y) )=((1-1/n)x,y)$ and all quads $Q$ satisfying $Q\leq Q_n$. This sequence is Cauchy, but does not have a limit, since (by the requirement that the elements of $\HH$ are closed) the limiting object would need to contain the quad $Q(z)\equiv z$, while the limit cannot contain this quad by definition of $d_\HH^{\op{mod}}$.} 

At several occasions we will use the following lemma to upgrade convergence statements from $d_\HH$ to $d_\HH^{\op{mod}}$. 
\begin{lemma}
	Assume $\omega_\eta$ for $\eta\in(0,1]$ is a collection of random elements in $\HH$ such that $\omega_\eta\rta\omega_\infty\in\HH$ a.s.\ for the metric $d_\HH$ as $\eta\rta 0$. If $\omega_\infty$ has the law of the critical percolation scaling limit, then $\omega_\eta\rta\omega_\infty\in\HH$ a.s.\ for the metric $d_\HH^{\op{mod}}$ as $\eta\rta 0$.
	\label{prop:ss-meas}
\end{lemma}
\begin{proof}
	The lemma is a direct consequence of \cite[Lemma 5.1]{ss-planar-perc}.
\end{proof}

The metric $d_\HH^{\op{mod}}$ is sometimes easier to work with since it is explicit. However, when we consider c\`adl\`ag processes $\omega(\cdot)$ such that $\omega(t)\in\HH$ for each $t\in\R_+$ we want to use the metric $d_\HH$ due to completeness of the space.
\begin{definition}[Skorokhod space]
	For $T>0$ let $(\Sk_T,d_{\Sk_T})$ denote the set of c\`adl\`ag functions $\omega:[0,T]\to\HH$ equipped with the following metric $d_{\Sk_T}$
	\eqbn
	d_{\Sk_T}(\omega,\omega')
	=\inf_{\phi}\sup_{t\in[0,T]}
	d_{\HH}(\omega(\phi(t)),\omega'(t)),
	\eqen
	where $\phi:[0,T]\to[0,T]$ is an increasing homeomorphism.
	
	Let $(\Sk,d_{\Sk})$ denote the set of c\`adl\`ag functions $\omega:\R_+\to\HH$ equipped with the following metric $d_{\Sk}$
	\eqbn
	d_{\Sk}(\omega,\omega')
	=\sum_{k=1}^{\infty}
	\inf_{\phi}
	2^{-k}\wedge\bigg(\sup_{t\in\R_+\,:\, t\wedge\phi(t)\leq k }d_{\HH}(\omega(\phi(t)),\omega'(t))\bigg),
	\eqen
	where $\phi:\R_+\to\R_+$ is an increasing homeomorphism. 
	\label{def:skorokhod}
\end{definition}

For the case when $\gamma>(2-\sqrt{5/2},2)\setminus\{\sqrt{3/2} \}$ we do not prove convergence of $\omega_{\eta}^\gamma(\cdot)$ in Skorokhod space, but rather in the following weaker topology.
\begin{definition}
	For $T>0$ assume $\omega,\omega':[0,T]\to \scr H$, and define the following distance
	\eqbn
	d_{L^1}(\omega,\omega')=\int_{0}^{T}d_{\scr H}( \omega(t),\omega'(t) )\,dt\,.
	\eqen
	If $\omega,\omega':\R_+\to\scr H$ define the following distance
\eqbn
d_{L^1}(\omega,\omega')
=\sum_{k=1}^{\infty}
2^{-k}\wedge d_{L^1}(\omega|_{[0,k]},\omega'|_{[0,k]})\,.
\eqen
\label{def:L1}
\end{definition}

\subsection{Continuum Euclidean dynamical percolation}
\label{sec:gps-dp}
In the following section, we present the main definitions and results used to prove the convergence of the classical dynamical percolation to its continuum counterpart (i.e., the Euclidean case $\gamma=0$). This element will be important, as the first step in the proof of convergence of LDP for $\gamma\neq 0$ follows the same lines as in the classical case. We refer to Section \ref{s.3} for further details.

Let us start by defining four arm events and the four arm exponent. For a coloring of $\Tg_\eta$ we define an {\bf arm} to be a simple path of vertices such that all the vertices have the same color and consecutive vertices in the path are adjacent in $\Tg_\eta$. Let $A_1$ and $A_2$ be bounded simply connected domains in $D$, such that $\ol{A_1}\subset A_2$. Define $A=A_2\setminus A_1$, so that $A$ is a topological annulus. We say that a site $z$ is $A$-important if $z\in A_1$ and if there are four arms of alternating color connecting $z$ to $\partial A_2$. We call this event a four arm event. For $0<r<R$ let $\alpha_4^\eta(r,R)$ denote the probability that the four arm event happens with $A_1$ (resp.\ $A_2$) the square of side length $2r$ (resp.\ $2R$) centred at $z$. 
It was proven in \cite{smirnov-werner-percolation} that $\alpha_4^\eta(r,R)=(r/R)^{5/4+o(1)}$.

Next let us define the set of $\eps$-important pivotal points.
\begin{definition}[Importance of a point]\label{d. important}
	Let $\eps>0$ and consider the grid $\eps\Z^2$. For $z\in\R^2$ let $A_z$ be the topological annulus as defined above, with $A_1$ equal to the square of $\eps\Z^2$ containing\footnote{In case $z$ lies on the grid (i.e., at least one of its coordinates is an integer multiple of $\eps$), choose $A_1$ arbitrarily among the possible squares.} $z$, and $A_2$ the square of side length $3\eps$ concentric with $A_1$. We say that $z$ is $\eps$-important if $z$ is $A_z$-important. 
\end{definition}
Let us now consider convergence of the $\eps$-important points. Define 
\begin{equation}\label{e.lambdaeps}
\lambda^\epsilon_\eta(d^2z)
:= \sum_{x\in \Tg_\eta}\1_{z\in B^{\op h}_\eta(x)}\1_{x\text{\,\,is\,\,}\eps\text{-important}} \alpha^\eta_4(\eta,1)^{-1} d^2z\,.
\end{equation}
The next theorem follows from \cite[Theorem 1.1]{gps-pivotal} (see also \cite[Theorem 2.13]{gps-near-crit}).
\begin{thm}[\cite{gps-pivotal}] 
	For any $\epsilon>0$, there is a measurable map $\lambda^\epsilon$ from $(\HH,d_\HH)$ into the space of finite Borel measures on $\ol D$ such that as $\eta \to 0$
	\[(\omega_\eta,\lambda^\epsilon(\omega_\eta))\stackrel{d}{\to}(\omega_\infty,\lambda^\epsilon(\omega_\infty)),\]
	where we use the Schramm-Smirnov topology in the first coordinate and the weak convergence of measures in the second coordinate.
	\label{thm:gps-dp-main}
\end{thm}

\subsection{Liouville quantum gravity}\label{ss.GMC}
Our presentation of LQG is going to be based on \cite{berestycki-elem}, and for more advanced results we are going to rely on \cite{aru17} (see also the useful review \cite{rhodes-vargas-review}). Let $h$ be a log-correlated field on a bounded simply connected domain $D$. More precisely, let $h$ be a centred Gaussian field with correlations given by a non-negative definite kernel
\begin{equation}
K(x,y):= -\log(|x-y|)+ g(x,y),
\label{eq:kernel}
\end{equation}
where $g$ is continuous over $\ol D\times \ol D$. In the rest of the paper we will always assume that the considered field $h$ satisfies these assumptions. As ``$K(x,x)=\infty$'' this definition does not make rigorous sense, and we obtain a  precise mathematical definition by considering a centred Gaussian process $(h,\rho)$ such that the variance satisfies 
\begin{equation*}
\E\left[(h,\rho)^2\right]=\iint_{D\times D} \rho(x) K(x,y) \rho(y) d^2z,
\end{equation*}
where $\rho$ can take any value such that the right side is finite. In particular, $\rho$ can be any continuous function in $\overline D\times \overline D$.

The LQG measures may be constructed via an approximation procedure. Take $\epsilon>0$ and define $h_\epsilon=h*\theta_\epsilon$, where $\theta_\epsilon$ is an appropriate mollifier. If $h_\eps$ is a circle average or if $\theta_\eps$ is smooth, one can argue that  $h_\epsilon(\cdot)$ has a continuous version. We assume  $\Var(h_\epsilon(z))=\log\eps^{-1}+O(1)$ for all $z\in D$ with distance at least $\eps$ from $\C\setminus D$, where $O(1)$ is uniform in the choice of $z$ bounded away from $\partial D$.

Let $\gamma\geq0$, let $\sigma$ be a positive measure on $D$, and let $h$ be a log-correlated field. We define the LQG measure $\mu^\sigma_{\gamma h}$ associated to $h$ with base measure $\sigma$ and parameter $\gamma$ by 
\begin{equation} \label{e. def GMC}
\mu^\sigma_{\gamma h}(d^2z)= \lim_{\epsilon\to 0}\eps^{\gamma^2/2} e^{\gamma h_\epsilon(z)}\sigma(d^2z),
\end{equation}
where the limit is taken in the topology of weak convergence of measures. In this work, we write $\mu_{\gamma h}$ when $\sigma$ is the Lebesgue measure restricted to $D$.

The non-triviality of the limit depends on the dimension of the measure. For $d\geq0$, we say that $\sigma$ has finite \textbf{$d$-energy} if
\eqb
\EE_d(\sigma):=\iint \frac 1 {| x- y|^{d}} \sigma(d^2x) \sigma(d^2y)<\infty\,.
\label{eq11}
\eqe 
We define $\text{dim}(\sigma)$ as the supremum of $d$ such that the $d$-energy is finite, i.e.,
\eqbn
\text{dim}(\sigma) = \sup\{ d\geq 0\,:\,\EE_d(\sigma)<\infty \}\,.
\eqen 

When $\gamma<\sqrt{2 \text{dim}(\sigma)}$ then the limit in \eqref{e. def GMC} exists in $L^1$ and is non-trivial a.s.\ (see e.g.\ \cite{berestycki-elem}). By e.g.\ \cite{kahane},
\eqb
\mu^\sigma_{\gamma h}(D)>0\qquad\text{a.s.}
\label{p.nontrivial}
\eqe
Furthermore, we proved the following lower bound for $\mu^\sigma_{\gamma h}(D)$ in \cite{GHSS}.
\begin{prop}\label{p.mass}[Corollary 3.2 in \cite{GHSS}]
	For $d,\gamma\geq 0$ define $\theta=\theta(d,\gamma):= \frac {d - \gamma^2} {d + \gamma^2}$.
	If $\gamma<\sqrt{d}$, then there exists a $K>0$ (depending only on the law of the log-correlated field $h$) such that for any
	\begin{equation}
	t\geq t_0:= K \left[\frac{\EE_d(\sigma)}{\sigma(D)}\right]^{1/\theta}
	\end{equation}
	we have
	\begin{equation}\label{e. first paper 1}
	\E\left[e^{-t\mu^\sigma_{\gamma h}(D)} \right]\leq \frac{K}{\sigma(D) t^\theta}\,.
	\end{equation}
\end{prop}
Let us now discuss the regularity of LQG measures. See e.g.\ \cite[Corollary 6.5]{aru17} for the following result. We remark that the constraint $\gamma<2-\sqrt{5/2}$ corresponds to central charge $c<1$; in particular, the variant of Liouville dynamical percolation corresponding to dynamical percolation on uniformly sampled maps ($c=0$) is covered.
\begin{lemma}\label{prop7}
	Let $\delta>0$ and $\gamma\in[0,2)$, let $h$ be a log-correlated field, and define $\beta_\gamma:=2-2\gamma+\gamma^2/2$. Then there exists a random $C>0$ such that for all $r\in(0,1)$ and $z\in D$, we have that $\mu_{\gamma h}(B^{\op h}_r(z))<Cr^{\beta_\gamma-\delta}$. 
	
	In particular, if $\gamma<2-\sqrt{5/2}$, there is a deterministic $\delta>0$  and a random $C>0$ such that $\mu_{\gamma h}(B^{\op h}_r(z))\leq C\alpha^\eta_4(r,1)r^\delta=Cr^{5/4+\delta +o(1)}$ a.s.
\end{lemma}
A variant of this lemma holds for the measure $\wt \mu_{\gamma h}^{C}$ defined in \eqref{eq:cutoffmeasure}.
\begin{lemma}\label{l.bound measure MC}
	There is a universal constant $K>0$ such that for all $x\in D$, $C>0$, and $r=2^{-n}\leq 2^{-1}$, 
	\begin{equation}
	\wt\mu_{\gamma h}^C(B^{\op h}_r(x))
	=\mu_{\gamma h}(B^{\op h}_r(x)\cap \M_C)
	\leq 
	CK\alpha^\eta_4(r,1)r^\varrho\,.
	\end{equation}
\end{lemma}
\begin{proof}
	If $\wt\mu^C_{\gamma h}(B^{\op h}_r(x))>0$ then there exists $z\in D$ such that $z\in \M_C$ and $|x-z|<r$. Thus, 
	\begin{align*}
	\mu_{\gamma h}(B^{\op h}_r(x))
	\leq \mu_{\gamma h}(B^{\op h}_{2r}(z))
	\leq C\alpha^\eta_4(2r,1)r^\varrho
	\leq CK\alpha^\eta_4(r,1)r^\varrho,
	\end{align*}
	where we use quasi-multiplicativity of the four arm probability \cite{gps-pivotal} in the last step.
\end{proof}

The $\gamma$-LQG measure is supported on so-called $\gamma$-thick points.
\begin{lemma}
	For $C>0$ and $\varrho>0$ let $\M_C$ be defined by \eqref{e.MC intro}. For $\gamma<\sqrt{3/2}$ and $\varrho$ sufficiently small as compared to $\gamma$, 
	$$
	\lim_{C\rta\infty}\mu_{\gamma h} (D\setminus \M_C)=0.
	$$
	\label{prop:thickpt}
\end{lemma}
\begin{proof}
	Let us start by defining
\begin{equation}
\M_{C,n}:=\{x \in D: \mu_{\gamma h} (B^{\op h}_{2^{-n}}(x))<C\alpha^{2^{-n}}_4(2^{-n},1)2^{-n\varrho} \}.
\end{equation}
Note that $\M_{C}=\bigcap \M_{C,n}$.  Let us first bound the Liouville mass of $D\backslash M_{C,n}$, to do this let us first define $x_n$ as the element of $2^{-n}\Z^d$ closest to $x$ and note that
\[ \1_{x\notin \M_{C,n}}\leq\left( \frac{\mu_{\gamma h} (B^{\op h}_{2^{-(n-1)}}(x_n))}{C\alpha^{2^{-n}}_4(2^{-n},1)2^{-n\varrho}}\right)^p.\]
This implies that
\begin{align*}
\int  \1_{x\notin \M_{C,n}} d\mu^\gamma(dx)&\leq C^{-p} \sum_{x \in 2^{-n}\Z^d} (\mu_{\gamma h}(B^{\op h}_{2^{-(n-1)}}(x_n)))^{p+1}(\alpha^{2^{-n}}_4(2^{-n},1))^{-p}2^{n\varrho p}.
\end{align*}
Now, we want to use the expected value. To do that, we use \cite[Proposition 4.1 and Corollary 6.2]{aru17}, which states that for any log-correlated field and any $q<4/\gamma^2$,
\begin{align}\label{e. momentsagain?}
\E\left[\mu_{\gamma h}(B^{\op h}_r(x))^q \right]\leq r^{ -\gamma^2 q^2/2 + (2+\gamma^2/2)q +O(1)},
\end{align}
where the $O(1)$ is uniform in $x$. Since $\gamma<\sqrt{3/2}$ we can find a small $\varrho>0$ and a constant $K>0$ such that for all $n\in\N$ we have $K\alpha^{2^{-n}}_4(2^{-n},1)>2^{-n(2-\gamma^2/2 - 2\varrho)}$. This gives that for $n\in\N$, 
\begin{align*}
\E\left[\int  \1_{x\notin \M_{C,n}} d\mu^\gamma(dx)\right ]&\leq C^{-p} \sum_{x \in 2^{-n}\Z^d} 2^{n\Big(\frac{\gamma^2 (p+1)^2}{2}-(2+\gamma^2/2)(p+1)+O(1)\Big)}(\alpha^{2^{-n}}_4(2^{-n},1))^{-p} 2^{n\varrho p}\\
&\leq O(1)K^p C^{-p} 2^{n\frac{\gamma^2 p^2}{2}-n \rho p}.
\end{align*}
Let us note that we can choose $p>0$ small enough so that the exponent is negative. This implies, summing over $n\in\N$, that $\E\left[\mu_{\gamma h} (D\setminus \M_C) \right] \leq O(1) K^pC^{-p}$. We conclude by taking $C\rta\infty$.
\end{proof}

\subsection{Central charge}
\label{sec:central-charge}

This section gives further background to understand the main motivations of this work (explained in Subsection \ref{ss.motiv}). It will not be used in the rest of the paper and can be skipped at first reading.
Liouville quantum gravity surfaces are associated with a \textbf{central charge} $c\in(-\infty,25)$ and a \textbf{background charge} $Q>0$. These parameters are related to each other by $c=25-6Q^2$.
Most probability literature on LQG considers the range $c\in(-\infty,1]$ (equivalently, $Q\geq 2$).
LQG for $c\in \C$ is studied in multiple works in the physics literature \cite{david-c>1-barrier,bh-c-ge1-matrix,fkv-c>1-I,fk-c>1-II,teschner04,zam05,ribault-cft,rs15,ijs16,ribault2018minimal}, but to our knowledge the only other papers which study $c\not\in(-\infty,1]$ in a probabilistic setting are \cite{c-greater-than-1,gwynne-pfeffer-c-gtr-1}, which consider the range $c\in(1,25)$. 

In \cite{shef-zipper} a Liouville quantum gravity surface with background charge $Q\geq 2$ is defined to be an equivalence class of pairs $(D,h)$, where $h$ is a distribution on a domain $D\subset\C$. Furthermore, two pairs $(D,h)$ and $(\wt D,\wt h)$ are equivalent if there exists a conformal map $\phi:\wt D\to D$ such that 
\begin{align}\label{eq:coordinatechange}
\wt h=h\circ\phi+Q\log|\phi'|\,.
\end{align}
It is observed in \cite{c-greater-than-1} that this definition of an LQG surface may be extended to $Q\in (0,2)$.

Let $(D,h)$ be an equivalence class representative for an LQG surface. Let $A\subset D$ and $d\in(0,2]$, and assume that the $d$-dimensional Minkowski content of $A$ defines a locally finite and non-trivial measure $\lambda$ which is supported on $A$. Assuming $\lambda$ has finite $d'$ dimensional energy for all $d'\in(0,d)$, one may define an LQG measure $\mu^\lambda_{\gamma h}=e^{\gamma h}d\lambda$ supported on $A$ for any $\gamma<\sqrt{2d}$ \cite{berestycki-elem,shef-kpz,rhodes-vargas-review}. If we want to interpret the measure $\mu^\lambda_{\gamma h}$ as intrinsic to the LQG surface it is natural to require that the measure is invariant under coordinate changes, i.e., if $\wt\mu^{\wt\lambda}_{\gamma \wt h}=e^{\gamma \wt h}d\wt\lambda$ for $\wt\lambda$ the Minkowski content of $\phi^{-1}(A)$, then $\wt\mu^{\wt\lambda}_{\gamma \wt h}(\phi^{-1}(U))=\mu^\lambda_{\gamma h}(U)$ a.s.\ for any fixed $U\subseteq D$. By the coordinate change formula \eqref{eq:coordinatechange} for LQG surfaces with a given background charge $Q$, it is seen that we need to have 
$$
Q=d/\gamma+\gamma/2
$$ 
in order for the LQG measure to be invariant under coordinate changes (see \cite[Proposition 2.1]{shef-kpz} and \cite[Proposition 2.2]{c-greater-than-1}). 

Our paper studies LQG measures supported on the set of CLE$_6$ pivotal points. The CLE$_6$ pivotal points have Hausdorff dimension $d=3/4$ \cite{miller-wu-dim} and (after applying a cut-off) well-defined $3/4$-dimensional Minkowski content which defines a non-trivial and locally finite measure \cite{hlls-pivot}. We now see why the case of uniformly sampled planar maps (which is the case relevant for the Cardy embedding project described in Section \ref{ss.motiv}) relies on the particular case of $\gamma=\sqrt{1/6}$ in Theorem \ref{thm1}. This is due to the fact that if one plugs $\gamma=\sqrt{1/6}$ into $(3/4)/\gamma+\gamma/2$ one recovers $Q_{c=0}=\sqrt{3/2}+\sqrt{2/3}$.

In greater generality, the $\gamma$-LQG measure supported on these points may be defined for $\gamma<\sqrt{3/2}$. Combining the formulas above, this gives $c<1$ (resp., $c\in(1,16)$) if and only if $\gamma\in(0,2-\sqrt{5/2})$ (resp., $\gamma\in(2-\sqrt{5/2},\sqrt{3/2})$). In other words, the two transitions points for $\gamma$ in Theorem \ref{thm1} corresponds to $c=1$ and $c=16$. See Remark \ref{rmk1} for a discussion of the first of these transition points.

\section{Convergence of the LDP: direct microscopic stability}
\label{s.3}
The proof of convergence of $\gamma$-LDP for $\gamma\in(0,2-\sqrt{5/2})$ is based on the proof of \cite[Theorem 1.4]{gps-near-crit}. As many of the techniques are the same in this case, we strongly advise the reader to read this chapter alongside with \cite{gps-near-crit}. In many parts, when needed, we will just cite the results from \cite{gps-near-crit}. We assume $\gamma\in(0,2-\sqrt{5/2})$ throughout the section so that we can apply Lemma \ref{prop7}.

\subsection{Dynamical percolation with cut-off}\label{s.cut-offeps}

The first step is to prove convergence of dLDP with a cut-off. This process is defined such that we only update $\eps$-important pivotal points.  It is easier to prove convergence of this process than of the full dynamic since the limiting process will make finitely many jumps in any bounded interval, as opposed to the case without a cut-off.

Let $(\omega^{\eps, \gamma}_\eta(t))_{t\geq 0}$ denote dynamical percolation on $\Tg_\eta$, where we only update pivotal points which are $\eps$-important for $\omega^\ga_\eta(0)$ (see Definition \ref{d. important}). This process can be sampled by considering a Poisson point process (PPP) $\{(x^{\eps,\eta}_i,t^{\eps,\eta}_i,\xi^{\eps,\eta}_i)\,:\,i\in\N \}$ on $D\times\R_+\times\{-1,1 \}$ with intensity $\lambda^\eps_\eta\times\op{Leb}\times\op{Uniform}$, where $\lambda^\eps_\eta$ is given by \eqref{e.lambdaeps}, and then set the color of $x^{\eps,\eta}_i$ equal to $\xi^{\eps,\eta}_i$ at time $t^{\eps,\eta}_i$ for each $i\in\N$ (such that $\xi_i^{\eps,n}=-1$ means white/closed and $\xi_i^{\eps,n}=1$ means black/open). 

Let us start by explaining how to define the continuum analogue $\omega^{\eps, \gamma}_\infty(\cdot)$ of $\omega^{\eps, \gamma}_\eta(\cdot)$, following the strategy of \cite[Sections 5 to 7]{gps-near-crit}. Let $\lambda^\eps$ denote the Euclidean pivotal measure supported on the $\eps$-important pivotal points of $\omega^{\eps,\gamma}_\infty(0)$ as in Theorem \ref{thm:gps-dp-main}, and consider the associated LQG measure $\mu^{\lambda^\eps}_{\gamma h}$. Consider a PPP $\{(x^\eps_i,t^\eps_i,\xi^\eps_i)\,:\,i\in\N \}$ on $D\times\R_+\times\{-1,1 \}$ with intensity $\mu^{\lambda^\eps}_{\gamma h}\times\op{Leb}\times\op{Uniform}$.

As explained above, it is immediate in the discrete that if we know $\omega^{\eps,\gamma}_\eta(0)$ and $\{(x^{\eps,\eta}_i,t^{\eps,\eta}_i,\xi^{\eps,\eta}_i)\,:\,i\in\N \}$, then we can determine $\omega^{\eps,\gamma}_\eta(t)$ for all $t\geq 0$. In the continuum this is not obvious. Garban, Pete, and Schramm \cite{gps-near-crit} develop a theory of so-called networks to argue that for a fixed quad $Q$, knowledge of $\omega^{\eps,\gamma}_\infty(0)$ and $\{(x^\eps_i,t^\eps_i,\xi^\eps_i)\,:\,i\in\N \}$ does determine $\omega^{\eps,\gamma}_\infty(t)$ for each $t\geq 0$. Since $\omega\in\scr H$ is determined by the events $\{Q\in\omega \}$ for countably many quads $Q$, this is sufficient to conclude.

For each fixed quad $Q$ and time $T\geq 0$, Garban, Pete, and Schramm define a network $N_{Q,T}$, which is a kind of graph structure with vertices $\{(x^\eps_i,t^\eps_i,\xi^\eps_i)\,:\,i\in\N,t^\eps_i\leq T \}$ and two types of edges (primal and dual). The network represents the connectivity properties of $Q$ at time $T$ if we do not have knowledge about the percolation in an infinitesimal neighborhood around each $x_i$ for which $t_i\leq T$. The network is defined as a limit of a certain mesoscopic network, and it is proved that $N_{Q,T}$ is measurable with respect to $\omega^{\eps,\gamma}_\infty(0)$ and $\{(x^\eps_i,t^\eps_i,\xi^\eps_i)\,:\,i\in\N,t_i\leq T \}$. One can determine whether $Q\in\omega^{\eps,\gamma}_\eta(T)$ by using $N_{Q,T}$ and the random variables $\xi^\eps_i$ for $t_i\leq T$.

To prove that $\omega_\eta^{\eps,\gamma}(\cdot)$ converges to $\omega_\infty^{\eps,\gamma}(\cdot)$ in the Skorokhod topology (Definition \ref{def:skorokhod}) we use the strategy of \cite[Section 7]{gps-near-crit}. The idea is to couple 
$\omega^{\eps,\gamma}_\eta(0)$, 
$\mu_{\gamma h}^{\lambda^\eps_\eta}$, and $\{(x^{\eps,\eta}_i,t^{\eps,\eta}_i,\xi^{\eps,\eta}_i)\,:\,i\in\N \}$
such that they are close to some limit 
$\omega^{\eps,\gamma}_\infty(0)$,
$\mu_{\gamma h}^{\lambda^\eps}$, and $\{(x^\eps_i,t^\eps_i,\xi^\eps_i)\,:\,i\in\N \}$. We use Proposition \ref{p.convergence_of_LQG} to argue the existence of an appropriate coupling.
We can guarantee that the discrete and continuum networks $N^\eta_{Q,t}$ and $N^\infty_{Q,t}$ are the same for each macroscopic quad $Q$ and each $t\leq T$ when $\eta$ is sufficiently small. To deduce that $\omega_\eta^{\eps,\gamma}(\cdot)$ and $\omega_\infty^{\eps,\gamma}(\cdot)$ are close for the Skorokhod topology one can introduce a so-called uniform structure (see \cite[Section 3]{gps-near-crit}).

The following theorem is proved exactly as \cite[Theorems 7.3 and 7.10]{gps-near-crit} based on the outline we gave above.
\begin{theorem}\label{thm2}
	Consider the setting above. 
	\begin{itemize}
		\item One can define a c\`adl\`ag process $( \omega^{\eps, \gamma}_\infty(t))_{t\in\R_+}$ with values in the quad-crossing space $\scr H$, which starts from uniform site percolation $\omega^\ga_\infty(0)$, and which is determined from $\omega^\ga_\infty(0)$ and $\{(x_i,t_i,\xi_i)\,:\,i\in\N \}$ in the exact same way as in \cite[Theorem 7.3]{gps-near-crit}.
		\item As $\eta \to 0$, the process $\omega^{\eps, \gamma}_\eta(t)$ converges in law in $(\Sk,d_{\Sk})$ to the process $(\omega^{\eps, \gamma}_\infty(t))_{t\in\R_+}$. Furthermore, the convergence is a.s.\ if the coupling is the one described above.
	\end{itemize}
\end{theorem} 

The following gives some further properties of the coupling described above.
Property (i) follows from Lemma \ref{prop:ss-meas} and from the fact that any fixed $q$ is a.s.\ a point of continuity for the limiting process (see \cite[Proposition 9.6]{gps-near-crit}). Property (v) follows since we condition on a particular instance of the field $h$ throughout the argument. The other properties follow from the analogous properties of the coupling in \cite{gps-near-crit}.
	\begin{itemize}
		\item[(i)] $\omega^\ga_\eta(q)\to \omega^\ga_\infty(q)$ a.s.\ for all $q\in\Q_+$ for $d^{\op{mod}}_{\HH}$, where $\Q_+:=\Q\cap\R_+$.
		\item[(ii)] For any dyadic $\epsilon>0$, the measures $\lambda^\epsilon_\eta$ and $\mu_{\gamma h}^{\lambda^\eps_\eta}$ converge to measures $\lambda^\epsilon$ and $\mu_{\gamma h}^{\lambda^\eps}$, respectively, supported on points that are $\epsilon$-important for $\omega^\ga_\infty(0)$.
		\item[(iii)] For any $T\in \N$ and any dyadic $\epsilon>0$, the finite set $\{(x_i^{\eps,\eta},t^{\eps,\eta}_i,\xi^{\eps,\eta}_i)\,:\,i\in\N,t_i\leq T \}$ converges a.s.\ to the finite set  $\{(x^\eps_i,t^\eps_i,\xi^\eps_i)\,:\,i\in\N,t_i\leq T \}$.
		\item[(iv)] The measures $\lambda^\eps$ and $\mu_{\gamma h}^{\lambda^\eps}$ have a.s.\ weak limits $\lambda$ and $\mu_{\gamma h}^\lambda$, respectively. The sets $\{(x^\eps_i,t^\eps_i,\xi^\eps_i)\,:\,i\in\N,t_i\leq T \}$ are increasing as $\eps\rta 0$ and have a limit $\{(x_i,t_i,\xi_i)\,:\,i\in\N,t_i\leq T \}$, which has the law of a PPP on $D\times\R_+\times\{-1,1 \}$ with intensity $\mu^{\lambda}_{\gamma h}\times\op{Leb}\times\op{Uniform}$. Similarly, the sets $\{(x^{\eps,\eta}_i,t^{\eps,\eta}_i,\xi^{\eps,\eta}_i)\,:\,i\in\N,t_i\leq T \}$ have a limit $\{(x_i^\eta,t_i^\eta,\xi_i^\eta)\,:\,i\in\N,t_i\leq T \}$, which is the point process defining the dLDP $\omega_{\gamma h}^\eta(\cdot)$.
		\item[(v)] The field $h$ is the same for all $\eta$.
	\end{itemize}

\subsection{Stability of LDP}
In this section we will prove Theorem \ref{thm1}(i), i.e., we will prove that for $\gamma \in(0,2-\sqrt{5/2})$ we have convergence of the dLDP $\omega^\ga_\eta(\cdot)$ in the appropriate spaces as $\eta \to \infty$. The whole proof works quenched in $h$, i.e., we prove the result for almost any instance of $h$. 

The main result of this section is the following proposition. Combined with Theorem \ref{thm2} and proceeding as in \cite[Section 9]{gps-near-crit}, it implies Theorem \ref{thm1} when $\gamma\in(0,2-\sqrt{5/2})$.
\begin{proposition}\label{pr.stab}
	Let $T>0$, $\gamma \in(0,2-\sqrt{5/2})$ and some instance $h$ of the log-correlated field be fixed. There exists a continuous function $\psi =\psi_{T,h} : [0,1]\to [0,1]$ with $\psi(0)=0$ such that uniformly in $0<\eta<\eps$,
	\[
	\Eb{ d_{\Sk_T}( \omega^\ga_\eta(\cdot),  \omega_\eta^{\eps, \gamma}(\cdot)) }\leq  \psi(\eps)\,.
	\] 
\end{proposition}

The proof proceeds similarly as the Euclidean version in \cite[Section 8]{gps-near-crit}. We will therefore omit many details in the proof, and point out only the places at which our argument differs from the one in \cite{gps-near-crit}. The main new input is Lemma \ref{prop7}.

To prove the proposition, we will need to introduce some notations as well as some preliminary lemmas. Since the entire section is about discrete configurations $\omega_\eta\in \HH$, we will often omit the subscript $\eta$ and denote the percolation configurations simply by $\omega$. We let $X=X_{\eta,T}$ denote the random set of sites of $\Tg_{\eta}$ which are updated along the dynamics $t\mapsto \omega_{\eta}(t)=\omega(t)$ for $t\in[0,T]$. Furthermore, we let $\Omega(\omega,X)$ denote the set of percolation configurations $\omega'$ such that $\omega'_x=\omega_x$ for all $x\notin X$, where $\omega_x\in\{-1,1 \}$ represents the color at site $x$. Finally, let $\mcl A_4(z,r,r')$ denote the $4$-arm event in the topological annulus $A(z,r,r') := B^{\op b}_{r'}(z)\backslash B^{\op b}_r(z)$, where $B^{\op b}_r(z)$ is the square of side length $2r$ centered at $z$.

\begin{lemma}\label{l.main}
Let $T>0$ and the instance of $h$ be fixed. 
Set $r_i:= 2^i\,\eta$, $N:= \lfloor\log_2(1/\eta)\rfloor$.
Let $\ev W_z(i,j)$ denote the event that
there is some $ \omega'\in\Omega( \omega,X)$ satisfying $\ev A_4(z,r_i,r_j)$.
Then for every pair of integers $i,j$ satisfying $0\le i<j<N$
and every $z\in\R^2$,
\begin{equation}\label{e.dinduct}
\P[\ev W_z(i,j)\mid h]
\le C_1\,\alpha^\eta_4(r_i,r_j)\,,
\end{equation}
where $C_1=C_1(T,h)$ is a constant that may depend only on
$T$ and $h$.
\end{lemma}

\begin{proof}
	The proof proceeds exactly in the Euclidean case, except that we use the new definition of $\P[x\in X\mid h]$. The reader is advised to also read the proof of \cite[Lemma 8.4]{gps-near-crit}, since many steps are skipped here. 
	
	Define $A_n=A(z,r_n,r_{n+1})$. Note that conditioned on $h$ the events $\{x_1\in X\}$ and $\{x_2\in X \}$ are independent, similarly as in the Euclidean case. In particular, defining $b_i^j:=\sup_z \Pb{\ev W_z(i,j)}$, we get as before
	$$
	\Pb{\ev W_z(i,j),\,\ev D}
	\le
	O(T)
	\sum_{n=i+1}^{j-2}
	\mu_{\gamma h}(A_n\cap \M_C)\,\alpha_4^\eta(\eta,1)^{-1}\,b_1^{n-1}\,b_i^{n-1}\,b_{n+2}^j,\,
	$$
	and further for some absolute constant $C_2$ and all $i,j$ with $j>i$,
	\begin{equation}
	\label{e.near}
	b_i^j\le C_2\,\alpha^\eta_4(r_i,r_j)\Bigl(1+T\,
	C_1^3
	\sum_{n=i+1}^{j-1}\,\frac{\mu_{\gamma h}(A_n)}{\alpha^\eta_4(r_n,1)}
	\Bigr)\,.
	\end{equation}
	Note that the latter bound is our variant of \cite[equation (8.3)]{gps-near-crit}. As in \cite{gps-near-crit} we show \eqref{e.dinduct} by induction on $j$, and for a fixed $j$ by induction on $j-i$. By Lemma \ref{prop7}, there exists $\varrho>0$ and $C(h)>0$ such that $\mu_{\gamma h}(A_n)<C\alpha^\eta_4(r_n,1)r_n^{\varrho}$. This and Lemma \ref{l.bound measure MC} imply that we can find a constant $M=M(T)\in\N$ such that for $N-j\geq M$, 
	\begin{equation}\label{e.near2}
T\,
(2\,C_2)^3
\sum_{n=i+1}^{j-1}\,\frac{\mu_{\gamma h}(A_n)}{\alpha^\eta_4(r_n,1)} \le 1\,.
	\end{equation}
	Choosing $C_1=2C_2$ and insert into \eqref{e.near} complete the proof by induction as in \cite{gps-near-crit}.
\end{proof}

For a site $z$ and a percolation $\wt\omega$ we will now define a quantity $Z(z)=Z_{\wt\omega}(z)$ which is closely related to the importance (Definition \ref{d. important}) of $z$. Let $Z(z)=Z_{\wt\omega}(z)$ denote the maximal radius $r$ such that the four arm event holds from the hexagon of $z$ to distance $r$ away. This is also the maximum $r$ for which changing the value of $\wt\omega(z)$ will change the white connectivity in $\wt \omega$ between two white points at distance $r$ away from $z$, or will change the black connectivity between two black points at distance $r$ away from $z$. Then set
$$
Z^X(z):=\sup_{\omega\in\Omega(\wt \omega,X)}Z_{ \omega'}(z),\, \ \ \  Z_X(z):=\inf_{\omega\in\Omega(\wt \omega,X)}Z_{ \omega'}(z) .
$$
In the same context as the lemma before we have the following result.
\begin{lemma}\label{l.change}
	For every site $z$ and every $\eps$ and $r$ satisfying
	$2\,\eta<\eps<2^4\,\eps<r\le 1$, we have 
	\[
	\Pb{Z^X(z)\ge r,\, Z_{ \omega}(z)\le \eps\mid h}\le O_{T,h}(1)\alpha^\eta_4(\eta,r) \ep^\delta.\]
\end{lemma}

\begin{proof}
	Assume $\omega'$ is such that $Z_{\omega'}(z)\ge r$, and let $x_1,\dots,x_m$ be some enumeration of the sites in $B^{\op b}_\ep(z)$ where $\omega'$ and $\omega$ are different. For each $j=0,1,\dots,m$, let $ \omega_j$ denote the configuration that agrees
	with $ \omega'$ on every site different from $x_{j+1},x_{j+2},\dots,x_m$,
	and agrees with $\wt \omega$ on $x_{j+1},\dots,x_m$.
	Then $ \omega_m= \omega'$ and $Z_{\wt \omega_0}(z)<\eps$.
	Let $k_{ \omega'}$ be the first $j$ such that $Z_{ \omega_j}(z)>r$. 
	
	Let $\wh X$ be the set of sites $x\in B^{\op b}_\ep(z)$ such that $x=x_{k_{\omega'}}$ for some $\omega'$ satisfying $Z_{\omega'}(z)\ge r$. Proceeding as in the proof of \cite[Lemma 8.5]{gps-near-crit} we see that
	\begin{equation}\label{e.near3}
\Pb{ Z^X(z)\ge r,\, Z(z)\le \eps,\,x\in\wh X\mid h}
\le
O_{T}(1)\,\alpha^\eta_4(\eta,r^x)\,\mu_{\gamma h}(B^{\op b}_\eta(x))\,\alpha^\eta_4(r,1)^{-1}\,.
	\end{equation}
	Since $\wh X$ is non-empty if $Z^X(z)\ge r$ and $Z(z)\le \eps$ both occur,
	\eqbn
	\begin{split}
		\Pb{ Z^X(z)\ge r,\, Z(z)\le \eps\mid h}
		&\le \sum_{x\in B^{\op b}_\ep(z)} \Pb{ Z^X(z)\ge r,\, Z(z)\le \eps,\,x\in\wh X\mid h}\\
		&\le O_{T}(1)\sum_{n=0}^{\log_2(\ep/\eta)} \alpha^\eta_4(\eta,r_n) \mu_{\gamma h}(A_n)\alpha^\eta_4(r,1)^{-1}\,.
	\end{split}
	\eqen
	The lemma now follows from Lemmas \ref{prop7} and \ref{l.bound measure MC} along with quasi-multiplicativity, i.e., $\alpha^\eta_4(\eta,1)\alpha^\eta_4(r,1)^{-1}\asymp \alpha^\eta_4(\eta,r)$.
\end{proof}
Next we state a similar result to Lemma \ref{l.change} which will be needed in a later work of the second and fourth coauthors \cite{hs-cardy}. Set
$$
Z_X(z):=\inf_{\wt \omega'\in\Omega(\wt \omega,X)}Z_{\wt \omega'}(z)\,.
$$
The proof of the following lemma is omitted, since it is identical to the proof of Lemma \ref{l.change}.
\begin{lemma}\label{l.change2}
	For every site $z$ and every $\eps$ and $r$ satisfying
	$2\,\eta<\eps<2^4\,\eps<r\le 1$, we have 
	$$
	\Pb{Z_X(z)\le \ep,\, Z_{ \omega}(z)\ge r}\le O_{T,h}(1)\alpha^\eta_4(\eta,r) \ep^\delta.
	$$
\end{lemma}

For any quad $Q\in \QQ^k$, if $r>0$ is smaller than the minimal distance from $\p_1 Q$ to $\p_3 Q$,
we will say that $Q$  is {\bf $r$-almost crossed} by $\omega=\omega_\eta\in\HH$ if there is an
open path in the $r$-neighborhood of $Q$ that comes within distance
$r$ of each of the two arcs $\p_1 Q$ and $\p_3 Q$.

The following lemma and proposition are proved exactly as \cite[Proposition 8.6 and Lemma 8.7]{gps-near-crit}, and the proofs are therefore omitted. Note that there is an exponent $\delta$ in the statement of these results, while the corresponding results in the Euclidean case have an explicit exponents depending on the four-arm exponent. Proposition \ref{pr.stab} follows from Lemma \ref{prop8} exactly as in \cite{gps-near-crit}. In particular, we observe from the proof that to deduce Proposition \ref{pr.stab} from the lemma it is sufficient with a variant of the lemma for which the considered probability converges 0 as $\ep\rta 0$. 

\begin{proposition}\label{p.almostCross}
	Let $T$, $h$, and $X$ be as above, and fix some quad $Q\in \QQ^\infty$. Let $r>0$ be smaller than the minimal distance between $\p_1Q$ and $\p_3 Q$, and suppose that $0<\eta<2\,\eta<\eps<2^5\,\eps<r\le 1$.
	Then the probability that there are some
	$\omega',\omega''\in\Omega(\omega,X)$ such that
	\begin{enumerate}[(a)]
		\item $Q$ is crossed by $\omega'$,
		\item $Q$ is not $r$-almost crossed by $\omega''$, and
		\item $\omega '(z)=\omega''(z)$ for every site $z$ satisfying $Z_{\omega}(z)\ge \eps$,
	\end{enumerate}  
	is at most 
	$$
	O_{T,Q,h}(\eps^\delta)\,\alpha^\eta_4(r,1)^{-1}.
	$$ 
\end{proposition}

See \cite[Definition 3.3]{gps-pivotal} for the notation $\Os_k(\cdot)$ used in the following lemma. Intuitively, for $k\in\N$ and $\omega\in\HH$, $\Os_k(\omega)$ denotes the set of percolation configurations which have the same crossing properties as $\omega$ for all quads in $\QQ^k$, possibly with some small deformations of size $2^{-k}$.
\begin{lemma}\label{prop8}
	Let $k\in \N$ and $T>0$ be fixed and suppose that $0<\eta<2\eta< \eps<2^{-k-20}$. Then the probability that there are
	$\omega',\omega''\in\Omega(\omega,X)$ such that
	\bi
	\item[(a)] $\omega'\notin \Os_k(\omega'')$,
	\item[(b)] $\omega''\notin \Os_k(\omega')$, and
	\item[(c)] $\omega'(z)=\omega''(z)$ for every site $z$ satisfying $Z_{\omega}(z)\ge \eps$,
	\ei
	is at most 
	$$
	O_{T,k,h}(\eps^\delta).
	$$
\end{lemma}

\section{Law of the spectral measure via inhomogeneous dynamical percolations}
\label{sec:spectralmeasure}
The spectral sample plays an important role  in our proofs, due to the identity \eqref{e. cov intro}. In \cite{gps-fourier} it was proved that the spectral sample associated with the crossing of a single quad $Q$ converges for the Hausdorff metric as a set in the complex plane to a conformally invariant limit with dimension $3/4$. Since the size of the spectral sample has the same first moment as the number of pivotal points (see e.g.\ \cite{kkl-boolean}), we also have tightness of the spectral measure. Non-triviality of subsequential limits follows from the bound on the lower tails of the spectral sample. However, it was not proved in \cite{gps-fourier} that the spectral measure converges in law, i.e., the uniqueness of the limit was not proved. We establish this in the following proposition.
\begin{proposition}\label{pr.convS}
	Consider quads $Q_1,\dots,Q_k$, let $f:\HH\to\{0,1 \}$ be the function which says whether all the quads are crossed, and let $\S_\eta$ be the corresponding spectral measure for percolation on the lattice $\Tg_\eta$. Then the measure $\S_\eta$ converges in law in the scaling limit for the weak topology to a measure $\S$. 
	\label{prop:spectralsample}
\end{proposition}

\begin{remark}
In \cite{gps-fourier}, the third open problem asked about the convergence in law of the coupling $(\S_\eta, \lambda_\eta)$, where $\S_\eta$ denoted the spectral sample viewed as a random set while $\lambda_\eta$ denoted the rescaled counting measure on $\S_\eta$. The convergence in law of the random set $\S_\eta$ (for the Hausdorff topology) was proved in \cite{gps-fourier} but not for the counting measure $\lambda_\eta$. In the present paper, $\S_\eta$ in fact denotes the measure $\lambda_\eta$  from \cite{gps-fourier} (this slight abuse of notation was motivated by the fact that $\lambda_\eta$ is already used to denote the weighted counting measure on pivotal points). We thus make significant progress on this open problem by showing the weak convergence of the second coordinate $\lambda_\eta$. It remains to prove the convergence of the joint coupling to fully answer the question raised in \cite{gps-fourier}.
\end{remark}

We will prove the proposition by giving a formula for $\E[\exp(-\phi\S)]$ for all bounded continuous functions $\phi:D\to\R_+$, where $\phi\S$ is the measure assigning mass $\int_U\phi\,d\S$ to any measurable set $U\subseteq D$. This is sufficient to characterize $\S$ due to the following result, which can be found in e.g.\ \cite[Corollary 2.3]{kallenberg17}.
\begin{theorem}[\cite{kallenberg17}]
	Let $S$ be a Polish space, and let $\xi$ and $\eta$ be random Borel measures on $S$. Then $\xi\eqD\eta$ if and only if $\E[e^{-\phi\xi}]=\E[e^{-\phi\eta}]$ for all bounded continuous functions $\phi:S\to\R_+$ with bounded support.
	\label{thm:kallenberg}
\end{theorem}
\begin{proof}[Proof of Proposition \ref{prop:spectralsample}]
	Let $\omega^\phi_\eta(\cdot)$ denote the dynamical percolation on $\Tg_\eta$ driven by the measure which has density $\phi$ relative to Lebesgue area measure
	(see Definition \ref{d.dp}). 
	By Lemma \ref{l.covariance},
	\eqb
	\Cov\left[f(\omega^\phi_\eta(0)),f(\omega^\phi_\eta(1))\right]=
	\E\left[ e^{-(\phi\S_\eta)(D)} \1_{\S_\eta(D)\neq 0}\right]\,.
	\label{eq3}
	\eqe
	Proceeding just as in Section \ref{s.3}, $\omega^\phi_\eta(\cdot)$ converges in law in Skorokhod space to a process $\omega^\phi_\infty(\cdot)$, and we can find a coupling of $\omega^\phi_\eta(\cdot)$ and $\omega^\phi_\infty(\cdot)$ satisfying the properties (i)-(iv) given below Theorem \ref{thm2}. In particular, (i) implies that 
	the left side of \eqref{eq3} converges to $\Cov\left[f(\omega^\phi_\infty(0)),f(\omega^\phi_\infty(1))\right]$ as $\eta\rta 0$.
	
	Since $\limsup_{\eta\to 0}\E[\S_\eta(D)]<\infty$ (see for example the proof of Theorem 10.4 in \cite{gps-fourier}), the measures $\S_\eta$ converge subsequentially in law as $\eta\rta 0$. Let $\S$ denote a subsequential limit. The lower bound in Theorem \ref{t. spectral tightness} gives that $(\S_\eta,\1_{\S_\eta(D)\neq 0})$ converges in law along the considered subsequence to $(\S,\1_{\S(D)\neq 0})$. By the bounded convergence theorem, the right side of \eqref{eq3} converges to $\sup_{\eta\in(0,1]}\E\left[ e^{-(\phi\S)(D)} \1_{\S(D)\neq 0}\right]$, so
	\eqbn
	\Cov\left[f(\omega^\phi_\infty(0)),f(\omega^\phi_\infty(1))\right]=
	\E\left[ e^{-(\phi\S)(D)} \1_{\S(D)\neq 0}\right]\,.
	\eqen
	By Theorem \ref{thm:kallenberg}, this is sufficient to uniquely identify the law of $\S$. 
\end{proof}
An alternative proof of Proposition \ref{prop:spectralsample} can be obtained by using \cite[Theorem 4.3]{kallenberg97}, which says that a random vector $X$ with values in $\R_+^d$ for some $d\in\N$ is uniquely characterized by $\E[\exp(-X\cdot s)]$ for all $s\in\R_+^d$. This result can be used to find the joint law of $\S(U)$ for all $U$ in an arbitrary finite set, which is sufficient to identify the law of $\S$.

\section{Convergence of the LDP: indirect microscopic stability}\label{s.ConvLDP}
In this section we prove the convergence of the dLDP in the case when $\gamma\in[2-\sqrt{5/2},\sqrt{3/2})$. It is tempting to use the same proof as in the case $\gamma\in(0,2-\sqrt{5/2})$. The main problem is that, in the notation of Lemma {l.main}, a.s.\ there exist points $z$ such that the associated annuli $A_n$ satisfy $\sum_{n} \mu_{\gamma h} (A_n)\alpha^\eta_4(r_n,1)^{-1} =\infty$. This implies that Proposition \ref{pr.stab} cannot be proved using the same method. We circumvent this problem using the modified LDP. 

Recall the definition of $\wt \omega_\eta^{C,\gamma}(\cdot)$ as the dynamical percolation (Definition \ref{d.dp}) driven by the measure $\wt \mu_{\gamma h}^{C}(\cdot):=\mu_{\gamma h}(\cdot \cap \M_C)$, where $\M_{C}$ is given by \eqref{e.MC intro}.
This additional cut-off is useful, as Lemma \ref{l.bound measure MC} implies that $\sum_{n} \mu_{\gamma h} (A_n\cap \M_C)\alpha^\eta_4(r_n,1)^{-1} <\infty$ a.s.

The section is organized as follows. We start by observing that the modified LDP converges in the Skorokhod topology for any fixed $C>0$ to a limit $\wt \omega_\infty^{C,\gamma}(\cdot)$. Then, in order to show that 
$\wt \omega_\infty^{C,\gamma}(\cdot)$ converges when $C\rta\infty$ to a limiting process $\omega_\infty^{\gamma}(\cdot)$, we control via a coupling argument the amount of times a given quad changes from being crossed to not being crossed and vice versa. Once the convergence as $C\to \infty$ is achieved, we show how this implies that the finite-dimensional laws of $\omega^\ga_\eta(\cdot)$ converge. To show that we have stronger convergence, we introduce a topology on c\`adl\`ag processes in the quad-crossing space (namely, the $L^1$ topology in Theorem \ref{thm1} and Definition \ref{def:L1}) for which $\omega^\ga_\eta(\cdot)$ converges. We finish by showing that conditionally on $h$, $\omega^\ga_\infty(\cdot)$ is a Markov process.

\subsection{Convergence of the modified LDP}
The following proposition follows from the same proof as the convergence of the LDP in the case $\gamma\in(0,2-\sqrt{5/2})$.
\begin{prop}\label{p.Skorokhod mLDP}
For any  $0\leq \gamma < \sqrt{3/2}$, the dynamics $\wt \omega_\eta^{C,\gamma}(\cdot)$ converge in law to a limiting dynamic $\wt \omega_\infty^{C,\gamma}(\cdot)$ in the Skorokhod topology.
\end{prop}

\begin{proof}
The proof is the same as in the case $\gamma\in(0,2-\sqrt{5/2})$. There are essentially two differences. The first is that to prove the equivalent of Theorem \ref{thm2}, we need to rely on Proposition \ref{p.convergence_of_modified_LQG} instead of Proposition \ref{p.convergence_of_LQG}.

The second difference is that one needs to change $\mu_{\gamma h}$ to $\wt \mu_{\gamma h}^{C}$ in \eqref{e.near}, \eqref{e.near2}, and \eqref{e.near3}, and then use  Lemma \ref{l.bound measure MC} to conclude the equivalent result of \eqref{e.near2}.
\end{proof}

 For $C\in\N\cup\{\infty \}$ and $\eps\in[0,1]$ dyadic we let $\{(x^{C,\eps}_i, t^{C,\eps}_i,\xi^{C,\eps}_i)\,:\,i\in\N \}$ denote the PPP on $D\times\R_+\times\{-1,1 \}$ with intensity $\mu_{\alpha h}^{\lambda^\eps}(B^{\op h}_\eta(\cdot) \cap \M_C)\times \text{Leb}\times\op{Uniform}$ which is driving the process $\wt \omega^{C,\eps,\gamma}_\infty(\cdot)$, while $\{(x^{C,\eps,\eta}_i, t^{C,\eps,\eta}_i,\xi^{C,\eps,\eta}_i)\,:\,i\in\N \}$ is the PPP with intensity $\mu_{\alpha h}^{\lambda^\eps_\eta}(B^{\op h}_\eta(\cdot) \cap \M_C)\times \text{Leb}\times\op{Uniform}$ which is driving the process $\wt \omega^{C,\eps,\gamma}_\infty(\cdot)$. Note that when $\eps=0$ we may skip the superscripts $\eps$, while when $C=\infty$ and $\eps=0$ we may skip both these superscripts. 
\begin{remark}
One can describe a coupling where the convergence of these processes and the processes in Proposition \ref{p.Skorokhod mLDP} is a.s. The coupling is obtained similarly to the one described in Section \ref{s.cut-offeps} and by coupling the PPP for different $C$ in the natural way.
\begin{itemize}
	\item[(i)] For any dyadic $\epsilon\in(0,1]$, the measures $\lambda^\epsilon_\eta$ and $\mu^{\lambda^\eps_\eta}_{\gamma h}$ converge to measures $\lambda^\epsilon$ and $\mu^{\lambda^\eps}_{\gamma h}$, respectively, supported on points that are $\epsilon$-important for $\omega^\ga_\infty(0)$;
	\item[(ii)] for any $T,C\in \N$ and dyadic $\eps\in(0,1]$, the finite set $\{(x^{C,\eps,\eta}_i,t^{C,\eps,\eta}_i,\xi^{C,\eps,\eta}_i)\,:\,i\in\N \}$ converges a.s.\ to a finite set $\{(x_i^{C,\eps},t_i^{C,\eps},\xi_i^{C,\eps})\,:\,i\in\N \}$; 
	\item[(iii)] the sets $\{(x^{C,\eps,\eta}_i,t^{C,\eps,\eta}_i,\xi^{C,\eps,\eta}_i)\,:\,i\in\N \}$ and $\{(x_i^{C,\eps},t_i^{C,\eps},\xi_i^{C,\eps})\,:\,i\in\N \}$ are increasing as $C$ increases and $\eps$ decreases;
	\item[(iv)] for any $C\in\N$ we have  $\wt \omega_\infty^{C,\eps,\gamma}(\cdot) \to \wt \omega_\infty^{C,\gamma}(\cdot)$ in the Skorokhod topology as $\eps\rta 0$;
	\item[(v)] for any $C\in\N$ we have  $\wt \omega_\eta^{C,\gamma}(\cdot) \to \wt \omega_\infty^{C,\gamma}(\cdot)$ in the Skorokhod topology as $\eta\rta 0$;
	\item[(vi)] for any $q\in \Q_+$ (recall that $\Q_+=\Q\cap\R_+$) and $C\in\N$, we have $\wt \omega^{C,\gamma}_\eta(q)\to \wt \omega^{C,\gamma}_\infty(q)$ a.s.\ for $d_\HH^{\op{mod}}$;
	\item[(vii)] the field $h$ is the same for all $\eta$, $\eps$, and $C$.
\end{itemize}
\label{r.coupling}
\end{remark}

\subsection{Control of the jumps for LDP}
Recall that $\QQ^\infty$ is the set of all quads whose boundaries are contained in $(2^{-k}\Z^2)\cap D$ (see above \eqref{e.distance}) for some $k\in \N$. We are interested in the study of the jumps of $\omega_\eta^\ga(\cdot)$.
\begin{prop}\label{p.tightness}
	Let $0\leq T_1<T_2$ and $\Quad \in \QQ^\infty$, and let $\omega_{\eta}^{\sigma}(\cdot)$ be a dynamical percolation driven by a deterministic measure $\sigma$. Then 
	\begin{equation}\label{e.expected value crossings}
	\begin{split}
	&\Eb{\#\{t\in [T_1,T_2]: \Quad \in \omega^\sigma_\eta(t^-)\Delta \omega^\sigma_\eta(t^+) \}}\\
	&\qquad=
	(T_2-T_1)\sum_{x\in \Tg_\eta}\sigma(B^{\op h}_\eta(x))\alpha^\eta_4(\eta,1)^{-1}\P[x \text{ is pivotal for the crossing of } \Quad]\,.
	\end{split}
	\end{equation} 
	Furthermore, in the case of Liouville dynamical percolation for any $\gamma\in(0,2)$, for all $\Quad\in \QQ^\infty$ there a.s.\ exists a random constant $C(\Quad)<\infty$ such that
	\begin{equation}\label{e.expected value crossings 2}
	\sup_{\eta\in(0,1]}\E\left[\#\{t\in[T_1,T_2]\,:\, \Quad \in \omega^\gamma_\eta(t^-)\Delta \omega^\gamma_\eta(t^+) \}\mid h\right] \leq (T_2-T_1)C(\Quad)\,.
	\end{equation}
\end{prop}
\begin{proof}

The first claim is classical and a direct consequence of the invariance of the Bernoulli product measure under the dynamics $t \mapsto \omega_\eta^\sigma(t)$. The identity \eqref{e.expected value crossings} follows from the observation that each time a site $x$ flips according to a Poisson clock of rate $\sigma(B^{\op h}_\eta(x))\alpha^\eta_4(\eta,1)^{-1}$, it sees the invariant uniform measure and the flip adds a switching time if and only if $x$ is a pivotal point of $Q$ at that moment. 

The second claim follows from the first claim, $\E[\sigma(B^{\op h}_\eta(x))]=O(\eta^{2})$, and analysis of the percolation arm exponents. For the latter property, note that for any quad $\Quad$, we have uniformly in $x\in \Quad$ bounded away from the quad boundary and $\eta>0$, 
\eqb
\alpha^\eta_4(\eta,1)^{-1} \P[x \text{ is pivotal for the crossing of } \Quad] \leq  C_{\Quad}.
\label{eq14}
\eqe
To control the boundary effect, we consider three-arm and two-arm exponents near boundaries and corners of $\Quad$. See for example \cite{gps-fourier} or \cite{gs-book} for such analysis of boundaries and 90 degree corners. For the 270 degree corners one can proceed similarly: By conformal invariance of SLE$_6$ one can argue that the arm exponents for a 270 degree corner are $2/3$ times those for a half-plane, and then one argues that the probability of having a pivotal point at a certain location is bounded above by the probability of crossing certain annuli. 

We will provide some more details to the argument in the previous paragraph by explaining how to bound the probability that a site $v$ at $(r,r)$ for $r\ll 1$ is pivotal, assuming the first, second, and fourth quadrants restricted to (say) $[0,1]^2$ are contained in the quad, while the third quadrant does not intersect the quad. Furthermore, viewing $Q$ as a function  $Q:[0,1]^2\to\C$, assume that $Q((0,0))=0$. In order for $v$ to be pivotal there must be a four-arm event from $v$ to distance $r$, and there must be a two-arm event from distance $r$ to distance $1$; note that we have a two-arm event instead of a four-arm event in the second statement since two of the arms from $v$ to the quad boundaries may have arm length only of order $r$. The two events just described have probability $\alpha_4^\eta(\eta,r)$ and $\wt\alpha_2^\eta(\eta,r)$, respectively, up to a multiplicative constant, where $\wt\alpha_2^\eta(\cdot,\cdot)$ is the probability of a multi-color two-arm event for a 270 degree corner. The exponent for the two-arm event in the half-plane is $1$ \cite{smirnov-werner-percolation}, so $\wt\alpha_2^\eta(\eta,r)=(\eta/r)^{2/3+o(1)}$. It follows that the probability of $v$ being pivotal is $O( (\eta/r)^{5/4+o(1)}r^{2/3+o(1)})=O(\eta^{5/4+o(1)}r^{-7/12+o(1)})$.\footnote{This statement is still true with $\Theta(\cdot)$ instead of $O(\cdot)$, so \eqref{eq14} does \emph{not} hold uniformly over all $x\in Q$ since $r^{-7/12}\rta\infty$ as $r\rta 0$.} We do a similar estimate for other points near the origin. Note that for certain points (e.g.\ points of the form $(-r_1,r_2)$ with $0\ll r_2\ll r_1$) we may need to consider a three-arm event in addition to the two-arm event and the four-arm event. Combining the estimates, we get that the sum on the right side of \eqref{e.expected value crossings} is finite if we restrict to points in (say) $[-0.5,0.5]^2$.
\end{proof}

\subsection{Modified LDP is close to dLDP}
Let us start by noting that $\wt \mu_{\gamma h}^{C}(\cdot)$ is increasing in $C$. Thus, for fixed $\eta>0$ one can couple $(\wt \omega_\eta^{C,\gamma}(\cdot))_{1\leq C\leq \infty}$ in a natural way (note that $\wt\omega_{\eta}^{\infty,\gamma}(\cdot)=\omega_{\eta}^\gamma(\cdot)$). In this coupling, for a fixed $t>0$ and $C<C'$, the times before $t$ where the clock of $x\in \Tg_\eta$ rings for $\wt \omega_\eta^{C',\gamma}(\cdot)$ but not for $\wt \omega_\eta^{C,\gamma}(\cdot)$ follows a law of a PPP with rate $\mu_{\gamma h}(B^{\op h}_\eta(x)\cap(\M_{C'}\backslash \M_C))$. 
\begin{prop} \label{p.uniformicity C}
For any $\gamma\in(0,\sqrt{3/2})$, 
	take $\Quad\in \QQ^k$, $t>0$, and $C,C'\in [1,\infty]$ satisfying $C<C'$. Then there exists a function $\Upsilon_{t,Q}:[1,\infty]\to [0,1]$ such that $\lim_{C\to \infty} \Upsilon_{t,\Quad}(C)=\Upsilon_{t,\Quad}(\infty)=0$ and
	\begin{align*}
	\P\left(\Quad\in \wt \omega_\eta^{C,\gamma}(t)\Delta \wt \omega_\eta^{C',\gamma}(t)\right)\leq \Upsilon_{t,\Quad}(C)\,. 
	\end{align*}
\end{prop}
\begin{proof}
Fix $t>0$. As $C$ gets large, 
the natural idea here is to proceed as in Proposition \ref{p.tightness} by first sampling the configuration  $\wt \omega_\eta^{C,\gamma}(t)$ and then arguing that there are very few remaining possible switches (in expectation) because going from $\wt \omega_\eta^{C,\gamma}(t)$ to $ \wt \omega_\eta^{C',\gamma}(t)$ is governed by a dynamical percolation driven by a measure $\mu_{\gamma h}(\M_{C'}\backslash \M_C)$ of very small total mass as $C\to \infty$.

This intuition is mostly correct except that once $\wt \omega_\eta^{C,\gamma}(t)$ is sampled, it is not correct that one can obtain $\wt \omega_\eta^{C',\gamma}(t)$ by starting the dynamics at time $t=0$ from $\wt \omega_\eta^{C,\gamma}(t)$ and then updating sites {\em forward in time} along $u\in [0,t]$ according to the PPP driven by $\mu_{\gamma h}(\M_{C'}\backslash \M_C)$. One way to see what might go wrong is as follows: suppose you run forward in time from the initial configuration $\wt \omega_\eta^{C,\gamma}(t)$ and a site $x$ rings at time $u$. If that same site $x$ had been already updated when constructing $\wt \omega_\eta^{C,\gamma}(t)$ at some later time $v>u$, then in order to build $\wt \omega_\eta^{C',\gamma}(t)$ one should NOT update the site $x$ at time $u$. 

To deal with this issue, we shall condition both on the final configuration $\wt \omega_\eta^{C,\gamma}(t)$ as well as the PPP $P_{\eta,t}^{C,\gamma} \subset [0,t]^\Tg$ used to define $s \mapsto \wt \omega_\eta^{C,\gamma}(s)$. Given the pair $(\wt \omega_\eta^{C,\gamma}(t), P_{\eta,t}^{C,\gamma})$ we will sample the pair $(\wt \omega_\eta^{C,\gamma}(t), \wt \omega_\eta^{C',\gamma}(t))$ using a dynamical process $s \mapsto \wh \omega(s)$ for $s\in[0,t]$ constructed as follows. 

\begin{enumerate}
\item We initialize $\wh \omega(0):= \wt \omega_\eta^{C,\gamma}(t)$. 
\item As $s\in[0,t]$ increases, we update $\wh \omega(s)$ dynamically according to Poisson clocks governed by the measure $\mu_{\gamma h}(\M_{C'}\backslash \M_C)$ as follows: when a site $x$ rings at time $s$, this site is updated if and only if that same site $x$ had not been updated along the construction of $\wt \omega_\eta^{C,\gamma}(t)$ between times $s$ and $t$ (or in other words $\{(x,u), s\leq u \leq t \}  \cap P_{\eta,t}^{C,\gamma} = \emptyset$). 
\item We let $s\mapsto \ol \omega(s)$ denote the dynamical process which updates points without any restriction coming from the knowledge of $P_{\eta,t}^{C,\gamma}$.  As such, $\ol \omega(s)$ is updated more often than $\wh \omega(s)$ but has the advantage that given $\ol\omega(0)=\wh \omega(0)=\wt \omega_{\eta}^{C,\gamma}(t)$, it follows a regular dynamical percolation. 
\end{enumerate}

It is easy to check that $(\wh \omega(0), \wh \omega(t))$ has the desired law, namely $(\wh \omega(0), \wh \omega(t))\eqD(\wt \omega_\eta^{C,\gamma}(t), \wt \omega_\eta^{C',\gamma}(t))$. Note also that for all $s\in[0,t)$, the law of $(\wh \omega(s), \wh \omega(t))$ is NOT the same as $(\wt \omega_\eta^{C',\gamma}(s),\wt \omega_\eta^{C',\gamma}(t))$. Therefore our setup does not allow us to say anything about $\wt \omega_\eta^{C',\gamma}(s)$ for $s\in(0,t)$; in particular, we do not get convergence in Skorokhod space since this requires us to consider all $s\in[0,t]$.

A crucial point for what follows is for all $s\in[0,t]$, the marginal law of the configuration $\wh \omega(s)$ is i.i.d.\ percolation. This gives that for a constant $C>0$, 
\begin{align}\label{}
&\P\left(\Quad \in \wt \omega_\eta^{C,\gamma}(t)\Delta \wt \omega_\eta^{C',\gamma}(t)\right)  \leq \Pb{\#\{s\in [0,t]: \Quad \in \wh\omega(s^-)\Delta \wh\omega(s^+) \} \geq 1}  \nonumber \\
&\qquad \leq \Eb{\#\{s\in [0,t]: \Quad \in \wh\omega(s^-)\Delta \wh\omega(s^+) \}}   \leq \Eb{\#\{s\in [0,t]: \Quad \in \ol\omega(s^-)\Delta \ol\omega(s^+) \}}  \nonumber \\
&\qquad \leq C \mu_{\gamma h} (D\setminus \M_C),
\end{align}
where we proceed as in Proposition \ref{p.tightness} to justify the last inequality. Since the right side converges to zero a.s.\ as $C \to \infty$ by Lemma \ref{prop:thickpt}, this concludes our proof.
\end{proof}

Let us now see how Proposition \ref{p.uniformicity C} implies the fact that the finite-dimensional laws of $\omega^\gamma_\eta$ converge.

\begin{proposition}
	Let $\gamma\in(0,\sqrt{3/2})$. Then the finite-dimensional distribution of $\omega^\gamma_\eta(\cdot)$ converge in law in $(\HH,d_{\HH})$.
	\label{prop3} 
\end{proposition}
\begin{proof}
	We will only prove that $(\omega^\gamma_\eta(q))_{q\in\Q_+}$ converges, since the case of $\Q_+$ replaced by some other countable set can be treated in the same way. We will use the coupling of Remark \ref{r.coupling}. Let us first prove that for any fixed $q\in \Q_+$, the sequence $(\wt\omega_\infty^{C,\gamma}(q))_{C\in\N}$ is a Cauchy sequence as $C\rta\infty$. For $C<C'$,
	\begin{align*}
	\P\Big(d_\HH^{\op{mod}}(\wt\omega_\infty^{C,\gamma}(q),\wt\omega_\infty ^{C',\gamma}(q))<2^{-k}\Big)&\leq \sum_{\Quad\in \QQ^{k}} 	\P\left( \Quad\in \wt \omega_\infty^{C,\gamma}(q)\Delta \wt \omega_\infty^{C',\gamma}(q)\right)\\
	&\leq \lim_{\eta\to 0 } \sum_{\Quad\in \QQ^{k}} 	\P\left( \Quad\in \wt \omega_\eta^{C,\gamma}(q)\Delta \wt \omega_\eta^{C',\gamma}(q)\right)\\
	&\leq (\# \QQ_k) \sup_{\Quad \in \QQ_k}\Upsilon_{q,\Quad}(C)\to 0 \ \ \ \text{as } n,m\to \infty\,.
	\end{align*}
	In the second inequality we use (v) of Remark \ref{r.coupling},
	and in the third inequality we use Proposition \ref{p.uniformicity C}. 
	It follows that $(\wt\omega_\infty^{C,\gamma}(q))_{C\in\N}$ is a Cauchy sequence for $d_\HH^{\op{mod}}$, so it is a Cauchy sequence for $d_\HH$. By completeness of $(\HH,d_\HH)$, it has a limit in probability $\omega_\infty^{\gamma}(q)$ for the metric $d_{\HH}$.
	
	Let us continue using the same coupling, and show that the finite-dimensional distribution of $\omega_\eta^\gamma(\cdot)$ converge to those of $\omega_\infty^{\gamma}(\cdot)$. We will show convergence of $(\omega_\eta^\gamma(q))_{q\in\Q_+}$ converge to those of $(\omega_\infty^{\gamma}(q))_{q\in\Q_+}$; the case of $\Q_+$ replaced by some other countable set follows from the same argument. Note that for any $q\in \Q_+$, the triangle inequality gives 
	\begin{align}\label{e.3 terms mt}
	\begin{split}
	d_\HH(\wt\omega_\infty ^{\gamma}(q),\omega_\eta^{\gamma}(q))
	\leq&\, 
	d_\HH\left (
	\wt\omega_\infty ^{\gamma}(q), \wt\omega_\infty^{C,\gamma}(q)\right )
	+d_\HH\left (
	\wt\omega_\infty ^{C,\gamma}(q),\wt\omega_\eta^{C,\gamma}(q)\right )\\
	&+d_\HH\left (
	\wt\omega_\eta ^{C,\gamma}(q), \omega_\eta^{\gamma}(q)\right )\,.
	\end{split}
	\end{align}
	The first term does not depend on $\eta$ and converges to $0$ in probability as $C\to \infty$ by the previous paragraph. For any fixed $C\in \N$ the second term converges a.s.\ to $0$ as $\eta \to 0$ by (vi) of Remark \ref{r.coupling}. It remains to bound the third term. As before, we can bound
	\begin{align*}
	\P\left( d_\HH^{\op{mod}}\left (\omega_\eta^{\gamma}(q),\wt\omega_\eta ^{C,\gamma}(q)\right )\leq 2^{-k}\right)\leq (\# \QQ_k) \sup_{\Quad \in \QQ_k}\Upsilon_{q,\Quad}(C) \to 0 \ \ \ \text{as } C\to \infty\,.
	\end{align*}
	
	We conclude that for all $q\in \Q_+$ the following convergence holds in probability as $\eta\rta 0$ for $d_{\HH}$
	\[
	\omega_\eta ^\gamma(q)\to \omega_\infty^\gamma(q) \ \ \text{ as } \eta\to 0. 
	\]
	This implies that the finite-dimensional distribution of $\omega_\eta^\ga(\cdot)$ converge in probability (and therefore also in law) to those of $\omega_\infty^\gamma(\cdot)$. 
\end{proof}

\subsection{Liouville dynamical percolation is c\`adl\`ag}
Let us define 
\[\II:=\{f:[0,1)\to \{0,1\}: f\text{ c\`adl\`ag and } \lim_{t\uparrow 1} f(t) \text{ exists} 
\}.
\]
Let $\M([0,1))$ denote the set of all signed measures on $[0,1)$. Each function $f\in \II$ may be associated with a unique measure $\nu_f \in \M([0,1))$, namely the measure $\nu$ such that $\nu([0,t])=f(t)$. Note that this measure will have a point mass $\pm 1$ at each point where the value of $f$ changes from 0 to 1 or vice versa. This will allow us to work with the convergence of $\nu_{f_n}^+$ and $\nu_{f_n}^-$, the positive and negative part of the Jordan decomposition of $\nu_{f_n}$.

In the next lemma we consider a deterministic process $(\wh\varpi(t))_{t\in [0,1)}$ in $(\HH, d_\HH)$. We give a criterion which guarantees that $(\wh\varpi(t))_{t\in[0,1)}$ is c\`adl\`ag.
We identify $\HH$ with a subset of $\{0,1\}^{\QQ_\infty}$ by identifying $\omega\in\HH$ with $(f_\Quad(\omega))_{\Quad \in \QQ^\infty}$ for $f_\Quad(\omega)=\1_{\Quad \in \omega}$. This allows us to see a $\HH$-valued process $\omega(\cdot)$ as a sequence of functions $(f_\Quad(\omega(\cdot)))_{\Quad\in \QQ^\infty}$ from $[0,1)$ to $\{0,1\}$.
\begin{lemma} \label{l.cadlag}
	Let $(\varpi_n(t))_{t\in[0,1)}$ be a sequence of processes with value in $\HH$ such that for all $\Quad \in \QQ^\infty$, $f_\Quad(\varpi_n(\cdot))\in \II$. Define  $\nu^{\Quad}_n$ as the measure associated to $f_\Quad(\varpi_n(\cdot))$. Assume that for all $\Quad\in \QQ^\infty$, $(\nu^{\Quad}_n)^+$ and $(\nu^{\Quad}_n)^-$ converge weakly to certain measures $(\ol \nu^\Quad)^+$ and $(\ol \nu^\Quad)^-$. Furthermore, suppose there exists a dense set $S\subseteq (0,1)$ such that for all $t\in S$, $\varpi_n(t)$ converges to some $\wh \varpi(t)$ in $(\HH, d_\HH)$  and
	\[ S\cap \{t\in (0,1): t \text{ is in the support of $(\ol \nu^\Quad)^+$ or $(\ol \nu^\Quad)^-$ for some $\Quad \in \QQ^\infty$} \}=\emptyset.\]
	Then, $\wh \varpi(t)= \lim_{\substack{s\downarrow t \\ s\in S}} \varpi(s)$ is a c\`adl\`ag process in the space of Schramm-Smirnov and $\wh \varpi(s)= \varpi(s)$ for all $s\in S$.
\end{lemma}

\begin{proof}
Let us take $\Quad \in \QQ^\infty$ and note that for any $t\in (0,1)$ that is not in the support of $(\ol \nu^\Quad)^+$ or $(\ol \nu^\Quad)^-$, we have $f_\Quad(\varpi_n(t)) \to \ol \nu_\Quad([0,t])$. This implies that for all $s\in S$, $f_\Quad(\varpi(s)) = \ol \nu_\Quad([0,s])$. 

We show now that $\wh\varpi(t)$ exists and is right-continuous in $t$. To do that, it is enough to show that for any $t\in [0,1)$ and $k\in \N$ there exists a $\delta>0$ such that
\[
\sup\{d_\HH(\varpi(s_1),\varpi(s_2)): s_1,s_2\in [t,t+\delta)\cap S\}\leq 2^{-k}. 
\] 
This follows simply from the fact that $t\mapsto(\ol\nu^Q)^\pm([0,t))$ for $Q\in\QQ^\infty$ are c\`adl\`ag functions with finitely many jumps, which implies for all $k\in \N$, $\Quad \in \QQ^k$, and $t\in (0,1)$, there exists $\delta>0$ such that $f^\Quad(\wh \varpi(\cdot))$ is constant in $[t,t+\delta)\cap S$. 

To conclude we just need to prove that $\wh \varpi(\cdot)$ has a left limit. To do that,  it is enough to show that for any $t\in [0,1]$ and $k\in \N$ there exists a $\delta>0$ such that
\[\sup\{d_\HH(\varpi(s_1),\varpi(s_2)): s_1,s_2\in (t-\delta,t)\cap S\}\leq 2^{-k}.\]
This follows by a similar argument.
\end{proof}
\begin{remark}
	Note that the conditions of the lemma do not imply that $\varpi_n\rta\varpi$ in the Skorokhod topology. For example, assume $t\not\in S$, that $(t_n)_{n\in\N}$ and $(t'_n)_{n\in\N}$ are two sequences converging to $t$ such that $t_n<t'_n$ for all $n\in\N$, and that for two distinct quads $Q,Q'\in\QQ^\infty$ the measure $(\nu_n^Q)^+$ (resp.\ $(\nu_n^{Q'})^+$) has a point mass at $t_n$ (resp.\ $t'_n$) for all $n\in\N$. Then $\varpi_n(\cdot)$ does not converge in the Skorokhod topology, while (assuming all assumptions of the lemma are satisfied) the process $\wh\varpi(\cdot)$ is c\`adl\`ag such that both $(\nu^Q)^+$ and $(\nu^{Q'})^+$ have a point mass at $t$. Another example is the case when $f_Q=\1_{[t_n,t'_n)}$ for the same sequences $(t_n)_{n\in\N}$ and $(t'_n)_{n\in\N}$ and some quad $Q$.
\end{remark}

 The process $f_\eta^{\Quad}(\cdot):= \1_{\Quad \in \omega_\eta^\gamma(\cdot)}$ is c\`adl\`ag and, thus, belongs to $\II$. By Lemma \ref{prop3}, there exists a collection of random variables $(\omega_\infty^\gamma(q))_{q\in\Q_+}$ such that the following convergence holds in law 
 \eqb
 (\omega_\eta^\gamma(q))_{q\in\Q_+}\rta (\omega_\infty^\gamma(q))_{q\in\Q_+}\,.
 \label{eq6}
 \eqe
 We want to use these two results and Lemma \ref{l.cadlag} to argue that $(\omega_\infty^\gamma(q))_{q\in\Q_+}$ has an extension to $\R_+$ which is c\`adl\`ag.
 
 \begin{lemma}\label{l.cadlag modification}
 	Let $\gamma\in(0,\sqrt{3/2})$ and let $(\omega_\infty^\gamma(q))_{q\in\Q_+}$ be as in \eqref{eq6}. Then the following process is well-defined as a c\`adl\`ag function
 	\eqb
 	t\mapsto \lim_{q\downarrow t} \omega_\infty^\gamma(q)\,.
 	\label{eq7}
 	\eqe
 \end{lemma}
\begin{remark}
	In the remainder of the paper we let $(\omega_\infty^\gamma(t))_{t\in\R_+}$ denote the function in \eqref{eq7}.
\end{remark}
 \begin{proof}[Proof of Lemma \ref{l.cadlag modification}]
 	Let us take $S=\Q_+$ and treat
 	\begin{equation}\label{e.coupling final}
x_\eta:=\left ((\omega_\eta^\gamma(s))_{s\in S},(\nu^+_{f_\Quad(\omega_\eta^\gamma(\cdot))})_{\Quad \in \QQ^\infty},(\nu^-_{f_\Quad(\omega_\eta^\gamma(\cdot))})_{\Quad \in \QQ_\infty}\right )
 	\end{equation}
 	as a sequence in $\HH^S\times(\M([0,1)))^{\QQ_\infty}\times(\M([0,1)))^{\QQ_\infty} $ endowed with the product topology induced by $d_{\HH}$ and the weak topology. Note that the first coordinate is tight, as we have finite-dimensional convergence by Proposition \ref{prop3}. Furthermore, the second and the third terms are also tight thanks to Proposition \ref{p.tightness}. This implies that the law of $x_\eta$ is tight. By the Skorokhod embedding theorem, we have a coupling of (a subsequence of) $x_\eta$ such that for all $s\in S$ and $\Quad \in \QQ^\infty$, a.s.\ $\omega_\eta^\gamma(s)\to \omega_\infty^\gamma(s)$ for $d_\HH$,  $\nu^+_{f_\Quad(\omega_\eta^\gamma(\cdot))}\to (\nu^\Quad)^+$, and $ \nu^-_{f_\Quad(\omega_\eta^\gamma(\cdot))} \to (\nu^\Quad)^-$ for some random measures $(\nu^{\Quad})^+$ and $(\nu^{\Quad})^-$ in $\cM([0,1))$. Thanks to Lemma \ref{l.cadlag}, we see that it is enough to show that a.s.\ there is no point in $S$ in the support of $ (\nu^\Quad)^+$ or $ (\nu^\Quad)^-$. This follows, because for a given $\Quad \in \QQ$ and a given $s\in S$, the probability that $s$ is in the support of either $ (\nu^\Quad)^+$ or $ (\nu^\Quad)^-$ is $0$ as it can be computed from the Proposition \ref{p.tightness}.
 \end{proof}
 
 The following proposition implies the convergence of  $(\omega_\eta^\gamma(t))_{t\in\R_+}$ in the $L^1$ topology of Definition \ref{def:L1}, and therefore completes the proof of Theorem \ref{thm1}(ii).
 \begin{prop}
 	For any $T>0$, $\omega_\eta^{\gamma}(\cdot)$ converges in law  to $\omega_\infty^\gamma(\cdot)$ for the topology of  $L^1([0,T], (\HH,d_\HH))$.
 \end{prop}
 \begin{proof}
 	To prove this proposition, we will use the coupling of Remark \ref{r.coupling}. Recall that in this coupling $\omega_\eta^\gamma(q)\to \omega_\infty^\gamma(q)$ for the metric $d_\HH$ for all $q\in \Q_+$. Let us now measure the expected value of the $L^1$ distance between $\omega_\eta^\gamma(\cdot)$ and $\omega_\infty^\gamma(\cdot)$.   
 	\begin{align*}
 	\E\left[\int_0^T d_\HH(\omega_\eta^\gamma(t),\omega_\infty^\gamma(t))dt  \right] = \int_0^T \E\left[ d_\HH(\omega_\eta^\gamma(t),\omega_\infty^\gamma(t)) \right] dt\,.
 	\end{align*}
 	As $d_\HH$ is bounded, we only need to show that $\E\left[ d_\HH(\omega_\eta^\gamma(t),\omega_\infty^\gamma(t)) \right]$ converges to $0$ for each fixed $t\in [0,T]$. By the triangle inequality, for arbitrary $q\in\Q_+$,
 	\begin{align*}
 	d_\HH(\omega_\infty^\gamma(t),\omega_\eta^\gamma(t))
 	\leq
 	d_\HH(\omega_\infty^\gamma(t), \omega_\infty^\gamma(q))
 	+
 	d_\HH(\omega_\infty^\gamma(q), \omega_\eta^\gamma(q))
 	+
 	d_\HH(\omega_\eta^\gamma(q), \omega_\eta^\gamma(t))\,.
 	\end{align*} 
	Since $\omega_\infty^\gamma(\cdot)$ is c\`adl\`ag we can find $\delta>0$ such that for all $q\in[t,t+\delta]$ the expected value of the first term  is smaller than $\epsilon>0$. For any fixed $q$ the second term also converges to $0$ a.s.\ as $\omega_\eta^\gamma(q) \to \omega_\infty^\gamma(q)$ by the choice of coupling. Thus, to finish the proof we have to show that for sufficiently small $\delta'>0$ and all $q\in[t,t+\delta']$,
 	\[\sup_{\eta\in(0,1]} \E\left[ d_\HH(\omega_\eta^\gamma(q),\omega_\eta^\gamma(t))\right] <\epsilon.\]
 	This follows from Proposition \ref{p.tightness}.
 \end{proof}

\subsection{Markov property}
This section concerns the (conditional on $h$) Markov property of $\omega_\infty^\gamma(\cdot)$. Let us first remark that the process $\omega_\eta^\gamma(\cdot)$ is not a Markov process, as the past of the process gives us information about the underlying field $h$. However, when one conditions on $h$ then $\omega_\eta^\gamma(\cdot)$ is Markov. The following proposition says that same is true for $\omega_\infty ^\gamma(\cdot)$. The proof is identical to the proof of \cite[Theorem 11.1]{gps-near-crit} and is therefore omitted.
\begin{prop}[Markov property]
	For any $t>0$ the law of $(\omega^\ga_\infty (t))_{t\geq 0}$ given $h$ is that of a simple Markov process, reversible with respect to the law of $\omega^\ga_\infty(0)$.
\end{prop}

\section{Mixing of Liouville dynamical percolation}\label{s6}
In this section we discuss mixing properties of LDP. We study the covariance between the crossings of given finitely many quads at two different times for the limiting dynamic. These estimates are useful to prove the convergence of the finite-dimensional laws in the supercritical regime and to understand the mixing properties of the limiting dynamics in the subcritical regime.  

\subsection{Convergence of the quantum spectral measure}
In the following we are going to see that the modified Liouville spectral measure, defined in \eqref{e. liouville spectral sample}, converges in law as $\eta \to 0$. This allows us to use Lemma \ref{l.covariance} to understand the covariances of $\omega_\infty^\gamma(\cdot)$.
\begin{prop}\label{p.convergence_estimate}
	For $\eta>0$ and $\gamma\in(0,\sqrt{3/2})$, let $f_\eta$ be a Boolean function and let $\S_\eta$ be the corresponding spectral measure (see \eqref{e. identification}). Assume that $(\S_\eta,\1_{\S_\eta(D)\neq 0})$ converges in law to $(\S,\1_{\S(D)\neq 0})$, where we use the topology of weak convergence in the first coordinate. Furthermore, assume $\sup_{\eta\in(0,1]}\E[\EE_{d}(\S_\eta)]<\infty$ for some $d>\gamma^2/2$. Then 
	\begin{equation}
	\lim_{\eta \to 0} \Cov\left[f_\eta( \omega^\gamma_\eta(0)),f_\eta(\omega^\gamma_\eta(t))\mid h \right] = \E\left[\exp(-t\mu^\S_{\gamma h}(D))\1_{\S(D)\neq 0}\mid h\right]\,.
	\label{e. continuous covariance}
	\end{equation}
	Furthermore, for any $\gamma\in(0,2)$ such that $\mu^{\S_\eta}_{\gamma h}(D)$ converges in law to 0,
	\begin{equation}\label{e.to 0}
	\lim_{\eta \to 0} \E\left[f_\eta(\omega^\gamma_\eta(0))f_\eta(\omega^\gamma_\eta(t)) \right] =1\,.
	\end{equation}
\end{prop}
\begin{proof}
	For the first part, thanks to Skorokhod embedding we can assume that $(\S_\eta,\1_{\S_\eta=\emptyset})\rta(\S,\1_{\S=\emptyset})$ a.s. Let $h$ be the same field for all $\eta$. Thanks to Lemma \ref{l.covariance} and bounded convergence, it is enough to prove that the following convergence holds in probability 
	\[\lim_{\eta \to 0} \mu_{\gamma h}^{\S_\eta} (D) = \mu_{\gamma h}^\S(D).\]
	This follows from Proposition \ref{p.convergence_of_LQG}.
	
	For the second part, we note that by \eqref{eq5} and bounded convergence, 
	\begin{equation}
	\lim_{\eta \to 0} \E\left[f(\omega^\ga_\eta(0))f(\omega^\ga_\eta(t)) \right] =\lim_{\eta \to 0}\E\left[\exp(-t\mu^{\S_\eta}_{\gamma h}(D)) \right]\to 1\,.
\end{equation}
\end{proof}

The importance of the latter proposition is that it can be applied for crossings.
\begin{corollary}\label{c. convergence crossing}
	Let $\gamma\in(0,\sqrt{3/2})$, let $Q_1,...,Q_n \in \QQ^k$ be a collection of quads, and
	let $f:\HH\to\{0,1 \}$ be the function which says whether all quads $Q_1,...,Q_n$ are crossed. 
	Let $\S_\eta$ be the spectral measure associated with $f$ on the lattice $\Tg_\eta$, and 
	let $\S$ be the weak limit in law of $\S_\eta$ as in Proposition \ref{prop:spectralsample}. Then, almost surely
	\begin{equation}
	\Cov\left[f(\omega^\gamma_\infty(0)),f(\omega^\gamma_\infty(t))\mid h\right]= \E\left[\exp(-t\mu_{\gamma h}^\S(D))\1_{\S(D)\neq 0}\mid h \right]\stackrel{t\to \infty}{\longrightarrow} 0\,.
	\end{equation}
\end{corollary}
\begin{proof}	
	By Proposition \ref{prop:spectralsample}, we know that the measures $\S_\eta$ converges in law for the weak topology to some measure $\S$, and by the lower bound in Theorem \ref{t. spectral tightness} we see that $(\S_\eta,\1_{\S_\eta(D)\neq 0})$ converges in law to $(\S,\1_{\S(D)\neq 0})$.
	\begin{equation*}
	\lim_{\eta \to 0}\Cov\left[f(\omega^\gamma_\eta(0)),f(\omega^\gamma_\eta (t))\mid h\right]= \E\left[\exp(-t\mu_{\gamma h}^\S(D))\1_{\S(D)\neq 0} \mid h\right]\,.
	\end{equation*}
	Note that thanks to Proposition \ref{p.tightness}, $(f(\omega^\gamma_\eta(0)),f(\omega^\gamma_\eta(t)))\to (f(\omega^\gamma_\infty(0)),f(\omega^\gamma_\infty(t)))$  in law. Therefore
	\begin{equation}\label{e. covariance crossing}
	\Cov\left[f(\omega^\gamma_\infty(0))f(\omega^\gamma_\infty (t))\mid h\right]= \E\left[\exp(-t\mu_{\gamma h}^\S(D))\1_{\S(D)\neq 0} \mid h\right]\,.
	\end{equation}
	By \eqref{p.nontrivial}, $\mu_{\gamma h}^\S(\D)> 0$ a.s.\ if $\S(D)\neq 0$. Thus, the right  side of \eqref{e. covariance crossing} converges to $0$ as $t\to \infty$. 
\end{proof}

\subsection{Mixing properties of the subcritical regime}

After understanding the correlations of quad crossings, we can obtain information about the mixing in the subcritical regime. The following proposition gives the non-quantitative mixing results of Theorem \ref{p.quantitative decorrelation} for the case where the event $C(t)$ can be expressed in terms of a finite number of quad crossings. The proof combines Corollary \ref{c. convergence crossing} and the inclusion-exclusion principle.
\begin{prop}\label{prop1}
	Consider Liouville dynamical percolation $(\omega^\gamma_\infty(t))_{t\geq 0}$ of parameter $\gamma\in(0,\sqrt{3/2})$. Let $k$ be a natural number, and let $Q_1,\dots,Q_k$ be quads. For some $j\leq k$ and any $t\geq 0$ let $A(t)$ be the event that $Q_1,\dots,Q_j$ are \emph{not} crossed at time $t$, and that $Q_{j+1},\dots,Q_k$ are crossed at time $t$. Then for any event $B(0)$ measurable with respect to $\omega^\ga_\infty(0)$, we have that a.s. 
	\eqb
	\lim_{t\rta\infty}\P[ A(t); B(0) \mid h]= \P[A(0)]\P[B(0)]\,.
	\label{eq:decorr}
	\eqe
\end{prop}
\begin{proof}
	Corollary \ref{c. convergence crossing} gives that if $j=0$ then
	\eqbn
	\P[ A (t); A(0) \mid h]\rta \P[A(0)]^2\,.
	\eqen
	Applying \eqref{eq:fg} we get that \eqref{eq:decorr} holds for $j=0$.
	
	Consider the lexicographical ordering on pairs $(k,j)$ with $k\in\N$, $j\in\N\cup\{0 \}$, and $j\leq k$. We will prove \eqref{eq:decorr} by induction on $(k,j)$ with this ordering. The case $(k,j)=(1,0)$ is immediate since \eqref{eq:decorr} holds for $j=0$. Consider some $(k,j)$, and assume \eqref{eq:decorr} has be proved for all $(k',j')<(k,j)$. To conclude the proof by induction it is sufficient to argue that \eqref{eq:decorr} also holds for $(k,j)$. If $j=0$ this is immediate by the previous paragraph, so we assume that $j>0$.
	
	Let $\wt A(t)$ be the event that $Q_2,\dots,Q_j$ are \emph{not} crossed at time $t$, and that $Q_{j+1},\dots,Q_k$ are crossed at time $t$. Let $A'(t)$ be the event that $Q_1$ is \emph{not} crossed at time $t$. Then $A(t)=\wt A(t)\cap A'(t)$. Using this identity and the induction hypothesis for $(k-1,j-1)$ and $(k,j-1)$, we get
	\eqbn
	\begin{split}
		\P[ A(t); B(0) \mid h ]
		&= \P[\wt A(t); B(0) \mid h] - \P[\wt A(t);(A'(t))^c; B(0) \mid h]\\
		&\stackrel{t\to \infty}{\longrightarrow} \P[\wt A(t)]\cdot \P[B(0)]- 
		\P[\wt A(t);(A'(t))^c]\cdot
		\P[B(0)]\\
		&=\P[A(t)]\cdot\P[B(0)],
	\end{split}
	\eqen
	which concludes the proof by induction.
\end{proof} 

\begin{proof}[Proof of Theorem \ref{p.quantitative decorrelation}(i)]
	For any $t\geq 0$, the set of the events $A(t)$ from Proposition \ref{prop1} generates the same $\sigma$-algebra as $\omega^\ga_\infty(t)$ viewed as an element of $(\HH,d_{\HH})$ (see \cite[Theorem 1.13]{ss-planar-perc}). Therefore, by \cite[Theorem D, Section 13]{halmos-book}, given any $\eta>0$ we can find an event $A(t)$ as in Proposition \ref{prop1} such that $\P[ A(t)\Delta C(t) ]<\eta$. It follows that
	\eqbn
	\begin{split}
		&\limsup_{t\rta\infty}|\P[ C(t); B(0) \mid h ]-\P[C(t) ]\cdot\P[B(0)]|\\
		&\qquad\leq
		\limsup_{t\rta\infty}|\P[ A(t); B(0) \mid h ]-\P[A(t) ]\cdot\P[B(0)]|+2\eta= 2\eta\,.
	\end{split}
	\eqen 
	Since $\eta$ was arbitrary this concludes the proof.
\end{proof}

\section{Convergence in the supercritical regime}\label{s7}
	In this section we will prove Theorem \ref{thm1} for the supercritical case, i.e., $\gamma\in(\sqrt{3/2},2)$. We start by proving convergence of the finite-dimensional distribution, which is immediate by the following proposition. 
	\begin{prop}\label{p.supercritical is constant}
		Let $\gamma\in(\sqrt{3/2},2)$. Then for any $t>0$ and $\epsilon>0$,
		\[\P\left (d_\HH^{\op{mod}}(\omega^\gamma_\eta(0),\omega^\gamma_\eta(t))\leq \epsilon\right )\to 1 \ \ \text{ as }\eta \to 0.\]
		Thus, the finite-dimensional distribution of $\omega_\eta^\gamma(\cdot)$ converge to those of the constant process $\omega_\infty^\gamma(t)\equiv \omega^\ga_\infty(0)$. 
	\end{prop}
	\begin{proof}
		Note that when $\gamma\in(\sqrt{3/2},2)$, Proposition \ref{p.convergence_to_0} implies that if $\S_\eta$ is the spectral measure for the crossing of a quad $\Quad$, then $\mu_{\gamma h}^{\S_\eta} \to 0$ in law as $\eta\rta 0$. Thus, if $f_\Quad$ is the Boolean function encoding the crossing of $\Quad$,	
		\begin{align*}
		\P\left (d_\HH^{\op{mod}}(\omega^\ga_\eta(0),\omega^\ga_\eta(t))>2^{-k}\right )
		&\leq \sum_{\Quad \in \QQ^k} \P(\Quad \in \omega_\eta^\ga(0)\Delta \omega_\eta^\ga(t_j))\\
		&\leq (\# \QQ^k)\cdot \sup_{Q\in\cQ^\infty}(1-\E\left[f_\Quad(\omega^\ga_\eta(0))f_\Quad(\omega^\ga_\eta(t))\right] )\to 0.
		\end{align*}
		as $\eta \to 0$, thanks to \eqref{e.to 0}.
	\end{proof}

	We can now finish the proof of Theorem \ref{thm1}.
	\begin{proof}[Proof of Theorem \ref{thm1}(iii)] Let us take a coupling of $\omega_\eta^\gamma(\cdot)$ such that $\omega_\eta^\gamma(0)$ converges in probability to $\omega_\infty^\gamma(0)$ for $d_{\HH}$. Now, Proposition \ref{p.supercritical is constant} implies that the finite-dimensional distributions of $\omega_\eta^{\gamma}(\cdot)$ converges as $\eta \to 0$ to those of the constant process $\omega_\infty^\gamma(0)$. To show convergence in $L^1([0,T], (\HH,d_\HH))$, we study the expected value of the $L^1$ distance between $\omega_\eta^{\gamma}(\cdot)$ and $\omega_\infty^\gamma(0)$.
		\begin{align}
		\E\left[\int_0^T d_\HH(\omega_\eta^{\gamma}(t),\omega_\infty^{\gamma}(0))\,dt \right]= \int_0^T \E[d_\HH(\omega_\eta^{\gamma}(t),\omega_\infty^{\gamma}(0))]\,dt\,.
		\label{eq1}
		\end{align}
		For any $t\in [0,T]$, 
		\[
		d_\HH(\omega_\eta^{\gamma}(t),\omega_\infty^{\gamma}(0))
		\leq d_\HH(\omega_\eta^{\gamma}(t),\omega_\eta^{\gamma}(0))
		+
		d_\HH(\omega_\eta^{\gamma}(0),\omega_\infty^{\gamma}(0))\,. 
		\]
		The first term on the right side converges to $0$ in probability thanks to Proposition \ref{p.supercritical is constant}, and the second term converges to 0 in probability because $\omega_\eta^\gamma(0)$ converges to $\omega_\infty^\gamma(0)$. Since $t$ was arbitrary, this implies that the right side of \eqref{eq1} converges to 0 as $\eta\rta 0$.
	\end{proof}

\section{Quantitative decorrelation bounds}\label{s8}

In this section, we obtain explicit decorrelation bounds in the case $\gamma\in(0,\sqrt{3/4})$. We use Proposition \ref{p.mass} to obtain a quantitative estimates on the decorrelation of the crossings for cLDP. In Proposition \ref{p.quantitative} right below we have rewritten Proposition \ref{p.quantitative decorrelation} to be more explicit.
	The following lemma will be used in the proof.
	\begin{lemma}
		Let $Q$ be a rectangular quad and consider critical site percolation on $\Tg_\eta$ for some $\eta>0$. Let $\S_\eta$ denote the spectral measure associated with the crossing of $Q$. Then for any $d\in(0,3/4)$,
		\eqb
		\sup_{\eta\in(0,1)}\E[ \cE_d(\S_\eta) ]<\infty.
		\label{eq12}
		\eqe
		\label{prop22}
	\end{lemma}
	\begin{proof} 
		Consider the double integral over $x,y$ in \eqref{eq11}. Recall that the spectral sample has the same one- and two-point functions as the pivotal points (see e.g.\ \cite{kkl-boolean} and \cite[Section 1.1]{gps-fourier}). By 
		Proposition \ref{prop:energy}, \eqref{eq12} holds if we restrict the integral to points where $x,y$ are bounded away from the boundary. By an analysis of boundary and corner arm exponents, the integral is also finite if the points are near the boundary. See the last two paragraphs in proof of Proposition \ref{p.tightness} for a similar estimate.
	\end{proof}
\begin{prop}\label{p.quantitative}
Let us work in the context of Theorem \ref{p.quantitative decorrelation}. For all $\gamma\in(0,\sqrt{3/4})$ and  any $\xi<2\theta/5$  (recall from Proposition \ref{p.mass} that $\theta=\theta(d,\gamma):= \frac {d - \gamma^2} {d + \gamma^2}$), we have that
\[(\P[A(0)A(t)]-\P[A(0)]^2)t^{\xi}\to 0,\]
and 
\begin{equation}\label{e.decorrelation 2}
(\P[B(0)A(t)]-\P[B(0)]\P[A(0)])t^{\xi/2}\to 0\,.
\end{equation}
Furthermore, for  all $\xi<2\theta/5$ and almost surely in $h$, we have the following quenched decorrelation bound
\begin{align*}
&(\P[A(0)A(t)\mid h]-\P[A(0)]^2)t^\xi\to 0 \; \; \ \ \ \ \ \ \text{ and }\\
& (\P[g(0)A(t)\mid h]-\P[g(0)]\P[A(0)])t^{\xi/2}\to 0\,.
\end{align*}
\end{prop}
\begin{proof}
We can use Corollary \ref{c. convergence crossing}
 and  Proposition \ref{p.mass} to get
\begin{align}
\nonumber\Cov[A(0),A(t)]
&= \E\left[\exp(-t\mu_{\gamma h}^\S(D)) \1_{\S(D)\neq 0} \right] \\
\label{e. quantitative bound explicit}
&\leq   \P\left( K\left[\frac{\EE_d(\S)}{ \S(D)}\right]^{1/\theta}\geq  t \right) + \E\left[ \frac{K}{\S(D) t^{\theta}}\wedge 1 \right].
\end{align}

Let us control the first term. Take $a\in \R$, and upper bound the first term on the right side of \eqref{e. quantitative bound explicit} by
\begin{align*}
\P\left(\EE_d(\S)\geq t^{a} \right) + \P\left(\S(D)\leq K^\theta t^{a-\theta} \right) \leq \E\left[\EE_d(\S) \right] t^{-a}+O(1)t^{-2(\theta -a )/3}\,.
\end{align*}
Note that the first term on the right side is finite by Lemma \ref{prop22}. 

Now for any $b>0$, we can bound the second term on the right side of \eqref{e. quantitative bound explicit} by
\begin{align*}
t^{-b}+ \P\left[\S(D)t^\theta \leq t^{b} K \right]\leq t^{-b} + O(1)t^{-2(\theta-b)/3}.
\end{align*}
By taking $a=b=2\theta/5$, we obtain 
\begin{align*}
\Cov[A(0),A(t)]\leq O(1) \, t^{-\frac{2\theta}{5}}.
\end{align*}
Equation \eqref{e.decorrelation 2} is obtained by using \eqref{eq:fg2} and this result.

Finally, to conclude the proof, we obtain the quenched results thanks to the claim below. \end{proof}

\begin{claim}
	Let $X_t$ be a random decreasing process such that $\E\left[ X_t \right]t^{\xi}\to 0$ as $t \to \infty$. Then a.s.\ for all $\xi'<\xi$, $X_t t^{\xi'}\to 0$ as $t \to \infty$.
\end{claim} 
Before proving the claim let us note that this is exactly what is needed as $\P[A(0)A(t)\mid h]-\P[A(0)]^2$ is decreasing in $t$ and its expected value is $\P[A(0)A(t)]-\P[A(0)]^2$. Note that the second equation also follows the same argument.
\begin{proof}
	Let us note that as $X_t$ is decreasing it is enough to prove the claim for the sequence $t_n=2^{n}$. First let us take $\delta>0$ and use Markov's inequality to see that for $n$ sufficiently large,
	\begin{align*}
	\P[X_t t^{\xi}>t^{\delta}]\leq \E\left[X_t t^{\xi}\right] t^{-\delta} \leq 2^{-n\delta}.
	\end{align*}
	We conclude by applying the Borel-Cantelli lemma.
\end{proof}

\appendix
\addcontentsline{toc}{section}{Appendices}
\section*{Appendix}
\section{Limit properties of LQG measures}\label{app:lqg}

\subsection{Continuity of LQG measures}
Let $\gamma\geq 0$, let $\sigma^{n}$ be a sequence of random measures in a bounded domain $D\subset\C$ converging in probability for the Prokhorov topology to a measure $\sigma$ with finite total mass, and let $\mu^{n}_{\gamma h}=\mu^{\sigma^n}_{\gamma h}$ be the sequence of $\gamma$-LQG measures of $h$ with respect to $\sigma_n$. The goal of this section is to give a sufficient condition for $\mu^{n}_{\gamma h}$ to converge to $\mu_{\gamma h}^{\sigma}$, the $\gamma$-LQG of $h$ with respect to $\sigma$. 

We will use several estimates from \cite{berestycki-elem}, where it was proved that $\mu^\sigma_{\gamma h}$ is the limit of $\mu^\sigma_{\gamma h_\eps}$ in $L^1$ for $h_\eps$ a smooth approximation to $h$ (see \eqref{e. def GMC}). Many notations will be borrowed from that paper, and it is advisable that the reader is familiar with that paper before reading the proof. 
\begin{prop}\label{p.convergence_of_LQG}
	Take $d >0$ and assume that $\sup_n \E[\EE_{d}(\sigma^{n})]<\infty$. Then, for all $\gamma^2<2d$ and deterministic sets $\Os\subseteq \C$ such that $\sigma (\partial\Os)=0$ a.s., the LQG measures considered above are well-defined and we have that $\mu^{n}_{\gamma h}(\Os)\to \mu_{\gamma h}^{\sigma}(\Os)$ in $L^1$. Furthermore, $\mu^{n}_{\gamma h}|_{\Os}$ converges in probability to  $\mu_{\gamma h}^{\sigma}|_{\Os}$ in the topology of weak convergence of measures on $\Os$. If $\gamma^2<d$ and $\sigma^n\rta\sigma$ in $L^2$, then $\mu^{n}_{\gamma h}(\Os)\to \mu_{\gamma h}^{\sigma}(\Os)$ in $L^2$.
\end{prop} 
\begin{proof}
	Fix $\gamma^2<2d$. For simplicity we write $\mu^{n}_{\gamma h}$ as $\mu^{n}$. For some smooth approximation $h_\ep$ to $h$ (e.g.\ the circle average approximation) we write
	$\mu^{n}_{\gamma h_\ep}$ as $\mu^{n}_\ep$. 
	By the triangle inequality, for $n\in\N$,
	\begin{equation*}
	\E|\mu^{n}(\Os)- \mu(\Os)|
	\leq 
	\E|\mu^{n}(\Os)- \mu^n_{\eps}(\Os)|
	+
	\E|\mu^{n}_{\eps}(\Os)- \mu_{\eps}(\Os)|
	+
	\E|\mu_{\eps}(\Os)- \mu(\Os)|\,.
	\end{equation*}
	Since $h_\eps$ is smooth, $\sigma (\partial\Os)=0$, and $\sigma^n\rta\sigma$ in probability, we get that $\mu^{n}_{\eps}(\Os)\rta \mu_{\eps}(\Os)$ in probability. Using $\sup_n \E[\EE_{d}(\sigma^{n})]<\infty$, this gives that
	the second term on the right side converges to 0 as $n\rta\infty$ for any fixed $\eps$. The third term converges to 0 as $\ep\rta 0$ by e.g.\ the main result of \cite{berestycki-elem}.  Therefore, to show that $\mu^{n}_{\gamma h}(\Os)\to \mu_{\gamma h}^{\sigma}(\Os)$ in $L^1$ it is sufficient to handle the first term, i.e., to show that
	\begin{linenomath*}
		\begin{align*}
		\lim_{\eps\rta 0}\sup_{n\in \N} \E\left[|\mu^{n}(\Os)-\mu^{n}_\epsilon(\Os)| \right]=0.
		\end{align*}
	\end{linenomath*}
	This result follows from a close inspection of \cite{berestycki-elem}.
	
	For some $\eps_0\leq 1$ to be determined right below, define the following event $G_\eps^\alpha(x)$, which, roughly speaking, says that the field $h$ is not too large close to $x$  
	\eqbn
	G_\eps^\alpha(x) = \{ h_r(x) \leq \alpha\log(1/r) \text{\,\,for\,\,all\,\,}
	r\in[\eps,\eps_0]
	 \}\,.
	\eqen
	Then define
	\[I_{\epsilon}^n:=\int_{\Os} \1_{(G_\epsilon^a)^c(x)} \mu^n_\epsilon(d^2x), \ \ \ J_\epsilon^n:=\int_{\Os} \1_{G_\epsilon^a(x)} \mu^n_\epsilon(d^2x),\]
	and note that $\mu_\ep^n(\Os)=I_\eps^n+J_\eps^n$.
	By \cite[Lemma 3.2]{berestycki-elem}, for all $\eta>0$ there exists $\epsilon_0>0$ such that $\sup_{n\in \N} \E\left[ I_\epsilon^n\right]\leq \eta$ for all $\ep\in(0,\ep_0)$; we fix $\eps>0$ such that this condition is satisfied. It is sufficient to show the following
	\eqbn
	\lim_{\ep,\ep'\rta 0}\sup_{n\in\N}\E[(J_\ep^n-J_{\ep'}^n)^2]=0\,.
	\eqen
	We will prove this by showing the existence of a function $F:\Os\times\Os\to\R$ such that uniformly in $n$,
	$$
	\E\left[(J_\epsilon^n)^2\right],\,\E\left[J_\epsilon^n J_{\ep'}^n\right]\rta\iint_{\Os\times\Os} F(x,y)\,\sigma^n(d^2x)\sigma^n(d^2y)
	\quad\text{\,as\,\,}\ep,\ep'\rta 0.
	$$
	We just treat  $\E\left[(J_\epsilon^n)^2\right]$, since $\E\left[J_\epsilon^n J_{\ep'}^n\right]$ is treated in the same way. We have
	\eqb
	\begin{split}
		\E\left[(J_\epsilon^n)^2\right]
		=&\,\iint_{|x-y|\leq \delta} e^{\gamma^2\E\left[h_\epsilon(x)h_\epsilon(y)\right]  }\wt \P(G_\epsilon(x)\cap G_\epsilon(y))\,\sigma^n(d^2x)\sigma^n(d^2y)\\
		&+
		\iint_{|x-y|\geq \delta} e^{\gamma^2\E\left[h_\epsilon(x)h_\epsilon(y)\right]  }\wt \P(G_\epsilon(x)\cap G_\epsilon(y))\,\sigma^n(d^2x)\sigma^n(d^2y),
	\end{split}
	\label{eq2}
	\eqe
	where $\wt \P$ is a certain probability measure absolutely continuous with respect to $\P$ (defined above \cite[equation (3.8)]{berestycki-elem}). Now  \cite[equation (3.12)]{berestycki-elem} shows that we can find $\beta<d$ (which corresponds to choosing a nice $\alpha>0$ in \cite{berestycki-elem}) such that for all $n\in \N$, the first term on the right side of \eqref{eq2} is smaller than a constant depending only on the correlation kernel of $h$ times
	\begin{linenomath}
		\begin{align*}
		\iint_{|x-y|\leq \delta} |x-y|^{-\beta} \sigma^n(d^2x)\sigma^n(d^2y)\leq \delta^{\ol\beta-\beta} \sup_{n} \EE_{\ol\beta}(\sigma^n),
		\end{align*}
	\end{linenomath}
	where $\ol\beta$ is chosen such that $\ol\beta\in(\beta,d)$.
	Given $\eta>0$, let us choose $\delta>0$ such that the first term on the right side of \eqref{eq2} is smaller than $\eta$. Now, as in \cite[Lemma 4.1]{berestycki-elem}, when $|x-y|\geq \delta$, we have that $e^{\gamma^2\E\left[h_\epsilon(x)h_\epsilon(y)\right]  }\wt \P(G_\epsilon(x)\cap G_\epsilon(y))$ converges in the topology of uniform convergence to a function $F(x,y)$. Thus, uniformly in $n$, the second term on the right side of \eqref{eq2} converges to $\iint_{|x-y|\geq \delta}F(x,y)\sigma^n(d^2x)\sigma^n(d^2y)$. Since $\eta$ was arbitrary, this concludes the proof that $\mu^{n}_{\gamma h}(\Os)\to \mu_{\gamma h}^{\sigma}(\Os)$ in $L^1$.
	
	The next assertion of the lemma is that $\mu^{n}_{\gamma h}|_{\Os}$ converges in probability to  $\mu_{\gamma h}^{\sigma}|_{\Os}$ in the topology of weak convergence of measures on $\Os$. The proof can be carried out exactly as in \cite{berestycki-elem} and is therefore omitted.
	
	To conclude the proof of the lemma, we will argue that we also have $L^2$ convergence if $\gamma^2<d$ and $\sigma^n\rta\sigma$ in $L^2$. We have 
	\begin{equation*}
	\E|\mu^{n}(\Os)- \mu(\Os)|^2
	\leq 
	3\E|\mu^{n}(\Os)- \mu^n_{\eps}(\Os)|^2
	+
	3\E|\mu^{n}_{\eps}(\Os)- \mu_{\eps}(\Os)|^2
	+
	3\E|\mu_{\eps}(\Os)- \mu(\Os)|^2\,.
	\end{equation*}
	The second term on the right side converges to 0 since $h_\eps$ is smooth and by the assumption that $\sigma^n\rta\sigma$ in $L^2$. The third term on the right side converges to 0 by e.g.\ \cite{berestycki-elem}. The proof that the first term converges to 0 can be done as in the $L^1$ case, except that we may choose $\alpha>2$, which implies that $G_\eps^\alpha(x)$ does not occur for any $x$ a.s., so $I^n_\eps=0$ a.s.
\end{proof}

It is possible to bound uniformly the expected energy of the measure $\lambda^\epsilon$ and its approximations.
\begin{prop} For any $\epsilon>0$ and $d<3/4$,
	\[\sup_{\eta\in(0,1]}\E\left[\EE_d(\lambda^\epsilon(\omega_\eta)) \right]<\infty. \]
	\label{prop:energy}
\end{prop}
\begin{proof}
	This follows by the argument in the proof of \cite[Lemma 4.5]{gps-fourier}, where it is proved via quasi-multiplicativity that $\E[\lambda^\epsilon(\omega_\eta)^2]<\infty$.
\end{proof}

\subsection{Convergence to 0 of Liouville measures}\label{ss. convergence to 0}

Let $\beta>0$ and assume that $\sigma_\eta$ is a sequence of measures that can be written as
\begin{equation}\label{e.sigma}
\sigma_\eta(d^2z) = C_\eta\eta^{-\beta} \sum_{x\in I_\eta\subseteq \Tg_\eta} \1_{z\in B^{\op h}_\eta(x)}\,d^2z,
\end{equation}
where $C_\eta=\eta^{o(1)}$ is a deterministic sequence and $I_\eta$ is a (possibly random) set independent of $h$. 
Furthermore, assume that the expected total mass of the measure is bounded uniformly in $\eta$, i.e.,
\begin{equation}\label{e. size condition}
\sup_{\eta>0}C_\eta \eta^{2-\beta}\sum_{z\in \Tg_\eta}\P[z\in I_\eta]<\infty\,.
\end{equation}

We are interested in seeing when the Liouville measure associated to this measure converges to $0$. In particular, we are interested in proving the following proposition.
\begin{prop}\label{p.convergence_to_0}
	Assume $0\leq 2-\beta <\gamma^2/2$. Then the sequence of Liouville measures $\mu_{\gamma h}^{\sigma_\eta}$ converges to $0$ in probability.
\end{prop}
\begin{proof}
	We just need to prove that $\mu_{\gamma h}^{\sigma_\eta}(D)\to 0$ in probability. To do that, let us define
	\[A_\eta:=\{z\in D:  \mu_{\gamma h}(B^{\op h}_\eta(z))<\eta^{2-\gamma^2/2+\delta}\}\]
	for some $\delta>0$ to be determined. We have
	\begin{equation}\label{e.musigmaeta}
	\mu_{\gamma h}^{\sigma_\eta}(D)
	=C_\eta\eta^{-\beta} \sum_{z\in I_\eta\subseteq \Tg_\eta} \1_{z\in A_\eta} \mu_{\gamma h}(B^{\op h}_\eta(z))+C_\eta\eta^{-\beta} \sum_{z\in I_\eta\subseteq \Tg_\eta} \1_{z\notin A_\eta} \mu_{\gamma h}(B^{\op h}_\eta(z))\,.
	\end{equation}
	First we show that the first term on the right side of \eqref{e.musigmaeta} converges to $0$ in $L^1$. To do that, let us recall \cite[Proposition 4.1 and Corollary 6.2]{aru17}, which say that for any log-correlated field and any $q<4/\gamma^2$,
	\begin{align}\label{e. moments}
	\E\left[\mu_{\gamma h}(B^{\op h}_r(z))^q \right]\leq r^{ -\gamma^2 q^2/2 + (2+\gamma^2/2)q +O(1)},
	\end{align}
	where the $O(1)$ is uniform in $z$. Thus, for any $p>0$ the expected value of the first term on the right side of \eqref{e.musigmaeta} is upper bounded by
	\begin{align*}
	  &\E\left[	C_\eta\eta^{-\beta} \sum_{z\in I_\eta\subseteq \Tg_\eta}  \mu_{\gamma h}(B^{\op h}_\eta(z))^{1-p} \eta^{(2-\gamma^2/2+\delta)p}\right]\\
	&\qquad\leq \eta^{-\gamma^2(1-p)^2/2+(2+\gamma^2/2)(1-p)+o(1)}\eta^{(2-\gamma^2/2+\delta)p} \eta^{-\beta}\sum_{z\in \Tg}\P[z\in I_{\eta}]\\
	&\qquad=\eta^{-\gamma^2p^2/2+\delta p+O(1)}.
	\end{align*}
	Therefore, for any $\delta>0$ we can find $p>0$ sufficiently small such that the first term on the right side of \eqref{e.musigmaeta} goes to $0$ in probability.
	
	Now, we will show that the second term on the right side of \eqref{e.musigmaeta} converges to $0$ in probability. By Markov's inequality, 
	\begin{align*}
	\P(\mu_{\gamma h}(B^{\op h}_\eta(z))>\eta^{2-\gamma^2/2+\delta})
	\leq \eta^{ 2 - \left (2+\delta-\frac{\gamma^2}{2}\right )+O(1)}
	=\eta^{\frac{\gamma^2}{2}-\delta+O(1)},
	\end{align*}
	where again, the $O(1)$ can be taken uniformly for all $z\in I_\eta$. Therefore we can upper bound the probability that the second term on the right side of \eqref{e.musigmaeta} is bigger than $0$ by
	\begin{align*}
	\P(I_\eta \backslash A_\eta\neq\emptyset)
	=\E\left[ \E\left[ \sum_{z\in I_\eta} \1_{A_{z,\eta}} \,\Big|\, I_\eta \right] \right]
	\leq\E\left[|I_\eta| \right] \sup_{z}\P(A_{z,\eta}^c)
	\leq \eta^{-2+\beta}\eta^{\frac{\gamma^2}{2}-\delta+O(1)}\stackrel{\eta \to 0}{\longrightarrow} 0,
	\end{align*}
	where we have taken $2\delta=\gamma^2/2-2+\beta>0$. This is enough to conclude.
\end{proof}

\subsection{Convergence of the modified Liouville measure}
As in the section before, we work with measures of the type \eqref{e.sigma}. However, we now assume that $\gamma\in(0,\sqrt{3/2})$. We add the assumption that $\sigma_\eta\to \sigma$ a.s, and that $\sup_\eta \E[\EE_d(\sigma_\eta)]<\infty$ for a fixed $d>\gamma^2/2$. Let us note that Proposition \ref{p.convergence_of_LQG}  implies that $\mu_{\gamma h}^{\sigma_\eta}\to \mu_{\gamma h}^\sigma$ in probability for the weak topology. The issue we address in this section is the convergence of the measure
\begin{equation*}
\wt\mu_{\gamma h}^{C,\sigma_\eta}(d^2z):= C_\eta\eta^{-\beta}\sum_{x\in I_\eta} \1_{z\in B^{\op h}_\eta(x)\cap \M_C} \mu_{\gamma h}(d^2z)\,.
\end{equation*}
To do this, let us introduce the following set, where $\varrho$ is as in \eqref{e.MC intro},
\[\M^r_{C}:=\{x \in D: \mu_{\gamma h} (B^{\op h}_{2^{-n}}(x))<C\alpha^{2^{-n}}_4(2^{-n},1)(2^{-n})^{\varrho}, \text{ for all } n\leq \lfloor\log_2(r)\rfloor \},\]
and show the following lemmas.
\begin{lemma}\label{l.continuity mes}
	A.s. for any $n\in \N$,  the function $x\mapsto \mu_{\gamma h} (B^{\op h}_{2^{-n}}(x))$ is continuous.
\end{lemma}
\begin{proof}
	To see this let us define
	\[f(r):=\sup_{x\in D}\mu_{\gamma h} (\partial B^{\op h}_r(x)\cap D).\]
	Note that the lemma follows from just showing that $\P(f(2^{-n})=0)=1$ for all $n\in \N$.
	
	First, let us see that $f(r)$ is a measurable function of $h$. This follows because
	\[f(r)= \inf_{\epsilon>0} \sup_{x\in \Q^2\cap D}\mu_{\gamma h} ( (B^{\op h}_{r+\epsilon}(x)\backslash B^{\op h}_{r-\epsilon}(x))\cap D).\]
	
	The edges of the hexagonal lattice dual to $\Tg_\eta$ have angle with the $y$-axis equal to $0$, $2\pi/3$, or $4\pi/3$. Therefore, to conclude it is sufficient to show that a.s.\ no line in one of these three directions has positive mass. We will show this for lines parallel to the $y$-axis, but the two other directions can be treated by the exact same argument. 
	
	For simplicity we assume that $D\subset [0,1]^2$; the exact same argument works for $D$ contained in a larger square. For each $n\in\N$ let $\cI_n$ be a collection of $2^n$ rectangles with disjoint interior contained in $[0,1]^2$ of the form $[k2^{-n},(k+1)2^{-n}]\times[0,1]$. By a union bound, in order to conclude it is sufficient to show that for any $s>0$ and for all $I\in\cI_n$,
	\eqb
	\P[ \mu_{\gamma h}(I)>s ] < o_n(1)2^{-n}s^{-1},
	\label{eq8}
	\eqe
	where the $o_n(1)$ is uniform in $I$. Let $\ell=\lceil 2^{n/2}\rceil$, and divide $I$ into $\ell$ disjoint rectangles of width $2^n$ and height $2^n/\ell$. Define a new log correlated field $\wt h$ in $\wt D_n=[0,\ell]\times[0,2^n/\ell]$ as follows. Divide $\wt D_n$ into $\ell$ disjoint rectangles of width $2^n$ and height $2^n/\ell$, and let $\wt\cI_n$ denote this collection of rectangles. For some arbitrary enumeration of $\wt\cI_n$ and $\cI_n$ and $j=1,\dots,\ell$, set $\wt h$ restricted to the $j$th rectangle of $\wt\cI_n$ equal to $h$ restricted to the $j$th rectangle of $\cI_n$. Let $\check h$ be equal to $1+r$ times a log-correlated field in $[0,1]^2$ of the form \eqref{eq:kernel} which is independent of $n$, where $r>0$ is some small parameter to be determined; then the covariance kernel of $\check h$ is equal to $-(1+r)^2\log|x-y|+(1+r)^2g(x,y)$. For sufficiently large $n$, the covariance kernel of $\wt h$ will be pointwise smaller than the covariance kernel of $\check h|_{\wt D_n}$, so the by Kahane's convexity inequality \cite{kahane} (see also \cite{rhodes-vargas-review,aru17}), we have $\E[ \mu_{\gamma h}(I)^{1+r} ]=\E[ \mu_{\gamma \wt h}(\wt D_n)^{1+r} ]\leq \E[ \mu_{\gamma \check h}(\wt D_n)^{1+r} ]$. Proceeding as in e.g.\ \cite[Corollary 6.5]{aru17} we have $\E[ \mu_{\gamma \check h}(\wt D_n)^{1+r}]\asymp 2^{n(1+r)/2}$ for $r$ sufficiently small, so we get \eqref{eq8} by an application of Chebyshev's inequality.
\end{proof}

\begin{lemma}\label{l.boundary 0}
	For any $r>0$, we have that $\mu_{\gamma h}^\sigma(\partial \M_C^r) =0$ a.s.
\end{lemma}

\begin{proof}
	Let us define the set 
\begin{equation*}
		E_{C}^n:=\{x \in D: \mu_{\gamma h} (B^{\op h}_{2^{-n}}(x))=C\alpha^{2^{-n}}_4(2^{-n},1)2^{-n\varrho} \},
\end{equation*}
	and note that Lemma \ref{l.continuity mes} implies that $\partial \mu_{\gamma h}^\sigma(\partial \M_C^r) \subseteq \bigcup E_{C}^n$. Thus, it is enough to show that $\mu_{\gamma h}^\sigma(E_{C}^n)=0$. Thanks to Fubini's theorem, it is sufficient to show that for any fixed $x\in D$, a.s.,
	\eqb
	\P[\mu_{\gamma h} (B^{\op h}_{2^{-n}}(x))=C\alpha^{2^{-n}}_4(2^{-n},1)2^{-n\varrho}]=0.
	\label{eq10}
	\eqe
	By the proof of \cite[Lemma 5.1]{berestycki-elem} we can write $h$ on the form $h=\alpha g+h'$, where $g$ is a deterministic continuous function, $\alpha$ is a standard normal random variable, and $h'$ is a random log-correlated field independent of $\alpha$. We may assume that $g$ is not identically equal to zero in $B^{\op h}_{2^{-n}}(x)$. Condition on $h'$ and define the following random function
	$$
	G_{h'}(a) = \mu_{\gamma (ag+h')} (B^{\op h}_{2^{-n}}(x))
	=\int_{B^{\op h}_{2^{-n}}(x)} e^{\gamma a g(z)} d\mu_{\gamma h'}(z).
	$$
	By expanding $e^{\gamma a g(z)}$ pointwise as a power series in $a$, we get that, conditioned on $h'$, the function $a\mapsto G_{h'}(a)$ is real analytic. By calculating the second derivative of $a\mapsto G_{h'}(a)$, we see that the function is not constant. For any constant $c$, the set of points at which a non-constant analytic function is equal to $c$ cannot have any accumulation points; otherwise all derivatives of the function would be zero at this accumulation point, and the function would be constant. In particular, the set of points at which the function is equal to $c$ has zero Lebesgue measure. Since $G_{h'}(\alpha)\eqD \mu_{\gamma h} (B^{\op h}_{2^{-n}}(x))$ and $\alpha$ is a standard normal independent of $h'$, this implies \eqref{eq10}. 
\end{proof}

Let us use this lemma to prove the following proposition.
\begin{prop}\label{p.convergence_of_modified_LQG}
	For all $\gamma\in(0,\sqrt{3/2})$ and all open $\Os\subset D$ such that $\sigma (\Os)<\infty$ and $\sigma (\partial\Os)=0$ a.s., we have that  
	$\wt \mu_{\gamma h}^{C,\sigma_\eta}(\Os)
	=\mu_{\gamma h}^{\sigma_\eta}(\Os\cap \M_C)
	\to\mu_{\gamma h}^{\sigma}(\Os\cap\M_C)
	=\wt\mu_{\gamma h}^{C,\sigma}(\Os)$ in $L^1$ as $\eta \to 0$.
	Furthermore $\wt\mu^{C,\sigma_\eta}_{\gamma h}|_{\Os}$ converges in probability to  $\wt\mu_{\gamma h}^{C,\sigma}|_{\Os}$ in the topology of weak convergence of measures on $\Os$.
\end{prop}

\begin{proof}
	Let us start by fixing $r>0$ and upper bounding $|  \mu_{\gamma h}^{\sigma}(\M_C\cap \Os)-\wt \mu_{\gamma h}^{C,\sigma_\eta}(\Os)|$ by 
	\begin{equation}
	\begin{split}
	|\mu_{\gamma h}^{\sigma}(\M_C\cap \Os)
	-\mu_{\gamma h}^{\sigma}(\M_C^r\cap \Os)|
	&+
	| \mu_{\gamma h}^{\sigma}(\M_C^r\cap \Os)
	-\mu_{\gamma h}^{\sigma_\eta}(\M_C^r\cap \Os)|\\
	&+
	|\mu_{\gamma h}^{\sigma_\eta}(\M_C^r\cap \Os)
	-\wt \mu_{\gamma h}^{C,\sigma_\eta}(\Os)|\,.
	\end{split}
	\end{equation}
	Let us first note that as $r\to 0$ the first term converges to $0$ a.s.\ thanks to the fact that $\M_C^r\downarrow \M_C$ as $r\downarrow 0$. For the second term, we need to show is that for fixed $r>0$,  $\wt \mu_{\gamma h}^{C,\sigma_\eta}( \M_C^r\cap \Os)=\mu_{\gamma h}^{\sigma_\eta}(\M_C^r\cap \Os)\to \mu_{\gamma h}^{\sigma}(\M_C^r\cap \Os)$ as $\eta \to 0$. This is true thanks to the last assertion of Proposition \ref{p.convergence_of_LQG} and Lemma \ref{l.boundary 0}. For the last term, we use that $\M_C^r$ is decreasing in $r$ to show that it is equal to
	\begin{equation*}
	C_\eta \eta^{-\beta}\sum_{x\in I_\eta}\mu_{\gamma h}(B^{\op h}_\eta(x)\cap (\M_C^r\backslash \M_C)\cap \Os)\,.
	\end{equation*}
	Thus, its expected value is bounded by a constant times
	\begin{equation}
	\E\left[ \mu_{\gamma h}(\M_C^r\backslash \M_C)\right]\,.
	\end{equation}
	Note that this term is independent of $\eta$, and that $\M_C^r\downarrow \M_C$ as $r\downarrow 0$. Thus, we can use dominated convergence to show that this term converges to $0$ uniformly in $\eta$.
	
	The last assertion is proved similarly as the last assertion of Proposition \ref{p.convergence_of_LQG}.
\end{proof}

\section{Size of the spectral sample for multiple quad crossings}
\label{app:spectralsample}

Let $\qs$ be a collection of finitely many quads. For $R>1$ let $R\qs$ denote the same set of quads rescaled by $R$, i.e., 
$$
R\qs := \{RQ\,:\,Q\in\qs \}.
$$
Let $\Tg$ denote the triangular lattice where adjacent vertices have distance 1. For an instance $\omega$ of critical site percolation on $\Tg$ let $f(\omega)=f_{R\qs}(\omega)$ be the indicator function describing whether all the quads of $R\qs$ have an open crossing. Throughout this section we do \emph{not} rescale the triangular lattice; to be consistent with \cite{gps-fourier} we instead rescale the quads by $R$. Several observables throughout the section will depend on $R$, and by simplicity we will often omit the $R$ dependence in notations.

For any set $V\subset\C$ let $\cA_\square(V,\qs)$ denote the event that the vertices in $V\cap\Tg$ are pivotal for $\qs$, i.e., there exists a percolation configuration $\omega'$ such that $\omega|_{\Tg\setminus V}=\omega'|_{\Tg\setminus V}$ and $f(\omega)\neq f(\omega')$. Let $\cQ^o\subset \C$ denote the union of the complementary components $V$ of the quad boundaries which are such that $\P[\cA_\square(RV,R\qs)]>0$ for sufficiently large $R$. 
We assume throughout this and the next section that $\qs^o$ has finitely many connected components and the boundaries of the quads are piecewise smooth.\footnote{Notice that $\qs^o$ is always non-empty: There exist configurations where all the quads are crossed (e.g.\ if all sites of $\Tg$ are open) and there exist configurations where the quads are not all crossed (e.g.\ if all sites of $\Tg$ are closed). By moving from one configuration to the other by changing the sites one by one, we see that there exist configurations where we have a pivotal point.} Let $\cI$ denote the sites of $\Tg$ which are contained in at least one quad in $R\qs$. Throughout the section we let $\alpha_4(R)$ be defined by $\alpha_4(R)=\alpha^\eta_4(8,R)$ for $\eta=1$, where the right side is defined as in Section \ref{sec:gps-dp}. For $r<R$ we write $\alpha_4(r,R)$ instead of $\alpha^1_4(r,R)$ since we work with lattice $\eta=1$ throughout the appendix.

The following is the main result of this appendix. In other words, we prove that the size of the spectral sample $\S$ is of order $R^2\alpha_4(R)$. Note that the theorem was proved in \cite{gps-fourier} for the case where $\qs$ consists of a single quad.
\begin{theorem}
	\eqbn
	\lim_{s\rta\infty} \inf_{R>1} \P\left[ 
	|\scr S_f| \in [s^{-1}R^2\alpha_4(R),s R^2\alpha_4(R))\cup\{0 \}
	\right] = 1,
	\eqen
	where $|\cdot|$ denotes cardinality.
	\label{t. spectral tightness}
\end{theorem}
\begin{proof}
Our proof follows very closely the strategy from \cite{gps-fourier}. The main focus here is to extend the key arguments from that paper to the present multi-quads setting. In particular, the  theorem will follow from Theorem \ref{prop11} and Proposition \ref{prop10} below by exactly the same argument as in the proof of \cite[Theorem 7.4]{gps-fourier}.
\end{proof}

\begin{theorem}
	Let $U\subset\qs^o$ be open, and let $U'\subset\ol{U'}\subset U$. Then, for some constants $\ol r=\ol r(U',U,\qs)>0$ and $q(U',U,\qs)>0$, 
	for any $r\in [\ol r , R\diam(U)]$,
	\begin{align}
	\label{e.loc}
	\Pb{0<|\Spec_{f}\cap RU|\le r^2\,\alpha_4(r),\,\Spec_{f}\cap RU\subset RU'}\qquad\nonumber\\
	\le q(U',U,\qs)\, 
	\frac{R^2\,\alpha_4(R)^2}{r^2\,\alpha_4(r)^2}\,.
	\end{align} 
	\label{prop11}
\end{theorem}
\begin{proof}
	The theorem follows from Propositions \ref{prop14} and \ref{prop20}, and from \cite[Proposition 6.1]{gps-fourier}. See the proof of \cite[Theorem 7.1]{gps-fourier} for a similar argument.
\end{proof}

\begin{prop}
	Given any $\delta>0$ we can find an open set $U\subset\ol{U} \subset\qs^o$, such that $\P[\scr S_{f}\subset RU]>1-\delta$.
	\label{prop10}
\end{prop}
\begin{proof}
	Let $\gamma=\bigcup_{Q\in\cQ}\partial Q\subset\C$ be the union of the quad boundaries. Given $s>0$ let $g:\Omega\to \{-1,1 \}$ be measurable with respect to the $\sigma$-algebra $\mcl F_s$ of quad crossing information at distance $>Rs$ from $R\gamma$, such that
	\eqbn
	g(\omega) = 
	\begin{cases}
		-1 \quad&\text{if}\quad \P[f=-1\,|\,\mcl F_s]>1/2,\\
		1 &\text{otherwise}.
	\end{cases}
	\eqen
	By \cite[Theorem 1.5]{ss-planar-perc}, given any $\ep>0$ and a quad $Q$ it holds for all $s$ and sufficiently small and $R$ sufficiently large that $\P[ \ep<\P[ \omega(Q)\,|\,\mcl F_s ]<1-\ep ]<\ep$, where $\omega(Q)\in\{0,1 \}$ indicates whether $Q$ is crossed. (Note that it is important here to assume that the boundaries of our quads are piecewise smooth.) Therefore, for sufficiently small $s$ and sufficiently large $R$,
	\eqbn
	\P[ \ep<\P[ f=-1\,|\,\mcl F_s ]<1-\ep ]<\ep.
	\eqen
	It follows that for sufficiently small $s$ and sufficiently large $R$,    
	\eqbn
	\begin{split}
		\P[ f\neq g ]\leq&\,\, \P[ f\neq g; \ep< \P[f=-1\,|\,\mcl F_s]<1-\ep] \\
		&+ \P[ f=1; \P[f=-1\,|\,\mcl F_s]>1-\ep ]\\
		&+ \P[ f=-1; \P[f=1\,|\,\mcl F_s]>1-\ep ]
		< 3\ep,
	\end{split}
	\eqen
	which implies that with $\|\cdot\|$ denoting the $L^2$ norm we have $\|f-g\|<10\sqrt{\ep}$.  
	With $\op{tv}$ denoting total variation distance,
	\eqbn
	\op{tv}(\scr S_f,\scr S_g)
	\leq 
	\sum_{S\subset\mcl I} |\wh f(S)^2-\wh g(S)^2| \leq \|f-g\| \|f+g\| < 20\sqrt{\eps},
	\eqen
	where the second inequality follows from \cite[equation (2.7)]{gps-fourier}.
	The spectral sample of $\scr S_g$ has distance at least $Rs$ from $R\gamma$, so we see that the proposition holds with $U'$ instead of $U$ if we let $U'\subset\C$ be the points which have distance at least $s$ from $\gamma$. We have that $\scr S\cap(U'\setminus\qs^o)=\emptyset$, and we obtain the proposition by defining $U=U'\cap\qs^o$.
\end{proof}

\begin{remark}
Note that a possibly more direct proof of this proposition would consist in decomposing the $\alpha$-neighborhood of the boundaries of each quads into $O(\alpha^{-1})$ squares of side length $\alpha$ and then argue through a first moment bound by noticing that 
$$\Pb{\scr S_{f} \text{ intersects the $\alpha$-neighbourhood of the boundaries of quads}}
$$ 
is dominated by the sum over all these squares of the probability that the spectral sample intersects the fixed given square. In the bulk this probability is $O(\alpha^{5/4+o(1)})$ and one can conclude the proof along those lines after dealing with boundary issues. In some sense such boundary issues are already dealt with in the work \cite{ss-planar-perc}, which explains why we have chosen this other approach.
\end{remark}

As in \cite{gps-fourier}, the proof of Theorem \ref{prop11} relies on two key properties: few squares intersect the spectral sample and partial independence in the spectral sample. In the single quad case these properties are established in \cite[Section 4]{gps-fourier} and \cite[Section 5]{gps-fourier}, respectively. The proof given in \cite[Section 4]{gps-fourier} generalizes without difficulty to our multiple quads setting, while the argument in  \cite[Section 5]{gps-fourier} requires slightly more work. Therefore we will simply state our variant of \cite[Proposition 4.2]{gps-fourier} right below, while we provide a more detailed adaption of \cite[Section 5]{gps-fourier} in Section \ref{app:pi}.
\begin{prop}
	Consider a collection $\qs$ of finitely many quads, and let $\Spec$ be the spectral sample of $f_{\qs}$.
	Let $U'\subset U\subset \qs^o$, let $\wh R$ denote the diameter of $U$, let $a\in (0,1)$,
	and suppose that the  distance from $U'$ to the complement of $U$ is at least $a\,\wh R$.
	Let $\mathcal S(r,k)$ be the collection of all sets $S\subseteq\mcl I$ such that
	$\bigl|(S\cap U)_r\bigr|=k$ and $S\cap (U\setminus U')=\emptyset$.
	Then for $g(k):=2^{\coa\log_2^2 (k+2)}$, with $\coa>0$ large enough, and $\gamma_r(\wh R):=(\wh R/r)^2 \alpha_4(r,\wh R)^2$, we have
	$$
	\forall k,r\in\N_+\qquad
	\Pb{\Spec\in\mathcal S(r,k)} \le c_a\,g(k)\,\gamma_r(\wh R)\,,
	$$
	where $c_a$ is a constant that depends only on $a$ and $\qs$.
	\label{prop14}
\end{prop}
\begin{proof}
The proposition is proved by adapting the techniques of \cite[Section 4]{gps-fourier}. In particular, we construct so-called annulus structures for the collection of quads $\cQ$ by defining annulus structures for each component of $\cQ^o$ with diameter at least $a\wh R$.
\end{proof}
We also point out here that another generalization of the techniques needed here have been analyzed in the work \cite{garban-vanneuville}, where the needed extension of \cite[Section 4]{gps-fourier} happened to be more substantial and was thus written with more details.

\subsection{Partial independence in the spectral sample}
\label{app:pi}
For a set $S\subset\Tg$ and $r>0$ define $S_r$ to be the collection of squares in $r\Z^2$ that intersect $S$.
\begin{prop}
	Let $\qs$ be a collection of finitely many quads, and let $U$ be an open set whose closure is contained in $\qs^o$. For $R>0$, let
	$\Spec:= \Spec_{f_{R\qs}}$ be the spectral sample of $f_{R\qs}$, the $\pm 1$ indicator function
	for the crossing event in $R\qs$. Then, there is a constant $\ol r=\ol r(U,\qs)$
	such that
	for any box $B\subset R\,U$ of radius $r \in [\ol r,R\diam(U)]$ and any set $W$ with $W\cap B=\emptyset$, we have
	\begin{align*}
	\Pb{\Spec_{f_{R\qs}} \cap B' \cap \Rs \neq \emptyset \md
		\Spec_{f_{R\qs}}\cap B\neq \emptyset,\, \Spec_{f_{R\qs}}\cap W =\emptyset}\geq
	a(U,\qs)\,,
	\end{align*}
	where $B'$ is concentric with $B$ and has radius $r/3$,
	the random set $\Rs$ contains each element of $\mcl I$ independently with probability $1/(\alpha_4(r) r^2)$, and $a(U,\qs)>0$ is a constant that depends only on $U$ and $\qs$.
	\label{prop20}
\end{prop}
\begin{proof}
	The proof is identical to the proof of \cite[Proposition 5.11]{gps-fourier}. Propositions \ref{prop19} and \ref{p.two point estimate} below give the required first and second moment estimates.
\end{proof}

\begin{remark}\label{r.monot}
As we outline below, it is not too difficult to extend the proof of \cite{gps-fourier} to our present multiple-quad setting. 
Yet, one crucial property of our multiple-quad Boolean function $f=f_{R\qs}$ is that it is a monotone Boolean function. Otherwise the techniques from \cite{gps-fourier}, break down completely. See Remark 5.5 in \cite{gps-fourier}. This is the reason why we only control via Fourier analysis the intersection of several monotone events (crossing events) and deal with the more general ones via an inclusion-exclusion argument. 
\end{remark}

Let $B,W\subset\mcl I$ be disjoint. Let $\Lambda_B=\Lambda_{f,B}$ be the event that $B$ is pivotal for $f$. More precisely,
$\Lambda_B$ is the set of $\omega\in\Omega$ such that there is some $\omega'\in\Omega$
that agrees with $\omega$ on $B^c$ while $f(\omega)\ne f(\omega')$.
Also define $\llwb=\Lambda(B,W):=\Pb{\Lambda_B\md \mathcal{F}_{W^c}}$.

The following lemma is derived as in \cite[Section 5.3, equation (5.10)]{gps-fourier}.
\begin{lemma}
	Consider the setting of Proposition \ref{prop20}.
	Let $\wh f$ be a monotone function of a percolation configuration on $\cI$ such that $\wh f$ takes values in $\{ -1,1 \}$, and let $\wh{\S}$ denote the associated spectral sample. Let $\omega,\omega'$ be instances of critical percolation on $\Tg$ which are the same on $W^c$ and independent on $W$. Then the following two inequalities are equivalent for any constant $c_1>0$
	\eqbn
	\Pb{ \omega',\omega''\in A_\square(x,\qs)}
	\ge c_1\,
	\Pb{\omega',\omega''\in A_4(x,B)} \, \Pb{ \omega',\omega''\in A_\square(B,\qs)},
	\eqen
	\eqbn
	\Pb{x\in\wh\Spec,\ \wh\Spec \cap W=\emptyset} \geq c_1\, \Eb{\llwb^2}\, \alpha_4(r).
	\eqen
	\label{prop12}
\end{lemma}

Recall Definition \ref{def:quad}. We call the boundary arcs $\partial_1 Q$ and $\partial_3 Q$ (resp.\ $\partial_2 Q$ and $\partial_4 Q$) the {\bf open boundary arcs} (resp.\ {\bf closed boundary arcs}) of $Q$. For an instance of site percolation on $\Tg$ the quad $Q$ is crossed (resp.\ not crossed) if and only if there is a path of open (resp.\ closed) sites connecting the two open (resp.\ closed) boundary arcs.
\begin{figure}
	\begin{center}
		\includegraphics[scale=1.5]{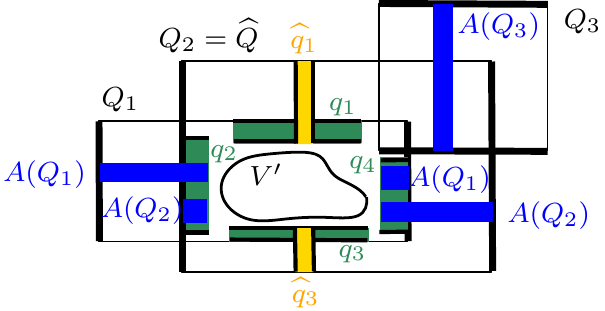}
	\end{center}
	\caption{Illustration of the quads defined in Lemma \ref{prop18}.  The bold black boundary arcs of the quads $Q_1,Q_2,Q_3,q_1,\dots,q_4,\wh q_1,\wh q_3$ indicate the boundary arcs which are connected on the event that the quads have an open crossing. For the quads $Q_1$ and $Q_2$ we are in case (b) of Lemma \ref{prop18}(ii), while for the quad $Q_3$ we are in case (a). Therefore $A(Q_1)$ and $Q(Q_2)$ consist of two quads each, while $A(Q_3)$ consists of a single quad. 
	}
\end{figure}

\begin{lemma}\label{l.constr}
	Let $V$ be a connected component of $\qs^o$. For any $V'\subset V$ such that $\overline{V'}\subset V$, we can find quads $q_1,\dots,q_4,\wh q_1,\wh q_3$ and a collection of quads $A(Q)$ for each $Q\in\qs$ such the following hold.
	\begin{itemize}
		\item[(i)] The quads  $q_1,\dots,q_4$ are bounded away from $V'$ and each other, and they are contained in $V$. One of the open boundary arcs of $q_1$ (resp.\ $q_3$) is equal to a boundary arc of $V$, while the other boundary arcs of $q_1$ (resp.\ $q_3$) are in the interior of $V$. The same property holds for $q_2$ and $q_4$, but with closed instead of open. The quads $q_1,\dots,q_4$ are in counterclockwise order around $\partial V$.
		\item[(ii)] For each quad $Q\in\qs$ one of the following properties (a) or (b) holds.
		\begin{itemize}
			\item[(a)] $A(Q)$ consists of a single quad $q$ contained in $Q$, which is such that the open boundary arcs of $q$ are contained in each of the open boundary arcs of $Q$. 
			\item[(b)] $A(Q)$ consists of two quads $q',q''$ contained in $Q$, which are such that one open boundary arc of $q'$ is contained in an open boundary arc of $Q$ and the other open boundary arc of $q'$ is contained in the closed boundary arc of either $q_2$ or $q_4$ which does not intersect $\partial V$. The same property holds for $q''$, except that $q''$ intersects the other open boundary arc of $Q$.
		\end{itemize}
		\item[(iii)] There is a quad $\wh Q\in\cQ$ such that $\wh q_1$ and $\wh q_3$ are contained in $\wh Q$. Furthermore, one closed boundary arc of $\wh q_1$ is contained in a closed boundary arc of $\wh Q$ and the other closed boundary arc of $\wh q_1$ is contained in the open boundary arc of $q_1$ which does not intersect $\partial V$. The same property holds for $\wh q_3$, except that the closed boundary arcs intersect the other closed boundary arc of $\wh Q$ and an open boundary arc of $q_4$, respectively.
		\item[(iv)] $q_1\cup q_3\cup \wh q_1\cup \wh q_3$ and $q_2\cup q_4\cup (\cup_{Q\in\qs}A(Q))$ are disjoint.
	\end{itemize}
	\label{prop18}
\end{lemma}
Observe that $V$ is pivotal for $\qs$ if 
all the quads in $\{q_2,q_4\}\cup\big(\bigcup_{Q\in\cQ} A(Q)\big)$ have open crossings, and none of the quads in $\{q_1,q_3,\wh q_1,\wh q_3\}$ have open crossings. (N.B. obviously this is not an iff). 
\begin{proof}
	Choose $R$ large and consider a coloring such that there is a pivotal point $x\in V$ for the event that all the quads in $R\qs$ are crossed. Let $x$ be closed. We may assume $R^{-1}x$ is bounded away from $\partial V$ and choose $V'$ such that $x\in V'$ and $\ol{V'}\subset V$. For quads for which there is an open crossing define $A(Q)$ and $q$ as in (ii)(a) by using the open crossing, e.g.\ consider a path of open hexagons in the dual lattice connecting the two open sides of $Q$ and let $q$ be contained in these hexagons. 
	
	For the remaining quads $Q\in\qs$ the vertex $x$ is pivotal. Define $\wh Q$ to be one of these remaining quads (in Figure \ref{fig:quads} we chose $\wh Q=Q_2$). If $Q$ is a quad which is not crossed (including, among others, the particular quad $\wh Q$), define $A(Q)$ as in (ii)(b) by using two open arms from $V$ to the open boundary arcs of $Q$. At this point we have not yet defined $q_2$ and $q_4$, so instead of the requirement involving $q_2,q_4$ in (ii)(b) we assume that one of the open sides of each quad in $A(Q)$ is contained in $V$. We may assume that the quads in $A(Q)$ do not enter and exit $V$ multiple times in the sense that for each $q'\in A(Q)$ the set $q'\setminus V$ has one connected component (viewing $q'$ as a subset of $\C$). If $q'$ does not satisfy this property then it will hold for some quad $\wt q$ contained in $q'$ (such that $\wt q$ still satisfies the requirements as specified in (ii)(b)), and we replace $q'$ by $\wt q$.
	
	Define  $\wh q_1,\wh q_3$ satisfying (iii) by using two closed arms from $V$ to the closed boundary arcs of $\wh Q$. Again we assume that $\wh q_1\setminus V$ and $\wh q_2\setminus V$ have one connected component.
	
	Note that (iv) is satisfied if we let the quads in $A(Q)$ for all $Q\in\qs$ along with $\wh q_1,\wh q_3$ be contained in the interior of the hexagons which define the crossings. Finally, we can find quads $q_1,\dots,q_4$ satisfying (i), (ii)(b), and (iii) (after doing local deformations of the parts of the quads in $A(Q)\cup\{\wh q_1,\wh q_3 \}$ intersecting $V$) since $\partial V$ can be divided into four arcs such that with these arcs in counterclockwise order, the first (resp.\ third) arc contains $\wh q_1\cap\partial V$ (resp.\ $\wh q_3\cap\partial V$), and the union of the remaining two arcs contain $q\cap\partial V$ for each $q$ in some set $A(Q)$.
\end{proof}

The following is our first moment estimate. It is an analogue of \cite[Proposition 5.2]{gps-fourier} for the case of multiple quads.
\begin{prop}
	Consider the setup of Proposition \ref{prop20}. There is a constant $c_1>0$ (depending on $U$ and $\cQ$) such that for any $x\in B'\cap\mcl I$,
	\begin{equation}\label{e.1stm}
	\Pb{x\in\Spec,\ \Spec \cap W=\emptyset} \geq c_1\, \Eb{\llwb^2}\, \alpha_4(r)\,.
	\end{equation}
	\label{prop19}
\end{prop}
\begin{proof} 
	In the proof below, we will rely on the notations introduced in \cite[Section 5]{gps-fourier}. It is sufficient to prove the first inequality of Lemma \ref{prop12} (which is a quasi-multiplicativity type of estimate). We will use for this the construction provided by Lemma \ref{l.constr}, where $V$ is the component of $\qs^o$ containing $R^{-1}x$. Let $V''$ be the connected component of $U$ which contains the point $R^{-1}x$, and set $d=R\op{dist}(V'',V^c)$. Then define $V'=\{y\in V\,:\, R\op{dist}(y,V^c)>d/3 \}$, so that $V''\subset V'\subset V$. 
	
	Let $\wh B\subset RV$ (resp.\ $\wh B'\subset RV$) be the square concentric with $B$ of side length $d/3+r$ (resp.\ $d/6+r$). Note that $\wh B$ has distance at least $d/3$ from $(RV')^c$. Let $L_0,\dots,L_7$ be defined as in \cite[Section 5]{gps-fourier} with the annulus $\wh B\setminus \wh B'$. Let $E$ be the event $\omega',\omega''\in\scr A_4(x,\wh B)$, with the additional requirement that the two open (resp.\ closed) arms cross the annulus $\wh B\setminus \wh B'$ inside $L_0,L_4$ (resp.\ $L_2,L_6$), and that there are open (resp.\ closed) paths that separate $\partial\wh B\cap L_j$ from $\partial\wh B'\cap L_j$ inside $L_j$ for $j=0,4$ (resp.\ $j=2,6$).
	Let $E'$ be the event that the quads $\bigcup_{Q\in \cQ}A(Q),q_2,q_4$ rescaled by $R$ have open crossings, and that $q_1,q_3,\wh q_1,\wh q_3$ rescaled by $R$ have closed crossings. Let $E''$ be the event that there is an open crossing from $\partial \wh B'$ to $R(\partial q_2\cap \partial V)$ inside $L_0\cup ((RV)\setminus \wh B)$, that there is a similar crossing with $L_4$ and $q_4$, and that there are similar closed crossings. We have
	\begin{align}
	\nonumber	&\P[ \omega',\omega''\in\scr A_4(x,B) ] 
	\P[ \omega',\omega''\in\scr A_\square(B,R\cQ) ]\\
	&\hspace{0.15\textwidth}\leq \P[ \omega',\omega''\in\scr A_4(x,B) ] 
	\P[ \omega',\omega''\in\scr A_4(B,\wh B) ]\tag{a}\label{e. a}\\
	&\hspace{0.15\textwidth}\preceq {\P}[ \omega',\omega''\in\scr A_4(x,\wh B) ]\tag{b}\label{e. b}\\
	&\hspace{0.15\textwidth}\preceq {\P}[ \omega',\omega''\in\scr A_4(x,\wh B)\cap E\cap E'\cap E'' ]\label{e. c}\tag{c}\\
	&\hspace{0.15\textwidth}\leq {\P}[ \omega',\omega''\in\scr A_\square(x,R\cQ) ]\label{e. d}\tag{d}.\\
	\end{align}
	Here \eqref{e. a} and \eqref{e. d} are immediate by inclusion of events, \eqref{e. b} is \cite[Proposition 5.6]{gps-fourier}, and 
	\eqref{e. c} follows by using the Russo-Seymour-Welsh theorem, the FKG inequality, and compactness.
\end{proof}

\begin{figure}
	\begin{center}
		\includegraphics[scale=1.5]{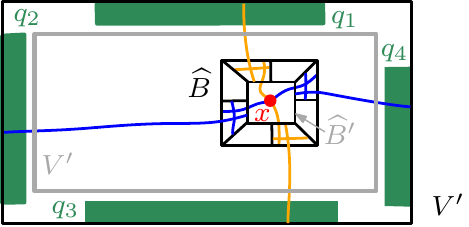}
	\end{center}
	\label{fig:quads}
	\caption{Illustration of the events $E'$ and $E''$ in the proof of Proposition \ref{prop19}. Open arms are blue and closed arms are orange.
	}
\end{figure}

The following is the (easier) second moment estimate. It is an analogue of \cite[Proposition 5.3]{gps-fourier} for the case of multiple quads.
\begin{prop}
	Let $\Spec$ be the spectral sample of $f=f_{R\qs}$, where $\qs$ is a collection of finitely many quads.
	Let $z$ be a point in one of the quads and let $r>0$. Set $B:= B(z,r)$ and $B':=B(z,r/3)$.
	Suppose that $B(z,r/2)\subset R\qs^o$ and that $B$ and $W$ are disjoint.
	Then for every $x,y\in B'\cap \mcl I$ we have
	\eqbn\label{e.2ndm}
	\Pb{x,y\in\Spec,\ \Spec \cap W=\emptyset} \leq c_2\, \Eb{\llwb^2}\, \alpha_4(|x-y|)\,\alpha_4(r)\,,
	\eqen
	where $c_2<\infty$ is an absolute constant.
	\label{p.two point estimate}
\end{prop}
\begin{proof}
	The proof is identical to the proof in \cite{gps-fourier}. Note that \cite[Lemmas 2.1 and 2.2]{gps-fourier}, which are used in the proof, hold for the spectral sample of general real-valued functions $f$ of the percolation configuration. For an arbitrary set $A\subset\mcl I$ we let $\Lambda_A$ be the event that $A$ is pivotal for our quad crossing event. One key geometric argument in proof which still holds in our setting is that if we condition on $\omega$ restricted to the complement of $W\cup\{x,y \}$ and if flipping $\omega_x$ affects $f(\omega)$, then we must have a four arm event from $x$ to distance $|x-y|/4$, and four arms in an annulus with outer boundary $\partial B$ and inner boundary defined by a box centered at $(x+y)/2$ with radius $2|x-y|$.
\end{proof}

\bibliographystyle{hmralphaabbrv} 
\bibliography{../biblong,../bibshort,../mybib0,../bib-emb}

\end{document}